\documentclass[reqno,english]{amsart}
\usepackage[english]{babel}
\usepackage{amsfonts,amsmath,amssymb,amsthm,latexsym,amscd}
\usepackage{mathrsfs,color,array,graphicx}
\usepackage{booktabs,longtable,enumitem}
\IfFileExists{microtype.sty}{\usepackage{microtype}}{}
\usepackage[hidelinks,bookmarksnumbered]{hyperref}

\let\oldbibliography\thebibliography
\renewcommand{\thebibliography}[1]{%
  \oldbibliography{#1}%
  \setlength{\itemsep}{0pt}%
}

\newtheorem{theorem}{Theorem}[section]
\newtheorem{lemma}[theorem]{Lemma}
\newtheorem{proposition}[theorem]{Proposition}
\newtheorem{corollary}[theorem]{Corollary}

\theoremstyle{definition}

\newtheorem{open problem}[theorem]{Open Problem}
\theoremstyle{remark}
\newtheorem{remark}[theorem]{Remark}
\newcommand{\bremark}{\begin{remark}}
\newcommand{\eremark}{\end{remark}}
\theoremstyle{plain}

\newcommand{\R}{{\mathbb R}}

\newcommand{\BE}{\begin{equation}}
\newcommand{\BEN}{\begin{equation*}}
\newcommand{\EE}{\end{equation}}
\newcommand{\EEN}{\end{equation*}}
\newcommand{\BL}{\begin{lemma}}
\newcommand{\EL}{\end{lemma}}
\newcommand{\BT}{\begin{theorem}}
\newcommand{\ET}{\end{theorem}}
\newcommand{\BP}{\begin{proposition}}
\newcommand{\EP}{\end{proposition}}
\newcommand{\BC}{\begin{corollary}}
\newcommand{\EC}{\end{corollary}}

\DeclareMathOperator*{\argmin}{arg\,min}

\numberwithin{equation}{section}
\allowdisplaybreaks[2]

\newcommand{\Z}{\mathbb Z}

\newcommand{\FCC}{L_{\mathrm F}}
\newcommand{\B}{L_{\mathrm B}}
\newcommand{\BCC}{\B}
\newcommand{\YB}{Y_{\mathrm B}}
\newcommand{\YF}{Y_{\mathrm F}}
\newcommand{\HH}{\mathcal H}
\newcommand{\PP}{\mathcal P_3^1}
\newcommand{\calP}{\PP}
\newcommand{\tr}{\operatorname{tr}}
\newcommand{\diag}{\operatorname{diag}}
\newcommand{\covol}{\operatorname{covol}}

\newcommand{\norm}[1]{\lVert#1\rVert_{\mathrm F}}
\newcommand{\opnorm}[1]{\lVert#1\rVert_{\mathrm{op}}}
\newcommand{\op}{\mathrm{op}}
\newcommand{\frob}{\mathrm F}
\newcommand{\lam}{\lambda_{\mathrm F}}

\newcommand{\Th}{\Theta}
\newcommand{\Id}{I}
\newcommand{\calK}{\mathcal K}
\newcommand{\epsref}{\varepsilon_{\rm ref}}
\newcommand{\gFE}{g_{\mathrm F,\mathrm E}}
\newcommand{\gFT}{g_{\mathrm F,\Theta}}
\newcommand{\gBT}{g_{\mathrm B,\Theta}}
\newcommand{\Symzero}{\operatorname{Sym}_0(3)}
\newcommand{\Lone}{\mathscr L_1}
\newcommand{\gBE}{g_{\mathrm B,\mathrm E}}
\newcommand{\delog}{\vartheta}
\hypersetup{
  pdftitle={On the minima of theta and Epstein zeta functions in dimension three},
  pdfauthor={Senping Luo and Juncheng Wei},
  pdfsubject={Theta and Epstein minimization: analytic estimates and proved computational lemmas}
}

\begin{document}

\title[Theta and Epstein energies in dimension three]{On Sarnak--Str\"{o}mbergsson conjecture}
\author[S. Luo]{Senping Luo}
\address[S.~Luo]{School of Mathematics and Statistics, Jiangxi Normal University, Nanchang, 330022, China; and Jiangxi Applied Mathematical Research Center.}
\email[S.~Luo]{luosp1989@163.com}
\author[J. Wei]{Juncheng Wei}
\address[J.~Wei]{Department of Mathematics, Chinese University of Hong
Kong, Shatin, NT, Hong Kong.}
\email[J.~Wei]{wei@math.cuhk.edu.hk}

\date{}
\begin{abstract}
Let $\Th(\alpha,L)=\sum_{v\in L}e^{-\pi\alpha|v|^2}$ for
       $\alpha>0$ and
 $E(L,s)=\sum_{v\in L\setminus\{0\}}|v|^{-2s}$ for
       $s>3/2$ be the theta and Epstein zeta functions associated to the lattice $L$, respectively. We are particularly interested in physically relevant dimension three.
Fix the covolume of the lattice $L$ to $1$. Up to an orthogonal transformation, we prove that
\begin{equation}\nonumber
\argmin_{|L|=1}\Th(\alpha,L)=
\begin{cases}
 \boldsymbol{\mathrm{FCC}}\;\;\mathrm{lattice},\;\;\;\;\;\;\;\;\; \;\;\;\;\;\;&\text{if}\;\; \alpha>1,\\
 \boldsymbol{\mathrm{BCC}}\;\;\mathrm{lattice}, \;\;\;\;\;\;\;\;\;\;\;\;\;\;\;&\text{if}\;\; \alpha<1,\\
 \boldsymbol{\mathrm{FCC}}\;\;\mathrm{or}\;\; \boldsymbol{\mathrm{BCC}}\;\;\mathrm{lattice},\;\;&\text{if}\;\;\alpha=1,
\end{cases}
\end{equation}
and
\begin{equation}\nonumber
\argmin_{|L|=1}E(L,s)=
 \boldsymbol{\mathrm{FCC}}\;\;\mathrm{lattice},\;\;\;\;\;\;\;\;\; \;\;\;\;\;\;\text{if}\;\; s>3/2.
\end{equation}
Therefore, we prove the
Sarnak--Str\"{o}mbergsson conjecture \cite[Inequalities (43)--(44), Section 5]{SS}.

\end{abstract}

\maketitle
\enlargethispage{12pt}
\setcounter{equation}{0}

\section{Introduction and main results}\label{sec:introduction}
In this paper, we study the minimization of theta and Epstein zeta
functions among three-dimensional lattices of fixed covolume. Our aim is
to determine the minimizing lattices for all positive Gaussian scales and
all convergent Epstein exponents, and to describe the different roles of
FCC and BCC in the two problems.

To state the main results, we give some notations. Vectors are columns, and $|v|=|v|_2=(v_1^2+v_2^2+v_3^2)^{1/2}$
denotes Euclidean length on $\R^3$. For a matrix $g$, the symbols
$g^T$ and $g^{-T}=(g^{-1})^T$ denote its transpose and inverse transpose.
Let $L=g\Z^3\subset\R^3$, where $g\in GL(3,\R)$, and set
$\covol(L)=|\det g|$. Unless otherwise stated, throughout this paper, we assume that $\covol(L)=1$.
Lattices are identified up to orthogonal equivalence: $L$ and $L'$ are
equivalent if $L'=UL$ for a real orthogonal matrix $U$, that is,
$U^TU=I$. Here and below $I$ is the identity matrix of the indicated
dimension, usually three. We write $L^*=g^{-T}\Z^3$ for the dual lattice.

Define the theta and Epstein functions as
\begin{align}
 \Th(\alpha,L)&=\sum_{v\in L}e^{-\pi\alpha|v|^2},\;\;\;\;\;\;\;\;\;\;
       \alpha>0,\label{T:eq:theta}\\
 E(L,s)&=\sum_{v\in L\setminus\{0\}}|v|^{-2s},
       \;\;\;\;\;\;\;\;s>3/2.\label{E:eq:Epstein}
\end{align}
They are the $f$-potential lattice energy of the formally discrete forms
 \begin{equation}\aligned\label{EFL}
E_f(L):=\sum_{\mathbb{P}\in L \backslash\{0\}} f(|\mathbb{P}|^2),\; |\cdot|\;\hbox{is the Euclidean norm on}\;\mathbb{R}^n,
\;\;L\;\;\hbox{is a lattice}
\endaligned\end{equation}
for Gaussian and Riesz potentials, respectively.
Lattice minimization problems, rooted in crystallization problems, analytic number theory, and solid state physics (see the reviews by Blanc-Lewin \cite{Bla2015}, Lewin \cite{Lewin2022,Lewin2025}, Radin \cite{Radin1987,Radin1991}), have attracted considerable attention in recent years;
see B\'etermin \cite{Bet2016,Betermin2021AHP,Betermin2021JPA,Betermin2021SIAM}, B\'{e}termin-Petrache \cite{Betermin2017,Betermin2019b}, B\'{e}termin-Friedrich-Stefanelli \cite{Betermin2021LMP}, B\'etermin-\v{S}amaj-Trav\v{e}nec \cite{Betermin2023SAM,Betermin2024AMP}.
In particular, in dimension two, we have extensively studied the two-component Coulomb energy (with Ren \cite{Luo2019}), Mueller-Ho energy (\cite{Luo2022}), Gaussian-related energy (\cite{LW2022,Luo2023,LW2025b,LW2026}), Lennard-Jones energy (\cite{LW2025}); still in dimension two, we also refer to the hexagonal lattice in diblock copolymers (Chen-Oshita \cite{Che2007}), the Abrikosov triangular vortex in superconductivity (Sandier-Serfaty \cite{Serfaty2010,Serfaty2012}, Serfaty \cite{Serfaty2018}) and in 2D crystallization (Theil \cite{Theil2006}).
In higher dimensions, Cohn and de Courcy-Ireland (\cite{Cohn2018}) obtained asymptotically sharp results. Moreover, Cohn--Kumar--Miller--Radchenko--Viazovska (\cite{Cohn2022}) established global universal optimality of $E_8$ and Leech lattices in dimensions $8$ and $24$, respectively. Building on these results, Petrache-Serfaty \cite{Petr2020} obtained the corresponding optimality of $E_8$ and Leech lattices in dimensions $8$ and $24$ for Coulomb and Riesz renormalized energies, as well as jellium and periodic jellium energies.

The theta and Epstein conventions agree with $\Th(L,i\alpha)$ and
$E(L,s)$ in Sarnak--Str\"{o}mbergsson \cite[(39) and (1)]{SS}.
Deleting the zero-vector contribution does not change the theta minimizers.
The unit covolume constraint is for convenience: dilation by $c>0$ replaces $\alpha$
by $c^2\alpha$ and multiplies the Epstein energy by $c^{-2s}$.

If $Y=g^Tg$, then
\begin{equation}
\label{p31}
 \PP=\{Y\in\R^{3\times3}:Y=Y^T>0,\ \det Y=1\},
 \qquad Y\longmapsto\gamma^TY\gamma\quad(\gamma\in GL(3,\Z)).
\end{equation}
Following \cite[Section~1]{SS}, write $\Lone=\PP/GL(3,\Z)$ for the
space of covolume-one lattice classes modulo orthogonal equivalence,
and $[L]$ for the class of $L$. We use $E(Y,s)$ and $\Th(\alpha,Y)$ for the corresponding
lattice energies. Here $Y>0$ means positive definite, and
$Y[m]:=m^TYm$ for $m\in\Z^3$; in particular,
$Y^{-1}[m]=m^TY^{-1}m$. We write $\tr A=\sum_i A_{ii}$ for the trace.
Logarithmic volume-preserving deformations are parametrized by
\[
 \Symzero=\{X\in\R^{3\times3}:X=X^T,\ \tr X=0\}.
\]
The Frobenius and operator norms are denoted by $\norm{\cdot}$ and
$\opnorm{\cdot}$, respectively.
For $A=(a_{ij})\in\R^{3\times3}$, the Frobenius norm is
\[
 \norm{A}
 :=\left(\sum_{i,j=1}^3a_{ij}^2\right)^{1/2}
 =\bigl(\tr(A^TA)\bigr)^{1/2}.
\]
The operator norm is induced by the Euclidean vector norm:
\[
 \opnorm{A}
 :=\sup_{v\in\R^3\setminus\{0\}}\frac{|Av|_2}{|v|_2}
 =\max_{|v|_2=1}|Av|_2
 =\sqrt{\lambda_{\max}(A^TA)},
\]
where $\lambda_{\max}(A^TA)$ is the largest eigenvalue of the positive
semidefinite matrix $A^TA$.

Set $\rho=2^{1/3}$ and $\lam=\rho$. The covolume-one FCC and BCC lattices are
\begin{align}
 \FCC&=\rho^{-1}\{m\in\Z^3:m_1+m_2+m_3\equiv0\pmod2\},\label{E:eq:FCC}\\
 \B=\FCC^*&=\rho^{-2}\{m\in\Z^3:m_1\equiv m_2\equiv m_3\pmod2\}.
 \label{E:eq:BCC}
\end{align}
Their shortest squared lengths are $\rho$ and $3\rho^2/4$.
For the Epstein deformation formulas we use
\begin{equation}\label{eq:epstein-frame}
 \gFE=\rho^{-1}\begin{pmatrix}0&1&1\\1&0&1\\1&1&0\end{pmatrix},
 \qquad
 \YF=\rho\begin{pmatrix}1&1/2&1/2\\1/2&1&1/2\\1/2&1/2&1\end{pmatrix}.
\end{equation}
Both determinants are one. The physical frame used in the theta branch
is given in Section~\ref{sec:common-geometry}; it has the same Gram matrix.

\subsection{Main results}
We first state the result for the theta function. Its minimizer changes
with the Gaussian scale, as described in the following theorem.
\begin{theorem}[Minima of 3D Theta functions]\label{T:thm:main}
Let $L\subset\R^3$ be a lattice of covolume one, and let
$\alpha>0$. The following statements hold.
\begin{itemize}
  \item [(1)] If $\alpha>1$, then $\Th(\alpha,L)\ge\Th(\alpha,\FCC)$, equality holds if and only if $L$ is orthogonally equivalent
to $\FCC$.
  \item [(2)] If $\alpha<1$, then $\Th(\alpha,L)\ge\Th(\alpha,\BCC)$, equality holds if and only if $L$ is orthogonally equivalent
to $\BCC$.
  \item [(3)] If $\alpha=1$, then $\Th(\alpha,L)\ge\Th(\alpha,\FCC)=\Th(\alpha,\BCC)$, equality holds if and only if $L$ is orthogonally
equivalent to $\FCC$ or $\BCC$.
\end{itemize}
\end{theorem}
Note that the cases $\alpha>1$ and $0<\alpha<1$ are related by duality.
Indeed, the Poisson identity \cite[(41)]{SS} gives
\begin{equation}\label{T:eq:poisson}
 \Th(\alpha,L)=\alpha^{-3/2}\Th(\alpha^{-1},L^*),
\end{equation}
and hence
\begin{equation}\label{T:eq:endpointtie}
 \Th(1,\FCC)=\Th(1,\BCC).
\end{equation}

In contrast to the theta function, the Epstein zeta function has the same
minimizer throughout its convergence range. More precisely, we have the
following result.
\begin{theorem}[Minima of 3D Epstein zeta functions]\label{E:thm:main}
For every lattice $L\subset\R^3$ of covolume one and every
real $s>3/2$,
\begin{equation}\label{E:eq:main-all}
 E(L,s)\ge E(\FCC,s),
\end{equation}
with equality if and only if $L$ is orthogonally equivalent to $\FCC$.

\end{theorem}
The defining Epstein series converges for $s>3/2$. At the critical
exponent, the comparison concerns the finite part
\[
 C(L):=\lim_{s\to3/2}
 \left(E(L,s)-\frac{2\pi}{s-3/2}\right),\qquad C(Y):=C(L)
 \quad\text{when }Y\text{ represents }L,
\]
rather than the divergent series. Section~\ref{E:sec:ewald} proves the
existence of this limit and derives its relation to spectral height,
whose FCC minimum was proved in \cite[Theorem~2]{SS}.

Theorems \ref{T:thm:main} and \ref{E:thm:main} give a confirmative answer to  the
Sarnak--Str\"{o}mbergsson conjecture \cite[Inequalities (43)--(44), Section 5]{SS}.

Theorem \ref{T:thm:main} stands in sharp contrast to the classical two-dimensional case (Montgomery \cite{Mon1988}), where the hexagonal lattice is optimal for all $\alpha>0$.
 The first rigorous results on three-dimensional theta functions from a different viewpoint were established by B\'etermin and Petrache \cite{Betermin2017} using dimension reduction techniques; in contrast, our approach relies on the physically and chemically motivated Bain path. Furthermore, while the bcc and fcc lattices are known to be critical points for a wide class of lattice energies \cite{Betermin2019}, B\'etermin \cite{Betermin2019} further proved that the simple cubic lattice is a saddle point of the three-dimensional Gaussian energy. The three-dimensional Gaussian energy has attracted significant attention; see, e.g., the series of works by B\'etermin and collaborators \cite{Betermin2019b, Betermin2021M3AS, Betermin2021AHP, Betermin2021JPA, Betermin2021LMP}. Regarding the local minimality of the fcc lattice for 3D Epstein zeta functions, see Ennola \cite{Ennola}.
Notable progress on the three-dimensional Lennard-Jones lattice energy has also been made by B\'etermin, \v{S}amaj, and Trav\v{e}nec \cite{Betermin2023SAM}.

As noted in a recent review by Ca\~{n}izo and Ramos-Lora \cite{Canizo2024}, ``From the point of view of rigorous results, the situation in three dimensions is almost completely open.''
Theorems \ref{T:thm:main} and \ref{E:thm:main} provide a comprehensive result in dimension three on optimization problems arising in number theory, condensed matter physics, and phase transitions in real-life situations; see the review by physicists Grimvall--K\"{o}pe--Ozoli\c{n}\v{s}--Persson \cite{Gri2012}, Craievich-Weinert-Sanchez-Watson \cite{Crai1994}, Radin \cite{Radin1987, Radin1991}, and the insightful mathematical reviews by Blanc--Lewin \cite{Bla2015} and Lewin \cite{Lewin2022}, as well as the references therein.
Sarnak--Str\"{o}mbergsson conjecture \cite{SS} (Section~5), which was
highlighted by Blanc and Lewin \cite{Bla2015} as ``a very
important conjecture: its proof would be an important advance both in analytic number
theory and in solid-state physics.''; see also B\'{e}termin--Petrache \cite{Betermin2017}, the reviews by Blanc--Lewin \cite{Bla2015} and Lewin \cite{Lewin2022,Lewin2025}.

The fcc lattice is widely observed in nature. It is renowned as the optimal sphere packing in three dimensions---a result known as Kepler's conjecture, proved for lattices by Gauss \cite{Gauss1840} and in the general case by Hales \cite{Hales} using a computer-assisted proof.
It also arises as the ground state for three-dimensional crystallization of atomistic configurations (Flatley and Theil \cite{Theil2015}).
More recently, the bcc lattice has been identified as the energy minimizer for certain long-range interaction potentials (Ren and Wei \cite{Ren2023}).
The mcc lattice can be viewed as the geometric mean of the fcc and bcc lattices, corresponding to the period matrix of a hyperelliptic Riemann surface \cite{Conwaybook}.

\subsection{Proof strategy and context}\label{sec:proof-strategy}
The proofs consist of a local comparison and a comparison away from the
reference lattices, extending the method of Sarnak--Str\"{o}mbergsson \cite{SS}
from $s=3/2$ to $s>3/2$ up to the singular part. For $A\in\PP$, let $A^{1/2}$ be its positive definite
symmetric square root and write $A^{1/2}L=\{A^{1/2}v:v\in L\}$.
For the theta function, we prove in Theorem~\ref{T:thm:local} that
\[
 \Th(\alpha,A^{1/2}\FCC)-\Th(\alpha,\FCC)
 \ge\frac{u}{50}e^{-u}\|A-I\|_{\frob}^2,
 \qquad u=\pi\rho\alpha,
\]
for every $A\in\PP$ with $\|A-I\|_{\op}\le1/100$ and $u\ge39/10$.
For $\alpha\ge2$, the radius increases to $1/25$.
Here the second shell is needed to control the diagonal Hessian near the
self-dual scale. For large scales, we use an exponential-moment inequality
to keep the neighbourhood fixed. By duality, we also obtain the local
comparison at BCC near $\alpha=1$.

For the Epstein zeta function, we introduce the positive kernel
of \cite[(24)]{SS}:
\begin{equation}
\label{E:eq:G}
 G(a,x):=\int_1^\infty e^{-xt}t^{a-1}\,dt,
 \qquad a\in\R,\quad x>0.
\end{equation}
For $s\in\R$ and $Y\in\PP$, define the positive Ewald sum
\[
 \HH_s(Y):=\sum_{m\in\Z^3\setminus\{0\}}
 \bigl[G(s,\pi Y[m])+G(3/2-s,\pi Y^{-1}[m])\bigr].
\]
For a lattice $L$ with Gram matrix $Y$, we also write
$\HH_s(L)=\HH_s(Y)$; an integral change of basis only reindexes the sums.
This sum converges for every real $s$. In Section~\ref{E:sec:ewald}, we
relate its differences to those of the Epstein energy and, at $s=3/2$,
to the critical finite part. Cubic symmetry makes the first variation
vanish. We then estimate both Hessian coefficients and the third-order
remainder, obtaining Lemma~\ref{E:lem:third-criterion}. Keeping the signs
of the cubic and quartic terms gives larger neighbourhoods at low
exponents. For $8\le s\le20$ and $s\ge20$, the uniform local estimates
are given in Theorems~\ref{E:thm:local8} and~\ref{E:thm:local20}.

It remains to compare the energies outside these neighbourhoods. A
short-vector estimate removes the cusp, leaving a compact set of lattice
shapes. We bound the competitor from below by a finite positive sum and
the reference from above by adding its full tail. For theta energy,
convexity and a positive scale derivative at $\alpha=5$ give the comparison
for every $\alpha\ge5$. For Epstein energy, we use upper-curvature
interpolation at low exponents and convex power sums at higher exponents.
These estimates hold on complete parameter intervals, not only at the
values used to evaluate their endpoints.

The local--global strategy follows the height proof of Sarnak and
Str\"ombergsson \cite[Theorem~2 and Section~4]{SS}: cusp exclusion and
symmetry-adapted local estimates leave a compact comparison. They state the
theta conjecture in \cite[(43)]{SS}; their Epstein conjecture
\cite[(44)]{SS} also includes exponents below $3/2$, which are not covered
by Theorem~\ref{E:thm:main}.
In this convention, the spectral height is $h(L):=Z_L'(0)$, where
$Z_L(w)=(4\pi^2)^{-w}E(L^*,w)$ is the spectral zeta function of the
flat torus $\R^3/L$, continued meromorphically to a neighbourhood of
$w=0$. With $\gamma$ denoting Euler's constant, \cite[(3), (24)--(26)]{SS}
gives
\begin{align}\label{T:eq:heightintegral}
 h(L)&=\log(4\pi)-\gamma-\frac23\notag
+\int_1^\infty(\Th(t,L)-1)t^{1/2}\,dt
       +\int_1^\infty(\Th(t,L^*)-1)\frac{dt}{t},
\end{align}
The height minimum does not imply a pointwise theta comparison: the
formula integrates over scales and also involves the dual lattice.
Likewise, the FCC theta inequality cannot simply be integrated to prove
Theorem~\ref{E:thm:main}, since the theta minimizer changes below the
self-dual scale.

The modified Grenier domain and cubic invariant bases are taken from
\cite[Sections~3--4]{SS}; the kernel-specific estimates are developed here.
Ennola \cite{Ennola} treats local Epstein optimality. These results do not replace
the exterior comparisons.

\subsection{Outline of the paper}\label{sec:paper-outline}
The proof is divided into the following steps. Sections~\ref{sec:analytic-geometry}
and~\ref{sec:shell-compactness} give the common preparations. We prove the
theta theorem in Sections~\ref{sec:theta-local}--\ref{sec:theta-global}
and the Epstein theorem in
Sections~\ref{sec:epstein-local}--\ref{sec:epstein-global}.

In Section~\ref{sec:introduction}, we fix the notation and state the main
results. We also explain the relation with the height problem and the
need for separate theta and Epstein arguments.

In Section~\ref{sec:analytic-geometry}, we introduce the analytic and
geometric preparations. Poisson summation gives the Ewald representation
and the critical finite part. We then describe the modified Grenier
fundamental domain and the deformation coordinates used near FCC and BCC.

In Section~\ref{sec:shell-compactness}, we derive the shell estimates and
compute the two Gaussian Hessian components. We then use short-vector
bounds to reduce both minimization problems to compact sets. In
particular, the minima are attained before their locations are determined.

In Section~\ref{sec:theta-local}, we prove the local theta estimates. The
first two shells and the determinant constraint give the FCC comparison
on bounded scale intervals. A six-point exponential-moment inequality
treats large scales, and duality gives the BCC endpoint comparison.

In Section~\ref{sec:theta-global}, we complete the proof of
Theorem~\ref{T:thm:main}. We first reduce the unbounded scale range to value
and derivative inequalities at one scale. Finite Gaussian sums and
whole-box derivative estimates give the remaining comparisons; combining
these with the local results yields all equality cases.

In Section~\ref{sec:epstein-local}, we estimate the Ewald kernel and its
complete shell tails. We use these bounds to control the first three
variations at FCC, and then retain the signed cubic and quartic terms to
enlarge the local neighbourhood for low exponents.

In Section~\ref{sec:epstein-low-comparisons}, we derive the comparison on
complete low-exponent intervals from endpoint values and an upper bound
for the second exponent derivative. We also estimate the BCC--FCC gap and
exclude a full neighbourhood of BCC.

In Section~\ref{sec:epstein-high-local}, we prove the local FCC comparison
for $s\ge8$. We first treat $s\ge20$ by contact-direction and power-tail
estimates, and then use the resulting reference bound in the Taylor
argument for $8\le s\le20$.

In Section~\ref{sec:epstein-global}, we establish the exterior power-sum
comparisons and combine the exponent ranges to prove
Theorem~\ref{E:thm:main}. We conclude with the critical finite part, other
covolumes, and positive Gaussian mixtures supported on one side of the
self-dual scale.

In Appendix~\ref{app:numerical}, we give the evaluation of the finite
inequalities, the arithmetic error bounds and the complete-covering
checks. The four divisions of the appendix also collect the retained
quantitative bounds. Operating instructions remain in the separate
reproducibility supplement.

\section{Ewald representation and lattice geometry}\label{sec:analytic-geometry}\label{sec:common-geometry}\label{E:sec:reduction}

The purpose of this section is to set up the analytic and geometric framework used in both minimization problems. The main analytic conclusions are the pole-cancellation identity in Lemma~\ref{E:lem:Ewald} and the regularity statement in Proposition~\ref{E:prop:finitepart}; the main geometric input is the modified Grenier reduction together with the logarithmic deformation coordinates around FCC and BCC. We first introduce the Ewald kernel and the regularized energy, then describe the reduced parameter domain and finally record the deformation metric used in the local arguments.

\subsection{The Ewald representation and the critical finite part}\label{E:sec:ewald}
With the convention $Y[m]=m^TYm$ fixed in the introduction, write
\[
 \Theta_Y(t):=\Th(t,Y)=\sum_{m\in\Z^3}e^{-\pi tY[m]},\qquad t>0.
\]
For $a\in\R$ and $x>0$, we first recall the positive kernel (\ref{E:eq:G}) introduced in
Section~\ref{sec:proof-strategy}. It agrees with the kernel $G$ in \cite[(24)]{SS}. Directly from the
integral, it is positive, increasing in $a$ and decreasing in $x$.
Differentiating under the integral sign, we obtain
\begin{equation}
 \partial_x^jG(a,x)=(-1)^jG(a+j,x),\qquad
 \partial_a^kG(a,x)=\int_1^\infty e^{-xt}t^{a-1}(\log t)^k\,dt.
 \label{E:eq:G-derivatives}
\end{equation}
In particular, $\partial_a^kG(a,x)>0$ for every integer $k\ge0$,
whereas $(-1)^j\partial_x^jG(a,x)>0$ for every integer $j\ge0$.

The Ewald sum defined in Section~\ref{sec:proof-strategy} is therefore
\begin{equation}
 \HH_s(Y)=\sum_{m\ne0}\bigl[G(s,\pi Y[m])+G(3/2-s,\pi Y^{-1}[m])\bigr].
 \label{E:eq:H}
\end{equation}
Here and in later lattice sums, an index $m\ne0$ with no other range
specified runs over $\Z^3\setminus\{0\}$. The positive sum denoted
$F_a(Y)$ in \cite[(25)]{SS} is related to our convention by
\[
 \HH_s(Y)=F_{3/2-s}(Y)=F_s(Y^{-1}).
\]
In particular, $\HH_{3/2}(Y)$ agrees with their $F_0(Y)$. We shall use
this correspondence when passing between the Epstein energy and height.
The series and their derivatives of any fixed order converge locally
uniformly for $(Y,s)\in\PP\times\R$, where $\PP$ is defined at (\ref{p31}). To see this, fix a compact set of
matrices. Both quadratic forms then have a common lower bound
$c|m|^2$, and the differentiated terms are bounded by summable Gaussian
majorants. The same estimates hold in a sufficiently small complex
neighbourhood, where the real parts of the quadratic forms remain
positive. Local normal convergence of these holomorphic extensions also
proves real analyticity.

We use the complete gamma function
\[
 \Gamma(w):=\int_0^\infty e^{-t}t^{w-1}\,dt,\qquad w>0.
\]
The next lemma gives the identity used in the Epstein comparison. It is
\cite[(23)]{SS}, or equivalently the split Mellin formula (42), written
in our notation. We give the proof to specify the pole terms and the
contribution of the dual lattice.
\begin{lemma}[Pole cancellation]\label{E:lem:Ewald}
For every $Y\in\PP$ and $s>3/2$,
\begin{equation}
 \pi^{-s}\Gamma(s)E(Y,s)=\frac1{s-3/2}-\frac1s+\HH_s(Y),
 \label{E:eq:Ewald}
\end{equation}
and consequently
\begin{equation}
 \pi^{-s}\Gamma(s)[E(Y,s)-E(\YF,s)]=\HH_s(Y)-\HH_s(\YF).
 \label{E:eq:Ewald-gap}
\end{equation}
\end{lemma}
\begin{proof}
Since $\det Y=1$, Poisson summation gives
$\Theta_Y(t)=t^{-3/2}\Theta_{Y^{-1}}(1/t)$. On the other hand, for
$s>3/2$, the Mellin representation is
\[
 \pi^{-s}\Gamma(s)E(Y,s)=\int_0^\infty(\Theta_Y(t)-1)t^{s-1}\,dt.
\]
We split the integral at one and apply Poisson summation on the first
interval. After the change of variables from $t$ to $1/t$, the
lattice-independent terms sum to $(s-3/2)^{-1}-s^{-1}$. Expanding the
two remaining theta functions yields \eqref{E:eq:Ewald}. Absolute
convergence justifies the changes of variables and the interchange of
summation and integration. Subtracting the identity at FCC proves
\eqref{E:eq:Ewald-gap}.
\end{proof}

An important consequence of \eqref{E:eq:H} is that every summand is
positive. Thus a finite subsum gives a lower bound for a competitor,
whereas an upper bound for the FCC value requires the omitted tail as
well. We shall use this distinction in all exterior comparisons.

We next determine the finite part $C(Y)$ and its relation to height.
For $w>0$, we write $\psi(w)=\Gamma'(w)/\Gamma(w)$ for the logarithmic
derivative of the complete gamma function.
\begin{proposition}[Regularity of the critical finite part]\label{E:prop:finitepart}
For every $Y\in\PP$, the finite part $C(Y)$ exists and is real analytic in
$Y$. With $s_0=3/2$ and $\psi=\Gamma'/\Gamma$,
\begin{equation}\label{E:eq:critical-expansion}
 E(Y,s)={}\frac{2\pi}{s-s_0}\notag
 +2\pi\bigl(\HH_{s_0}(Y)-2/3-\psi(s_0)+\log\pi\bigr)
   +O(s-s_0).
\end{equation}
The remainder is uniform for $Y$ in a fixed compact subset of $\PP$.
In particular, for any $Y,Y_0\in\PP$,
\begin{equation}\label{E:eq:finitepart-gap}
 C(Y)-C(Y_0)=2\pi\bigl(\HH_{3/2}(Y)-\HH_{3/2}(Y_0)\bigr).
\end{equation}
This identity follows directly from the Laurent expansion.
\end{proposition}
\begin{proof}
Set $a(s)=\pi^s/\Gamma(s)$. This function is analytic near $s_0$, and
$a(s_0)=2\pi$, while $a'(s_0)=2\pi(\log\pi-\psi(s_0))$.
Multiplying \eqref{E:eq:Ewald} by $a(s)$ and expanding at the critical
exponent gives \eqref{E:eq:critical-expansion}. The local normal
convergence of $\HH_s(Y)$ and its derivatives makes the remainder
uniform on compact matrix sets. Its constant term proves the existence
and real analyticity of $C$. Subtracting the two constant terms gives
\eqref{E:eq:finitepart-gap}, as required.
\end{proof}

Recall the spectral height $h(L)=Z_L'(0)$ and spectral zeta function
$Z_L(w)=(4\pi^2)^{-w}E(L^*,w)$ defined in Section~\ref{sec:proof-strategy}.
The continuation obtained
from \eqref{E:eq:Ewald} satisfies $E(L,0)=-1$; the functional equation
\cite[(2)]{SS} and height formula \cite[(26)]{SS} give
\begin{equation}
 h(L)=\log(4\pi)-\gamma-\frac23+\HH_{3/2}(L)
      =\frac{C(L)}{2\pi}+2-2\gamma.
 \label{E:eq:height-finitepart}
\end{equation}
The second equality follows from $\psi(3/2)=2-\gamma-2\log2$.
Consequently, comparing $\HH_{3/2}$ is equivalent to comparing heights
or critical finite parts. This formulation avoids evaluating the
defining Epstein series at its pole.

\subsection{Reduction and reference lattices}\label{sec:reduced-representatives}
We now fix the geometric domain for the minimization. We use the
modified Grenier fundamental domain for the congruence action of
$GL(3,\Z)$ on $\PP$ adopted by Sarnak and Str\"ombergsson
\cite[Section~4, p.~15]{SS}. Their underlying reduction reference is
Terras \cite[Section~4.4.3, (4.34)]{TerrasII}. For our purposes, the
essential property is that every lattice class has a representative in
the closed domain. We retain all boundary faces and do not use
uniqueness of representatives on them.

In the Iwasawa parametrization \cite[(27)]{SS}, their variables
$t_{12},t_{13},t_{23}$ are denoted here by $a,b,c$, respectively.
We write $\diag(d_1,d_2,d_3)$ for the diagonal matrix with these entries:

\begin{equation}\label{T:eq:iwasawa}
 Y=N^TDN,\quad
 N=\begin{pmatrix}1&a&b\\0&1&c\\0&0&1\end{pmatrix},\quad
 D=\diag(d_1,d_2,d_3),
\end{equation}
where $y_1,y_2>0$, $a,b,c\in\R$, and the logarithmic coordinates are
$\ell_1=\log y_1$ and $\ell_2=\log y_2$. The diagonal factors are
\begin{equation}\label{T:eq:di}
 d_1=e^{(2\ell_1+\ell_2)/3},\qquad
 d_2=e^{(-\ell_1+\ell_2)/3},\qquad
 d_3=e^{(-\ell_1-2\ell_2)/3}.
\end{equation}
Every matrix in $\calP=\{Y=Y^T>0:\det Y=1\}$ has a unique Iwasawa
parametrization before integral changes of basis are imposed. Thus the
logarithmic coordinates only reparametrize the positive variables in
\cite[(27)]{SS}. In \cite{SS}, Sarnak and Str\"ombergsson have obtained the modified domain
by translating one half of Grenier's domain under
$Y\mapsto T_2^TYT_2$, where
\[
 T_2=\begin{pmatrix}1&0&1\\0&1&0\\0&0&1\end{pmatrix}\in GL(3,\Z).
\]
This is their modification preceding the inequalities on p.~15 of
\cite{SS}. In our notation, those inequalities are
\begin{align}\label{T:eq:grenier}
 1&\le(1+a-b)^2+e^{-\ell_1}\bigl((1-c)^2+e^{-\ell_2}\bigr),\notag\\
 1&\le(a-b)^2+e^{-\ell_1}\bigl((1-c)^2+e^{-\ell_2}\bigr),\notag\\
 1&\le a^2+e^{-\ell_1},\notag\\
 1&\le b^2+e^{-\ell_1}(c^2+e^{-\ell_2}),\\
 1&\le(b-1)^2+e^{-\ell_1}(c^2+e^{-\ell_2}),\notag\\
 1&\le c^2+e^{-\ell_2},\notag
\end{align}
with
\begin{equation}\label{T:eq:abcbox}
 0\le a\le\frac12,\qquad0\le b\le1,\qquad0\le c\le\frac12.
\end{equation}
By \cite[Section~4, p.~15]{SS}, every covolume-one lattice class has a
representative satisfying \eqref{T:eq:grenier}--\eqref{T:eq:abcbox}.
Hence the reduction applies to the full five-dimensional space of
lattice shapes, not merely to a symmetric family. Only this covering
property is needed below.

Equivalently, in the nonlogarithmic coordinates used for the Epstein
comparison,
\[
 d_1=y_1^{2/3}y_2^{1/3},\qquad d_2=d_1/y_1,\qquad d_3=d_2/y_2.
\]
For the Epstein argument, we replace $e^{-\ell_i}$ by $y_i^{-1}$ in
\eqref{T:eq:grenier}. This gives the same reduced domain in the
nonlogarithmic coordinates. We denote the resulting Gram matrix by
$Y(\xi)$, with $\xi=(\ell_1,\ell_2,a,b,c)$ or
$\xi=(y_1,y_2,a,b,c)$ as specified in each argument. A shape box is a
closed coordinate rectangle. The notation $Y\in B$ means
$Y=Y(\xi)$ for some $\xi\in B$.

The distinguished representatives are those listed immediately after
the domain inequalities in \cite[Section~4, p.~15]{SS}:
\begin{align}\label{T:eq:reps}
 \FCC:&\quad (y_1,y_2,a,b,c)=(4/3,9/8,1/2,1/2,1/3),\notag\\
 \BCC:&\quad (y_1,y_2,a,b,c)=(9/8,4/3,1/3,2/3,1/2).
\end{align}
For the local estimates, we choose the following physical bases. The
FCC basis is the matrix $g_0$ of \cite[Section~4, p.~17]{SS}; the BCC
basis is chosen compatibly with duality, which we verify below:
\begin{equation}\label{T:eq:bases}
 \gFT=\rho^{-1}\begin{pmatrix}1&0&1\\-1&-1&0\\0&1&1\end{pmatrix},\qquad
 \gBT=\rho^{-2}\begin{pmatrix}1&-1&0\\1&1&0\\1&1&2\end{pmatrix}.
\end{equation}
Then $\gFT\Z^3=\FCC$, $\gBT\Z^3=\BCC$, and their Gram matrices are
\begin{equation}\label{T:eq:grams}
 \YF=\rho\begin{pmatrix}1&1/2&1/2\\1/2&1&1/2\\1/2&1/2&1\end{pmatrix},\qquad
 \YB=\frac{\rho^2}{4}\begin{pmatrix}3&1&2\\1&3&2\\2&2&4\end{pmatrix}.
\end{equation}
For any Gram matrix $Y$, define the physical FCC strain
\begin{equation}\label{T:eq:AF}
 A_F(Y)=\gFT^{-T}Y\gFT^{-1}.
\end{equation}
The lattice with Gram matrix $Y$ is orthogonally equivalent to $A_F(Y)^{1/2}\FCC$. Near BCC it is more convenient to use the dual strain
\begin{equation}\label{T:eq:ABinv}
 A_B(Y)=\gBT^{-T}Y\gBT^{-1},\qquad
 A_B(Y)^{-1}=\gBT Y^{-1}\gBT^T.
\end{equation}
These identities allow us to apply the local inequalities in the
reduced coordinates, with the reference lattice and deformation fixed.

For a real symmetric matrix $X$, the matrix exponential is defined by
the series $e^X:=\sum_{j=0}^\infty X^j/j!$. For a positive definite
matrix with spectral decomposition
$A=U\diag(\lambda_1,\lambda_2,\lambda_3)U^T$, where $U$ is orthogonal
and $\lambda_i>0$, its symmetric logarithm and square root are
\[
 \begin{aligned}
  \log A&:=U\diag(\log\lambda_1,\log\lambda_2,\log\lambda_3)U^T,\\
  A^{1/2}&=U\diag(\sqrt{\lambda_1},\sqrt{\lambda_2},\sqrt{\lambda_3})U^T.
 \end{aligned}
\]

\begin{remark}[Two frames and two neighbourhood conventions]
The matrices $\gFE$ and $\gFT$ generate the same physical FCC lattice and
have the same Gram matrix. They are kept distinct because the two arguments use different matrix
formulas for strain containment.
The theta estimates bound $\|A-I\|_{\mathrm{op}}$, whereas the Epstein
estimates bound $\|\log A\|_{\mathrm F}$. Radii in these two conventions are not identified with one another.
\end{remark}

\subsection{Deformations and neighbourhood containment}\label{E:sec:geometry}\label{T:app:matrices}
We next describe the logarithmic coordinates used in the local Epstein
estimates. For a cubic reference lattice with basis $g_0$, set
\begin{equation}
 Y(X)=g_0^Te^Xg_0,\qquad X=X^T,\quad\tr X=0.
 \label{E:eq:exp-coordinates}
\end{equation}
For $U,V\in T_Y\PP$, use the invariant metric
\[
 \langle U,V\rangle_Y=\tr(Y^{-1}UY^{-1}V),\qquad
 T_Y\PP=\{U=U^T:\tr(Y^{-1}U)=0\}.
\]
The normalization is that of \cite[Section~2, p.~7]{SS}. By the
geodesic and distance formulas \cite[(30)--(31)]{SS}, the distance in
$\PP$ from $g_0^Tg_0$ to $Y(X)$ equals $\norm X$. We work with these
explicit representatives when testing neighbourhood containment. Set
$C=g_0^{-T}Yg_0^{-1}$. If $\opnorm{C-I}<1$, the logarithm series gives
\begin{equation}
 \norm{\log C}\le\frac{\norm{C-I}}{1-\opnorm{C-I}}.
 \label{E:eq:log-norm-bound}
\end{equation}
If, in addition, $\norm{C-I}<1$, then
\[
 \norm{\log C}\le\frac{\norm{C-I}}{1-\norm{C-I}}.
\]
The stronger Frobenius-norm condition ensures that the last denominator
is positive. Keeping the operator-norm bound in the logarithm series
also gives the sharper estimate
\begin{equation}\label{eq:strain-log-bridge}
 \norm{\log C}\le\delog(\delta)\norm{C-I},\qquad
 \delog(\delta)=
 \begin{cases}
 -\log(1-\delta)/\delta,&0<\delta<1,\\
 1,&\delta=0,
 \end{cases}
\end{equation}
whenever $\opnorm{C-I}\le\delta<1$. Indeed,
$\|(C-I)^k\|_{\frob}\le\delta^{k-1}\|C-I\|_{\frob}$, and summation of
the logarithm series gives \eqref{eq:strain-log-bridge}. Conversely, for
$X\in\Symzero$ and $r=\norm X$,
\begin{equation}\label{eq:log-strain-bridge}
 \opnorm{e^X-I}\le e^{\sqrt{2/3}\,r}-1.
\end{equation}
Indeed, if $\lambda$ is an eigenvalue of $X$, the other two eigenvalues
sum to $-\lambda$. Hence $\lambda^2\le2(r^2-\lambda^2)$, and
$|\lambda|\le\sqrt{2/3}\,r$. Applying the scalar exponential proves
the claimed estimate. These bounds relate the two notions of radius;
they do not identify the strain and logarithmic neighbourhoods.

We shall use the following enclosure convention. On a fixed domain,
$-$ and $+$ as subscripts or superscripts indicate lower and upper
bounds: $z^-\le z\le z^+$, or equivalently $z_-\le z\le z_+$.
They do not denote positive and negative parts. For a signed quantity,
$|z|_+$ denotes an upper bound for its absolute value.

To control an entire neighbourhood, fix a box $B$ in
$(y_1,y_2,a,b,c)$ and a centre in $B$. Rather than estimating the
entries of $Y-Y_0$ separately, we estimate the length of the coordinate
segment from the centre to each point. Let $r_1,r_2,r_a,r_b,r_c$
bound the distances from the centre to both endpoints in the respective
coordinates, and put
\[
 u=r_1/(y_1)_-,\quad v=r_2/(y_2)_-,\quad t_c=\max_B|c|.
\]
The straight coordinate segment from the centre to any point of $B$ has length at most the square root of
\begin{equation}
 \frac23(u^2+uv+v^2)
 +2(y_1)_+r_a^2+2(y_2)_+r_c^2
 +2(y_1)_+(y_2)_+(r_b+t_cr_a)^2.
 \label{E:eq:metric-path}
\end{equation}
To derive the bound, set $\Omega=dN\,N^{-1}$. Its upper-triangular
entries are $da,db-c\,da,dc$, and the squared line element is
\[
 \sum_i(d\log d_i)^2+2\sum_{i<j}(d_i/d_j)\Omega_{ij}^2.
\]
Substituting the three factors $d_i$ gives the first quadratic term in
\eqref{E:eq:metric-path}. Bounding the remaining factors by their
coordinate endpoints gives the other terms. Integration over the unit
parameter interval proves the length estimate.

Combining \eqref{E:eq:log-norm-bound} with \eqref{E:eq:metric-path}
and the triangle inequality, we obtain the required containment bound
for every point of the shape box. In particular, checking the centre
alone is not sufficient.

For FCC, let $J$ be the all-ones matrix and set $M_F=J-2I$.
Then $\gFE^{-1}=\rho M_F/2$, and the relative strain is
$\rho^2M_F^TYM_F/4$. For the BCC representative in Section~\ref{E:sec:reduction}, set
\[
 B_0=\begin{pmatrix}-1&0&1\\-1&0&-1\\1&1&0\end{pmatrix}.
\]
Define $\gBE$ by $\gBE^{-1}=\rho^{-1}B_0$. Then
$\gBE\Z^3=\BCC$, $\gBE^T\gBE=\YB$, and the BCC relative strain is
$\rho^{-2}B_0^TYB_0$. Thus both formulae equal $I$ at their respective
reference Gram matrices. In the Epstein branch we write
\[
 Y_{\mathrm B}(X)=\gBE^Te^X\gBE,\qquad X\in\Symzero.
\]
Thus a local assertion about a lattice requires the existence of a
representative of this form within the stated norm bound. No condition
is imposed on an arbitrary choice of lattice basis.

Finally, we record the inverse matrices needed later. They verify the
duality of the chosen FCC and BCC bases and reduce the neighbourhood
calculations to explicit scalar formulas.

The inverse FCC basis used in \eqref{T:eq:AF} is
\[
 \gFT^{-1}=\frac\rho2
 \begin{pmatrix}1&-1&-1\\-1&-1&1\\1&1&1\end{pmatrix}.
\]
The duality between the two physical bases can be checked without approximating $\rho$:
\[
 \gFT^T\gBT=
 \begin{pmatrix}0&-1&0\\0&0&1\\1&0&1\end{pmatrix}\in\operatorname{GL}_3(\Z).
\]
Together with unit covolumes, this verifies $\gBT\Z^3=(\gFT\Z^3)^*$.

For the Iwasawa matrix \eqref{T:eq:iwasawa}, write $z=ac-b$. Direct inversion gives
\begin{equation}\label{T:eq:inverseY}
 Y^{-1}=\begin{pmatrix}
 d_1^{-1}+a^2d_2^{-1}+z^2d_3^{-1}
 &-ad_2^{-1}-zcd_3^{-1}&zd_3^{-1}\\
 -ad_2^{-1}-zcd_3^{-1}
 &d_2^{-1}+c^2d_3^{-1}&-cd_3^{-1}\\
 zd_3^{-1}&-cd_3^{-1}&d_3^{-1}
 \end{pmatrix}.
\end{equation}
The inverse forms and relative strains can therefore be evaluated
without general matrix inversion. Their interval enclosures are
described in Appendix~\ref{app:numerical}.

\section{Leading-order estimates and compact reduction}\label{sec:shell-compactness}

The two main conclusions of this section are the following cusp-exclusion results. The first concerns the theta energy and the second the regularized Epstein energy. Once these propositions are established, all remaining competitors lie in a fixed compact part of the reduced lattice space; the rest of the section proves them from explicit shell estimates and monotone kernel bounds.

\begin{proposition}[Cusp exclusion]\label{T:prop:cusp}
Suppose that $L$ has a nonzero vector with squared length at most $2/3$. Then, for every $\alpha\ge1$,
\begin{equation}\label{T:eq:cuspstrict}
 \Th(\alpha,L)>\Th(\alpha,\FCC).
\end{equation}
\end{proposition}

The preceding proposition removes the cusp for the Gaussian problem by comparing the contribution of a shortest vector with the FCC reference. For the Epstein argument we need an analogous exclusion which is uniform in the exponent. The next proposition provides precisely this second compactness input, now for the regularized Ewald functional.

\begin{proposition}[Uniform cusp exclusion]\label{E:prop:cusp}
If a covolume-one lattice has a nonzero vector of squared length $q<71/100$, then
\[
 \HH_s(L)>\HH_s(\FCC)\qquad(s\ge3/2).
\]
In particular this holds whenever $d_1<71/100$ in \eqref{T:eq:iwasawa}.
\end{proposition}

We now establish the estimates used in these two propositions.

\subsection{Cubic shells and Gaussian variations}\label{sec:cubic-shells}\label{T:app:constants}
For a positive definite matrix $Y$, there is a constant $c>0$ such
that $z^TYz\ge c|z|^2$. This Gaussian majorant proves absolute and
local uniform convergence of \eqref{T:eq:theta} in $(\alpha,Y)$,
together with all derivatives of any fixed order on compact sets with
$\alpha>0$. It therefore justifies the termwise differentiations and
shell decompositions below.

We begin with the multiplicities of the FCC shells. Write
\begin{equation}\label{T:eq:shells}
 N_n=\#\{v\in\FCC:|v|^2=n\rho\},\qquad n\ge1.
\end{equation}
By the integer realization \eqref{E:eq:FCC}, the multiplicity is the
number of representations of $2n$ as a sum of three squares. The
required parity condition holds automatically, since this sum is even.
In particular,
\begin{equation}\label{T:eq:firstcounts}
 N_1=12,\qquad N_2=6.
\end{equation}
The first shell consists of the permutations of $\rho^{-1}(\pm1,\pm1,0)$, and the second consists of $\rho^{-1}(\pm2,0,0)$ and its coordinate permutations.

\begin{lemma}[Elementary shell bound]\label{T:lem:shellbound}
For every $n\ge1$, $N_n<30n$. In addition, every FCC shell has second moment
\begin{equation}\label{T:eq:isotropic}
 \sum_{|v|^2=n\rho}vv^T=\frac{n\rho N_n}{3}\Id.
\end{equation}
\end{lemma}
\begin{proof}
Once the first two integer coordinates are chosen, there are at most two choices for the third. Hence
\[
 N_n\le2(2\sqrt{2n}+1)^2
 =16n+8\sqrt{2n}+2<30n.
\]
Dividing by $n$, the last inequality follows from $n\ge1$ and
$16+8\sqrt2+2<30$. To prove the second assertion, observe that each
shell is invariant under sign changes and coordinate permutations.
The mixed moments vanish, and the three diagonal moments coincide.
Taking the trace now gives \eqref{T:eq:isotropic}.
\end{proof}

To compare the theta energies, we normalize by the decay of the first
FCC shell. Put
\begin{equation}\label{T:eq:uCF}
 u=\pi\rho\alpha,\quad q=e^{-u},\quad
 C_{\FCC}(\alpha)=e^u\bigl(\Th(\alpha,\FCC)-1\bigr)
 =\sum_{n\ge1}N_n q^{n-1}.
\end{equation}
The function $C_{\FCC}$ is decreasing. We use the elementary inequalities
\begin{equation}\label{T:eq:basicnumbers}
 \pi\rho>\frac{79}{20}=3.95,
 \qquad e^{39/10}>49,
 \qquad e^{79/20}>50.
\end{equation}
The scalar inequalities used in Section~\ref{T:app:constants} are
collected below. Their rational evaluation is given in
Appendix~\ref{app:numerical}.

\begin{lemma}[Normalized FCC upper bounds]\label{T:lem:CFbounds}
For $\alpha\ge1$,
\begin{equation}\label{T:eq:CF125}
 C_{\FCC}(\alpha)<\frac{25}{2}.
\end{equation}
If $R_{12}(\alpha)=\sum_{n=1}^{12}N_n e^{-(n-1)u}$, then
\begin{equation}\label{T:eq:reftail}
 0<C_{\FCC}(\alpha)-R_{12}(\alpha)<2\cdot10^{-18}.
\end{equation}
\end{lemma}
\begin{proof}
Here $q<1/50$. The first two shells and Lemma~\ref{T:lem:shellbound} give
\[
 C_{\FCC}(\alpha)
 \le12+6q+30\sum_{n\ge3}nq^{n-1}
 =12+6q+30\frac{q^2(3-2q)}{(1-q)^2}.
\]
At $q=1/50$ the last expression equals $729723/60025<25/2$. Similarly,
\[
 C_{\FCC}-R_{12}
 \le30\sum_{n\ge13}nq^{n-1}
 =30\frac{q^{12}(13-12q)}{(1-q)^2}<2\cdot10^{-18}.
\]
Both right-hand sides increase with $q\in(0,1)$, so evaluation at $q=1/50$ is legitimate.
\end{proof}

We next compute the Hessian at FCC. This calculation explains which
shells must be retained in the nonlinear estimate. For a traceless
symmetric matrix $H$, set $L_t=e^{tH/2}\FCC$. Two differentiations
give
\begin{equation}\label{T:eq:hessian}
 \left.\frac{d^2}{dt^2}\Th(\alpha,L_t)\right|_{t=0}
 =\sum_{v\in\FCC}e^{-a|v|^2}
 \left(a^2(v^THv)^2-a\,v^TH^2v\right),\qquad a=\pi\alpha.
\end{equation}
By \eqref{T:eq:isotropic}, the first variation is zero. Under cubic
symmetry, the tangent space splits into the two-dimensional diagonal
trace-free subspace and the three-dimensional off-diagonal subspace.
The Hessian is determined by its coefficients on these two subspaces.

If
\[
 C_2=\sum v_1^2e^{-a|v|^2},\quad
 A_4=\sum v_1^4e^{-a|v|^2},\quad
 B_4=\sum v_1^2v_2^2e^{-a|v|^2},
\]
then the two Hessian eigenvalues, in the Frobenius metric, are
\begin{equation}\label{T:eq:blocks}
 \mu_D=a^2(A_4-B_4)-aC_2,\qquad
 \mu_O=2a^2B_4-aC_2.
\end{equation}
These coefficients follow by evaluating the Hessian on
$\diag(1,-1,0)/\sqrt2$ and on the unit matrix with entries
$H_{12}=H_{21}=1/\sqrt2$, respectively.

The first shell contributes
\begin{equation}\label{T:eq:hessfirst}
 \mu_D^{(1)}=u(u-4)e^{-u},\qquad
 \mu_O^{(1)}=2u(u-2)e^{-u}.
\end{equation}
At $\alpha=1$, $u=\pi\rho<4$, so the first diagonal contribution is negative. The second shell contributes
\begin{equation}\label{T:eq:hesssecond}
 \mu_D^{(2)}=4u(2u-1)e^{-2u},\qquad
 \mu_O^{(2)}=-4ue^{-2u}.
\end{equation}
The negative diagonal contribution of the first shell is the reason
for retaining the second shell near the transition. Moreover, local
positivity requires both Hessian components to be positive; a positive
sum of the two does not suffice. The invariant decomposition is that of
\cite[Section~3, (21)]{SS}, while the coefficients above correspond
to the Gaussian kernel.

For later use, we collect the elementary exponential estimates. They
follow by bounding the positive Taylor remainder; the rational checks
are given in Appendix~\ref{app:constant-checks}. For $x\ge0$ and
$N+2>x$, set
\[
 S_N(x)=\sum_{k=0}^N\frac{x^k}{k!}.
\]
The elementary ratio bound on the positive tail gives
\begin{equation}\label{T:eq:exp-rational}
 S_N(x)\le e^x\le S_N(x)+
 \frac{x^{N+1}}{(N+1)!}\frac1{1-x/(N+2)}.
\end{equation}
The lower bound is strict when $x>0$, with equality at $x=0$.
Taking $N=100$ reduces the inequalities in
Table~\ref{T:tab:constants} to rational comparisons.
Appendix~\ref{app:numerical} gives the evaluation and rounding-error
bounds.

The inequalities and their roles are listed in Table~\ref{T:tab:constants}.
\begin{table}[!htbp]
\centering
\caption{Scalar exponential inequalities used in the theta estimates.}\label{T:tab:constants}
\begin{tabular}{@{}ll@{}}
\toprule
Inequality & Role\\
\midrule
$49<e^{39/10}<50$ & FCC local shell bounds\\
$50<e^{79/20}$ & Uniform FCC reference bound\\
$e^{37/20}>25/4$ & Cusp exclusion\\
$e^{102/25}<60$ & First local subinterval\\
$e^{1071/250}<75$ & Second local subinterval\\
$e^{561/125}<100$ & Third local subinterval\\
$e^{79/250}<11/8$ & Larger local ball, lower endpoint\\
$e^{64/25}<13$ & Larger local ball, upper endpoint\\
$e^{31/25}>10/3$ & BCC derivative tail\\
$e^{7/10}>2$ & Small-exponential cutoff\\
\bottomrule
\end{tabular}
\end{table}
The rational inequalities $1.259<\rho<1.26$ follow by cubing. Together with $3.14<\pi<22/7$, they imply
\[
 \pi\rho>3.95,\qquad
 \pi(\rho-2/3)>1.85,\qquad
 31/25<\pi\rho^2/4<2.
\]
The corresponding arctangent-series bounds for $\pi$ and exact cube comparisons for $\rho$ are given in Appendix~\ref{app:numerical}.

\subsection{Cusp exclusion and compact comparison domains}\label{sec:cusps}
We now turn to compact reduction. For the theta energy, a sufficiently
short vector and its negative already give a lower bound exceeding the
FCC value. For $i=1,2,3$, let $e_i$ denote the $i$th standard coordinate
column vector of $\R^3$.

We now prove Proposition~\ref{T:prop:cusp}.

\begin{proof}
The vector and its negative give
\[
 \Th(\alpha,L)-1\ge2e^{-2\pi\alpha/3}
 =2e^{-u}e^{\pi\alpha(\rho-2/3)}.
\]
By $\pi(\rho-2/3)>37/20$ and $e^{37/20}>25/4$, this lower bound
exceeds $(25/2)e^{-u}$. Lemma~\ref{T:lem:CFbounds} then yields the
strict comparison. The same estimate holds uniformly in $\alpha$;
the selected vector need not be a shortest vector.
\end{proof}

Since $Y[(1,0,0)]=d_1$, Proposition~\ref{T:prop:cusp} reduces the
comparison to $d_1\ge2/3$. The third and sixth inequalities in
\eqref{T:eq:grenier} give $y_1,y_2\le4/3$. Substituting these bounds
in $d_1^3=y_1^2y_2$, we obtain
\begin{equation}\label{T:eq:ycompact}
 \frac{\sqrt2}{3}\le y_1\le\frac43,
 \qquad \frac16\le y_2\le\frac43.
\end{equation}
In particular, the remaining reduced domain is contained in
\begin{equation}\label{T:eq:rootmath}
 \mathcal R=[-1,1/3]\times[-2,1/3]\times[0,1/2]\times[0,1]\times[0,1/2]
\end{equation}
in the coordinates $(\ell_1,\ell_2,a,b,c)$. Let $\calK$ denote the closed subset of this rectangle satisfying \eqref{T:eq:grenier} and $d_1\ge2/3$. It is compact.

We first note that the minimum is attained. The compact set $\calK$
contains the FCC representative. For fixed $\alpha\ge1$, continuity
gives a minimum on $\calK$, while every reduced representative outside
$\calK$ has energy strictly greater than the FCC value by the cusp
estimate. Thus this compact minimum is global. For $0<\alpha<1$,
Poisson summation and the bijection $[L]\mapsto[L^*]$ transfer the
existence of a minimum from the reciprocal scale $1/\alpha$.
This proves attainment at every positive scale, before identifying
the minimizers.

For the Epstein energy, we use the cutoff $71/100$ from the height
argument in \cite[Section~4, pp.~15--16]{SS}. We first record its
endpoint bounds and then extend the exclusion to every exponent under
consideration. The following kernel-ratio lemma gives this extension.

\begin{lemma}[Monotone kernel ratio]\label{E:lem:ratio}
For $a\in\R$, $x>0$ and $u\ge0$, the ratio $G(a+u,x)/G(a,x)$ is at least one and is nonincreasing in $x$.
\end{lemma}
\begin{proof}
Consider the probability density proportional to $e^{-xt}t^{a-1}$ on
$[1,\infty)$. The ratio in the lemma is $\mathbb E_x(t^u)$ and is
therefore at least one. Its derivative with respect to $x$ is
$-\operatorname{Cov}_x(t,t^u)\le0$: the sign follows by integrating
$(t-v)(t^u-v^u)\ge0$ against the product measure. Exponential
integrability justifies differentiation. This proves the lemma.
\end{proof}

\begin{lemma}[Critical reference bounds]\label{E:lem:cusp-inputs}
The following strict inequalities hold:
\begin{equation}
 \HH_{3/2}(\FCC)<0.11336,
 \qquad 2G(3/2,71\pi/100)>0.114
 \label{E:eq:cusp-values}
\end{equation}
\end{lemma}
\begin{proof}[Computer-assisted proof]
For the first inequality, apply the rational kernel bounds of
Section~\ref{E:sec:kernels} to the finite FCC and BCC shells and add
the complete reference tails. For the second, use the corresponding
lower enclosure for the single kernel. The cutoffs, rational
evaluations and enclosure checks are given in
Appendix~\ref{app:elementary-revision}.
\end{proof}

We next prove Proposition~\ref{E:prop:cusp}.

\begin{proof}
We use only the endpoint estimates in Lemma~\ref{E:lem:cusp-inputs}.
Set $R=G(s,\pi q)/G(3/2,\pi q)\ge1$. The squared lengths of all
nonzero FCC vectors are at least $2^{1/3}>q$. By
Lemma~\ref{E:lem:ratio}, each primal FCC term at $s$ is bounded by
$R$ times its value at $3/2$. The dual terms decrease with $s$, so
they satisfy the same bound with factor $R$. Summing, we obtain
\[
 \HH_s(\FCC)\le R\HH_{3/2}(\FCC).
\]
Keeping only the two primal competitor vectors gives
\[
 \HH_s(L)\ge2G(s,\pi q)=2R G(3/2,\pi q)>0.114R>\HH_s(\FCC).
\]
Finally, $Y[e_1]=d_1$.
\end{proof}

The third and sixth reduction inequalities imply $y_1,y_2\le4/3$. If $d_1\ge71/100$, then $y_1^2y_2\ge(71/100)^3$, whence
\[
 y_1\ge\sqrt{\tfrac34(71/100)^3}>\frac{51}{100},
 \qquad y_2\ge\frac9{16}(71/100)^3>\frac15.
\]
Therefore the single compact rectangle
\begin{equation}
 K=[51/100,4/3]\times[1/5,4/3]\times[0,1/2]\times[0,1]\times[0,1/2]
 \label{E:eq:root-rectangle}
\end{equation}
contains every remaining reduced representative for $s\ge3/2$.
Points in $K$ that violate a reduction inequality may either be
excluded or retained for direct comparison. The intersection of $K$
with the closed reduction conditions and $d_1\ge71/100$ is compact
and contains FCC. Hence continuity and the cusp estimate give a
minimum of $\HH_s$ for each $s\ge3/2$. By
\eqref{E:eq:Ewald-gap} and \eqref{E:eq:finitepart-gap}, a minimum
also exists for $E(\cdot,s)$ when $s>3/2$, and for $C$.
It remains to determine the minimizing lattice classes.

\section{Local comparison for the theta energy}\label{sec:theta-local}\label{T:sec:local}

The main results of this section are the following two local comparison statements. Theorem~\ref{T:thm:local} gives a quantitative FCC neighbourhood, uniform in the relevant scale ranges, while Proposition~\ref{T:prop:bcclocal} controls a BCC neighbourhood at the self-dual transition. Their proofs have three steps: a two-shell estimate near FCC, a scale-uniform six-point estimate for large parameters, and finally the duality argument at BCC.

\begin{theorem}[Uniform quantitative FCC neighbourhood]\label{T:thm:local}
Let $A\in\calP$ and $u=\pi\rho\alpha$. Inequality \eqref{T:eq:localquant} holds in either of the following situations:
\begin{enumerate}[label=(\roman*),leftmargin=2em]
\item $u\ge39/10$ and $\|A-\Id\|_{\op}\le1/100$;
\item $\alpha\ge2$ and $\|A-\Id\|_{\op}\le1/25$.
\end{enumerate}
In either neighbourhood, equality of the two theta energies holds if and
only if $A=\Id$. The uniformity concerns the neighbourhood radius; the
unnormalized lower-bound coefficient $(u/50)e^{-u}$ depends on the scale.
\end{theorem}

Theorem~\ref{T:thm:local} settles the FCC side of the local analysis. At the self-dual scale, however, the BCC lattice also has to be controlled because it is the dual symmetric configuration and lies close to the transition region. The next proposition gives the complementary BCC estimate; together, these two statements provide the local input used later in the global theta comparison.

\begin{proposition}[Quantitative BCC comparison at the transition]\label{T:prop:bcclocal}
Let $Y\in\calP$ and suppose
\begin{equation}\label{T:eq:bccball}
 \|\gBT Y^{-1}\gBT^T-\Id\|_{\op}\le\frac1{100}.
\end{equation}
For $1\le\alpha\le101/100$,
\begin{equation}\label{T:eq:bcccomparison}
 \Th(\alpha,Y)\ge\Th(\alpha,\BCC)\ge\Th(\alpha,\FCC).
\end{equation}
More precisely, put $C_B=A_B(Y)^{-1}$ and $u_* = \pi\rho/\alpha$. Then
\begin{align}\label{T:eq:bcc-quantitative}
 \Th(\alpha,Y)-\Th(\alpha,\FCC)
 \ge{}&\frac{\alpha^{-3/2}u_*e^{-u_*}}{50}
          \|C_B-I\|_{\frob}^2\notag\\
      &+\frac{\alpha-1}{200}.
\end{align}
Equality of the leftmost and rightmost energies in
\eqref{T:eq:bcccomparison} holds if and only if $\alpha=1$ and $Y=\YB$.
\end{proposition}

We proceed to the estimates entering these two statements.

\subsection{Two-shell comparison near FCC}
We begin with the determinant constraint. To distinguish the two
quadratic contributions, for a symmetric matrix $T$ write
\[
 D(T)=\sum_{i=1}^3T_{ii}^2,\qquad
 O(T)=\sum_{1\le i<j\le3}T_{ij}^2,
 \qquad \|T\|_{\frob}^2=D(T)+2O(T).
\]

\begin{lemma}[Trace control on the determinant-one surface]\label{T:lem:det}
Suppose $A=\Id+T\in\PP$ and $\|T\|_{\op}\le\delta<1$, with $\delta\ge0$. Then
\begin{equation}\label{T:eq:tracebounds}
 \frac{\|T\|_{\frob}^2}{2(1+\delta)}
 \le\tr T\le
 \frac{\|T\|_{\frob}^2}{2(1-\delta)}.
\end{equation}
Put $\eta=\tr T/3$ and $S=T-\eta\Id$. If $\delta<2/3$, then
\begin{equation}\label{T:eq:etabounds}
 0\le\eta\le\frac{\|S\|_{\frob}^2}{6-9\delta},
 \qquad \|S\|_{\frob}\le\sqrt3\,\delta.
\end{equation}
\end{lemma}
\begin{proof}
For $|x|\le\delta$,
\[
 x-\log(1+x)=x^2\int_0^1\frac{r}{1+rx}\,dr,
\]
and therefore
\[
 \frac{x^2}{2(1+\delta)}\le x-\log(1+x)
 \le\frac{x^2}{2(1-\delta)}.
\]
Apply the scalar inequalities to the eigenvalues of $T$ and sum.
Since $\sum\log(1+x_i)=\log\det A=0$, this gives
\eqref{T:eq:tracebounds}, and in particular $\eta\ge0$.
We also have $\eta\le\delta$. Substituting
$\|T\|_{\frob}^2=\|S\|_{\frob}^2+3\eta^2$ in the upper trace
bound, we obtain
\[
 6(1-\delta)\eta\le\|S\|_{\frob}^2+3\eta^2
 \le\|S\|_{\frob}^2+3\delta\eta.
\]
Rearranging proves the first inequality in \eqref{T:eq:etabounds}.
The second follows from
$\|S\|_{\frob}^2\le\|T\|_{\frob}^2\le3\delta^2$.
This completes the proof.
\end{proof}

We next estimate the theta difference. The first two shells provide
the quadratic terms, while the remaining shells are controlled by a
linear lower bound. For $u\ge39/10$ and $0\le\delta<1$, define
\begin{align}\label{T:eq:Bcoeff}
 B_\delta(u)&=\frac{4+4e^{-u}+1/25}{2(1-\delta)},\notag\\
 C_D(u,\delta)&=\frac{u}{2}e^{-\delta u}
      +4u e^{-(1+2\delta)u}-B_\delta(u),\\
 C_O(u,\delta)&=2u e^{-\delta u}-2B_\delta(u).\notag
\end{align}

\begin{lemma}[Nonlinear two-shell lower bound]\label{T:lem:twoshell}
Let $A=\Id+T$ satisfy the assumptions of Lemma~\ref{T:lem:det}. If $u=\pi\rho\alpha\ge39/10$, then
\begin{equation}\label{T:eq:twoshell}
 e^u\bigl(\Th(\alpha,A^{1/2}\FCC)-\Th(\alpha,\FCC)\bigr)
 \ge u\bigl(C_D(u,\delta)D(T)+C_O(u,\delta)O(T)\bigr).
\end{equation}
\end{lemma}
\begin{proof}
We treat the first shell, the second shell and the remainder in turn.
For a first-shell vector, put $b_v=\rho^{-1}v^TTv$. Its twelve values
are
\[
 \frac{T_{ii}+T_{jj}}2+T_{ij},\qquad
 \frac{T_{ii}+T_{jj}}2-T_{ij},\qquad i<j,
\]
each repeated twice. They satisfy $|b_v|\le\delta$ and
\begin{equation}\label{T:eq:firstmoments}
 \sum b_v=4\tr T,\qquad
 \sum b_v^2=D(T)+(\tr T)^2+4O(T).
\end{equation}
The scalar inequality
\begin{equation}\label{T:eq:scalarquad}
 e^{-x}\ge1-x+\tfrac12e^{-M}x^2\qquad(|x|\le M)
\end{equation}
follows from Taylor's integral remainder. Thus the first-shell change, after multiplication by $e^u$, is at least
\begin{equation}\label{T:eq:firstlower}
 -4u\tr T+\frac{u^2}{2}e^{-\delta u}
       \bigl(D(T)+(\tr T)^2+4O(T)\bigr).
\end{equation}
On the second shell, the normalized changes of squared length are
$2T_{ii}$, each occurring twice. Applying the same scalar inequality,
we obtain the lower bound
\begin{equation}\label{T:eq:secondlower}
 -4ue^{-u}\tr T+4u^2e^{-(1+2\delta)u}D(T).
\end{equation}

It remains to estimate the higher shells. We use
$e^{-x}\ge1-x$ and the shell identity \eqref{T:eq:isotropic}.
The contribution of each such shell is bounded below by
\[
 -\frac{u}{3}nN_n e^{-(n-1)u}\tr T.
\]
The coefficient from all shells $n\ge3$ is controlled by
\begin{align}\label{T:eq:Atail}
 \frac13\sum_{n\ge3}nN_nq^{n-1}
 &\le10\sum_{n\ge3}n^2q^{n-1}\\
 &=10\frac{q^2(9-11q+4q^2)}{(1-q)^3}<\frac1{25}.\notag
\end{align}
Indeed, $q\le1/49$ and evaluation of the middle expression at $q=1/49$ gives $52685/1354752<1/25$.

Combining \eqref{T:eq:firstlower}, \eqref{T:eq:secondlower} and
\eqref{T:eq:Atail}, we discard the nonnegative term containing
$(\tr T)^2$ and then use the upper bound for $\tr T$ in
\eqref{T:eq:tracebounds}. The coefficients of $D(T)$ and $O(T)$
are $uC_D$ and $uC_O$, respectively. This is
\eqref{T:eq:twoshell}.
\end{proof}

\begin{proposition}[A one-percent FCC ball, bounded scales]\label{T:prop:smallballbounded}
Let $A\in\PP$ and $u=\pi\rho\alpha$. If $39/10\le u\le64$ and $\|A-\Id\|_{\op}\le1/100$, then
\begin{equation}\label{T:eq:localquant}
 \Th(\alpha,A^{1/2}\FCC)-\Th(\alpha,\FCC)
 \ge\frac{u}{50}e^{-u}\|A-\Id\|_{\frob}^2.
\end{equation}
\end{proposition}
\begin{proof}
By Lemma~\ref{T:lem:twoshell}, it remains to bound the two scalar
coefficients. Taking $\delta=1/100$ and using $u\ge39/10$, we have
\[
 B_\delta(u)\le\frac{4+4/49+1/25}{2(99/100)}=\frac{102}{49}.
\]
We estimate the positive terms in $C_D$ on the four intervals of
Table~\ref{T:tab:diagonal-bounds}. Each lower bound exceeds
$102/49+1/50$.
\begin{table}[!htbp]
\centering
\caption{Rational bounds for the positive diagonal terms in the bounded-scale FCC estimate.}\label{T:tab:diagonal-bounds}
\begin{tabular}{@{}ll@{}}
\toprule
Range of $u$ & Lower bound for $\frac u2e^{-u/100}+4ue^{-51u/50}$\\
\midrule
$[3.9,4]$ & $\frac{39}{20}\frac{24}{25}+\frac{78/5}{60}$\\[2pt]
$[4,4.2]$ & $2\frac{479}{500}+\frac{16}{75}$\\[2pt]
$[4.2,4.4]$ & $\frac{21}{10}\frac{239}{250}+\frac{84/5}{100}$\\[2pt]
$[4.4,64]$ & $\frac{11}{5}\frac{239}{250}$\\
\bottomrule
\end{tabular}
\end{table}

For the first three rows use $e^{-x}\ge1-x$ and, respectively,
\[
 e^{4.08}<60,\qquad e^{4.284}<75,\qquad e^{4.488}<100.
\]
For the last interval, $u e^{-u/100}$ increases on $[4.4,64]$.
We may therefore use its left endpoint and discard the second positive
term. These four cases give $C_D\ge1/50$ throughout the required
range.

To bound the off-diagonal coefficient, use the increase of
$u e^{-u/100}$ on $[3.9,64]$. This gives
\[
 C_O\ge2\frac{39}{10}\left(1-\frac{39}{1000}\right)
          -2\frac{102}{49}>3.
\]
Combining the two coefficient bounds with
Lemma~\ref{T:lem:twoshell}, we obtain
$u(C_DD+C_OO)\ge(u/50)(D+2O)$. This proves the proposition.
\end{proof}

\subsection{A fixed FCC neighbourhood at large scales}
We now turn to the unbounded scale range. A fixed-order Taylor
estimate is insufficient here because $u\|T\|$ need not remain
bounded. We overcome this difficulty by the following estimate for
six exponentials, which does not require small exponents.

\begin{lemma}[Six-point moment bound]\label{T:lem:mgf}
Let $x_1,\ldots,x_6\in\R$ have mean zero. Define $\sigma\ge0$ and $M$ by
\[
 \sigma^2=\frac16\sum_{j=1}^6x_j^2,
 \qquad M=\frac16\sum_{j=1}^6e^{x_j}.
\]
Then
\begin{equation}\label{T:eq:mgf}
 \log M\ge\frac1{13}\min(\sigma^2,\sigma).
\end{equation}
\end{lemma}
\begin{proof}
If all $x_j$ are zero, the conclusion is immediate. Otherwise,
let $P$ be the sum of the positive $x_j$. At most five entries are
positive, and the zero-sum condition gives
\[
 \sum_{x_j<0}x_j^2\le P^2
 \le5\sum_{x_j>0}x_j^2,
\]
Hence $\sum_{x_j>0}x_j^2\ge\sigma^2$. Apply $e^x\ge1+x$ when
$x$ is negative and $e^x\ge1+x+x^2/2$ when $x$ is positive.
After summation, the linear terms cancel, and we obtain
\begin{equation}\label{T:eq:mgfquad}
 M\ge1+\frac{\sigma^2}{12}.
\end{equation}
Consequently, for $\sigma\le1$,
\[
 \log M\ge\frac{\sigma^2}{12+\sigma^2}\ge\frac{\sigma^2}{13}.
\]
When $1\le\sigma\le8$, this lower bound is at least $\sigma/13$,
as follows from $(\sigma-1)(\sigma-12)\le0$.

To treat the remaining range, let $m=\max x_j$. The zero-sum
condition places every $x_j$ in $[-5m,m]$. Averaging
$(m-x_j)(x_j+5m)\ge0$ gives $\sigma^2\le5m^2$. It follows that
\[
 \log M\ge m-\log6\ge\frac{\sigma}{\sqrt5}-\log6.
\]
For $\sigma\ge8$, use $1/\sqrt5>4/9$ and $\log6<2$ to obtain
\[
 \log M>\frac{4\sigma}{9}-2\ge\frac{\sigma}{13}.
\]
The three ranges cover all possible variances. This completes the proof.
\end{proof}

\begin{proposition}[A one-percent FCC ball, unbounded scales]\label{T:prop:smallballlarge}
Let $A\in\PP$ and $u=\pi\rho\alpha$. Inequality \eqref{T:eq:localquant} holds when $u\ge64$ and $\|A-\Id\|_{\op}\le1/100$.
\end{proposition}
\begin{proof}
We separate the scalar and trace-free parts of the strain. Write
$A-\Id=T=\eta\Id+S$ as in Lemma~\ref{T:lem:det}. Then
\begin{equation}\label{T:eq:eta1percent}
 \eta\le\frac15\|S\|_{\frob}^2,
 \qquad \|S\|_{\frob}<\frac1{50}.
\end{equation}
The contribution of the first shell takes the exact form
\begin{equation}\label{T:eq:firstmgf}
 12e^{-u}e^{-u\eta}M,
 \qquad
 M=\frac16\sum_{k,\pm}
 \exp\left(\frac u2S_{kk}\pm uS_{ij}\right),
\end{equation}
where $k$ runs through $1,2,3$. For each $k$, the indices $i<j$
are the two complementary indices, so that
\[
 (k,i,j)=(1,2,3),\quad (2,1,3),\quad (3,1,2).
\]
The notation $\sum_{k,\pm}$ means that both signs are included for each
of these three triples. Thus the sum contains exactly six terms;
there is no additional independent summation over $i$ or $j$.
Explicitly,
\[
\begin{aligned}
 M=\frac16\Bigl[{}
 &e^{uS_{11}/2+uS_{23}}+e^{uS_{11}/2-uS_{23}}\\
 &+e^{uS_{22}/2+uS_{13}}+e^{uS_{22}/2-uS_{13}}\\
 &+e^{uS_{33}/2+uS_{12}}+e^{uS_{33}/2-uS_{12}}
 \Bigr].
\end{aligned}
\]
Thus $M$ is the mean of the six displayed exponentials. Their
exponents sum to zero: $\tr S=0$, and the two off-diagonal terms
cancel for each $k$. The corresponding variance is
\begin{equation}\label{T:eq:variance}
 \sigma^2=\frac{u^2}{12}\bigl(D(S)+4O(S)\bigr)
 \ge\frac{u^2}{12}\|S\|_{\frob}^2.
\end{equation}
The key estimate is $\log M\ge2u\eta$. We prove it in two cases.
If $\sigma\le1$, Lemma~\ref{T:lem:mgf} gives
\[
 \log M\ge\frac{u^2}{156}\|S\|_{\frob}^2
 \ge\frac25u\|S\|_{\frob}^2\ge2u\eta,
\]
since $64/156>2/5$. If $\sigma\ge1$, then
\[
 \log M\ge\frac{u\|S\|_{\frob}}{13\sqrt{12}}
 >\frac{u\|S\|_{\frob}}{52}
 \ge\frac25u\|S\|_{\frob}^2\ge2u\eta,
\]
where \eqref{T:eq:eta1percent} was used in the penultimate inequality.

By \eqref{T:eq:firstmgf}, this proves that the first-shell increase
is at least $12e^{-u}(e^{u\eta}-1)\ge12u\eta e^{-u}$.
For the higher shells, the linear convexity bound and
\eqref{T:eq:isotropic} show that the total possible decrease is at
most
\[
 u\eta e^{-u}\sum_{n\ge2}nN_nq^{n-1}.
\]
The coefficient satisfies
\begin{equation}\label{T:eq:highertail3}
 \sum_{n\ge2}nN_nq^{n-1}
 \le30\left(\frac{1+q}{(1-q)^3}-1\right)<3
 \qquad(q\le1/49).
\end{equation}
Subtracting this possible loss leaves an increase of at least
$9u\eta e^{-u}$. Finally, the lower trace estimate gives
\[
 \eta\ge\frac{\|T\|_{\frob}^2}{6(1+1/100)}
 \ge\frac17\|T\|_{\frob}^2,
\]
which yields a bound stronger than \eqref{T:eq:localquant} and
completes the proof.
\end{proof}

We now prove Theorem~\ref{T:thm:local}.

\begin{proof}
Combining Propositions~\ref{T:prop:smallballbounded} and
\ref{T:prop:smallballlarge} proves part (i).

For part (ii), we again separate bounded and unbounded scales.
Suppose first that $7.9\le u\le64$, and take $\delta=1/25$.
Then $B_\delta(u)<9/4$. Since $u e^{-u/25}$ is unimodal, its
minimum is attained at one of the endpoints. The estimates
\[
 e^{0.316}<\frac{11}{8},\qquad e^{2.56}<13
\]
imply $u e^{-u/25}>64/13$ throughout. It follows that
\[
 C_D>\frac{32}{13}-\frac94>\frac1{50},
 \qquad C_O>\frac{128}{13}-\frac92>\frac1{25}.
\]
Lemma~\ref{T:lem:twoshell} now gives \eqref{T:eq:localquant} on
this interval. The condition $\alpha\ge2$ implies $u>7.9$, as needed.

It remains to treat $u\ge64$. By Lemma~\ref{T:lem:det},
\[
 \eta\le\frac{25}{141}\|S\|_{\frob}^2<\frac9{50}\|S\|_{\frob}^2,
 \qquad\|S\|_{\frob}<\frac7{100}.
\]
We apply the argument of Proposition~\ref{T:prop:smallballlarge},
now comparing the logarithmic moment with $3u\eta/2$ rather than
$2u\eta$. For $\sigma\le1$, the required inequality follows from
$64/156>27/100$; for $\sigma\ge1$, it follows from
$1/52>(27/100)(7/100)$. Thus $\log M\ge3u\eta/2$.
The first-shell increase is at least $6u\eta e^{-u}$, while
\eqref{T:eq:highertail3} bounds the higher-shell loss by
$3u\eta e^{-u}$. Their difference is therefore at least
\[
 3u\eta e^{-u}\ge\frac{3u}{6(1+1/25)}e^{-u}\|T\|_{\frob}^2,
\]
which is again stronger than \eqref{T:eq:localquant}. The positive
quadratic bound gives the stated equality case.
\end{proof}

The same two-shell estimate allows the radius to depend on a bounded
scale interval. We record this consequence of
Lemma~\ref{T:lem:twoshell} for use in the exterior covering.

\begin{proposition}[Parameter-dependent neighbourhood test]\label{T:prop:flexible}
Let $0\le\delta<1$ and $39/10\le u_-\le u_+$. Put
\begin{align}\label{T:eq:flex}
 M_\delta&=\min\{u_-e^{-\delta u_-},u_+e^{-\delta u_+}\},\notag\\
 b_\delta&=\frac{4+4e^{-u_-}+1/25}{2(1-\delta)}.
\end{align}
If
\begin{equation}\label{T:eq:flexconditions}
 \frac12M_\delta+4u_+e^{-(1+2\delta)u_+}-b_\delta>0,
 \qquad 2M_\delta-2b_\delta>0,
\end{equation}
then, with $\alpha=u/(\pi\rho)$, one has
$\Th(\alpha,A^{1/2}\FCC)>\Th(\alpha,\FCC)$ for every $u\in[u_-,u_+]$
and $A\in\PP\setminus\{\Id\}$ satisfying $\|A-\Id\|_{\op}\le\delta$.
\end{proposition}
\begin{proof}
By unimodality, the minimum of $u e^{-\delta u}$ is the endpoint
minimum in \eqref{T:eq:flex}. Moreover, $u e^{-(1+2\delta)u}$
is decreasing for $u\ge39/10$, and $B_\delta(u)$ is decreasing.
Thus \eqref{T:eq:flexconditions} gives positive lower bounds for
both coefficients in \eqref{T:eq:twoshell}. Since $A\ne\Id$, at
least one of $D(T)$ and $O(T)$ is positive. The strict comparison
follows.
\end{proof}

Appendix~\ref{app:numerical} gives the finite applications of this
criterion. The proposition itself holds on the full parameter range
stated above.

\subsection{BCC and the self-dual endpoint}\label{T:sec:bcc-endpoint}
We finally consider BCC near the self-dual scale. To apply the local
FCC estimate by duality, we first compare the reference values as the
scale increases from $1$. Set
\begin{equation}\label{T:eq:gap}
 \Delta(\alpha)=\Th(\alpha,\BCC)-\Th(\alpha,\FCC).
\end{equation}
By \eqref{T:eq:endpointtie}, $\Delta(1)=0$.

\begin{lemma}[Endpoint derivative bound]\label{app:lem:endpoint-derivative}
For $1\le\alpha\le101/100$, the function
$\Delta(\alpha)=\Th(\alpha,\BCC)-\Th(\alpha,\FCC)$ satisfies
$\Delta'(\alpha)>1/200$.
\end{lemma}
\begin{proof}[Computer-assisted proof]
Differentiate the theta sums term by term and separate the retained
shells from the negative tail. This gives the lower sums in
\eqref{T:eq:gap-finite-rational}. Their outward rational evaluation
on the hundred closed intervals \eqref{T:eq:gapcells}, including
the complete negative tail, gives $\mathscr D_j>5119/10^6>1/200$.
Appendix~\ref{app:theta-local-data} contains the derivation and
recorded evaluations. Since the intervals cover $[1,101/100]$, the
inequality holds throughout the required range.
\end{proof}

\begin{lemma}[Endpoint gap]\label{T:lem:gap}
For $1\le\alpha\le101/100$,
\begin{equation}\label{T:eq:gapderivative}
 \Delta'(\alpha)>\frac1{200},
 \qquad
 \Delta(\alpha)\ge\frac{\alpha-1}{200}.
\end{equation}
The second inequality is strict when $\alpha>1$.
\end{lemma}
\begin{proof}
By Lemma~\ref{app:lem:endpoint-derivative},
$\Delta'(\alpha)>1/200$ on the whole interval. Since $\Delta(1)=0$,
integration from $1$ to $\alpha$ gives the stated gap, with strict
inequality for $\alpha>1$.
\end{proof}

We can now combine the reference gap with the local FCC estimate.
Poisson summation gives the following comparison throughout a
neighbourhood of BCC.

We finally prove Proposition~\ref{T:prop:bcclocal}.

\begin{proof}
Choose the lattice with Gram matrix $Y$ in the form
$A_B(Y)^{1/2}\BCC$. Its dual is $A_B(Y)^{-1/2}\FCC$.
At the reciprocal scale, we have
\[
 \pi\rho\alpha^{-1}>\frac{79/20}{101/100}>\frac{39}{10}.
\]
Theorem~\ref{T:thm:local}(i), applied to $C_B$, gives
\[
 \Th(\alpha^{-1},L^*)-\Th(\alpha^{-1},\FCC)
 \ge \frac{u_*e^{-u_*}}{50}\|C_B-I\|_{\frob}^2.
\]
By Poisson summation, multiplying the left-hand side by
$\alpha^{-3/2}$ gives $\Th(\alpha,Y)-\Th(\alpha,\BCC)$.
Adding the gap in Lemma~\ref{T:lem:gap} proves
\eqref{T:eq:bcc-quantitative}, hence \eqref{T:eq:bcccomparison}.
Both terms in the lower bound are nonnegative. They vanish together
only when $\alpha=1$ and $C_B=I$, or equivalently $Y=\YB$.
Poisson summation gives equality at that point, completing the proof.
\end{proof}

This estimate distinguishes the local and global roles of BCC.
It remains a strict local minimum near the endpoint, but its value is
larger than the FCC value for scales just above $1$. The global
comparison near the self-dual scale must therefore treat both
neighbourhoods.

\section{Global comparison for the theta energy}\label{sec:theta-global}

In this section we complete the global theta comparison stated in the introduction. After the cusp and the local FCC--BCC regions have been removed, only a compact family of exterior lattice shapes remains. The proof has three steps: first, a convex continuation lemma reduces the unbounded scale range to a comparison at one scale; second, finite Gaussian sums give lower bounds on whole parameter boxes; third, the resulting boxes are assembled by a finite covering argument. No search for critical points is needed in the exterior region.

\subsection{Convex continuation in the scale}\label{T:sec:continuation}
We first treat the unbounded scale variable. Recall that
$\rho=2^{1/3}$ is the squared minimum of FCC. Subtracting the common
zero-vector contribution and multiplying by $e^{\pi\rho\alpha}>0$
does not change the sign of an energy difference. By
\eqref{T:eq:uCF}, the normalized FCC value is
\[
 C_{\FCC}(\alpha)
 =e^{\pi\rho\alpha}\bigl(\Th(\alpha,\FCC)-1\bigr)
 =12+\sum_{n\ge2}N_n e^{-\pi\rho\alpha(n-1)},
\]
where $N_n$ denotes the number of FCC vectors of squared length $n\rho$.
In particular,
\[
 C_{\FCC}'(\alpha)
 =-\pi\rho\sum_{n\ge2}(n-1)N_n
       e^{-\pi\rho\alpha(n-1)}<0,
 \qquad
 \lim_{\alpha\to\infty}C_{\FCC}(\alpha)=12.
\]
Since $C_{\FCC}$ decreases, it is enough to construct a finite
competitor lower bound which exceeds the reference at one scale and
increases thereafter. The following construction gives such a criterion.

Let $\xi=(\ell_1,\ell_2,a,b,c)$ denote the logarithmic Iwasawa
coordinates in \eqref{T:eq:iwasawa}--\eqref{T:eq:di}, where
$\ell_j=\log y_j$ for $j=1,2$. Write
\[
 Q_\xi(z)=z^TY(\xi)z,
 \qquad z\in\Z^3.
\]
Here $Q_\xi(z)$ is the squared length of the lattice vector indexed by
the integer column vector $z$. The shape $\xi$ is held fixed whenever
we differentiate with respect to $\alpha$.

Let $\mathcal S\subset\Z^3\setminus\{0\}$ be a finite symmetric set,
meaning that $\mathcal S=-\mathcal S$. Define
\begin{equation}\label{T:eq:normalizedfinite}
 P^{\Theta}_{\mathcal S}(\xi,\alpha)
 :=\sum_{z\in\mathcal S}
       \exp\bigl(\pi\alpha(\rho-Q_\xi(z))\bigr).
\end{equation}
We use this finite sum only as a lower bound for the normalized theta
function. Since every omitted term is positive, we have
\begin{equation}\label{T:eq:finitebelow}
 e^{\pi\rho\alpha}\bigl(\Th(\alpha,Y(\xi))-1\bigr)
 \ge P^{\Theta}_{\mathcal S}(\xi,\alpha).
\end{equation}
Thus the desired comparison follows once
$P^{\Theta}_{\mathcal S}(\xi,\alpha)>C_{\FCC}(\alpha)$.
The competitor tail is omitted only from a lower bound. In contrast,
the complete reference tail must be included when bounding FCC from
above.

A term in \eqref{T:eq:normalizedfinite} increases, is constant or
decreases according as $Q_\xi(z)<\rho$, $Q_\xi(z)=\rho$ or
$Q_\xi(z)>\rho$. The whole sum need not be increasing. Nevertheless,
its convexity shows that a positive derivative at one scale is
sufficient, as we prove next.

\begin{lemma}[A comparison test]\label{T:lem:continuation}
Fix a shape $\xi$, a scale $\alpha_0\ge1$, and a finite symmetric set
$\mathcal S\subset\Z^3\setminus\{0\}$. If
\begin{equation}\label{T:eq:continuationconditions}
 \begin{aligned}
 P^{\Theta}_{\mathcal S}(\xi,\alpha_0)
     &>C_{\FCC}(\alpha_0),\\
 \partial_\alpha P^{\Theta}_{\mathcal S}(\xi,\alpha_0)
     &>0,
 \end{aligned}
\end{equation}
then
\[
 \Th(\alpha,Y(\xi))>\Th(\alpha,\FCC)
 \qquad\text{for every }\alpha\ge\alpha_0.
\]
The conclusion holds throughout a closed shape box if both inequalities
in \eqref{T:eq:continuationconditions} hold throughout that box, using
one fixed set $\mathcal S$ on the box.
\end{lemma}
\begin{proof}
Termwise differentiation of the finite sum gives
\begin{align*}
 \partial_\alpha P^{\Theta}_{\mathcal S}(\xi,\alpha)
 &=\pi\sum_{z\in\mathcal S}(\rho-Q_\xi(z))
       e^{\pi\alpha(\rho-Q_\xi(z))},\\
 \partial_\alpha^2 P^{\Theta}_{\mathcal S}(\xi,\alpha)
 &=\pi^2\sum_{z\in\mathcal S}(\rho-Q_\xi(z))^2
       e^{\pi\alpha(\rho-Q_\xi(z))}\ge0.
\end{align*}
It follows that $P^{\Theta}_{\mathcal S}(\xi,\cdot)$ is convex.
Its derivative is therefore nondecreasing, and the second inequality
in \eqref{T:eq:continuationconditions} makes the finite sum strictly
increasing on $[\alpha_0,\infty)$. Combining this fact with the
decrease of $C_{\FCC}$, for $\alpha\ge\alpha_0$ we obtain
\[
 \begin{aligned}
 e^{\pi\rho\alpha}\bigl(\Th(\alpha,Y(\xi))-1\bigr)
 &\ge P^{\Theta}_{\mathcal S}(\xi,\alpha)\\
 &\ge P^{\Theta}_{\mathcal S}(\xi,\alpha_0)\\
 &>C_{\FCC}(\alpha_0)\\
 &\ge C_{\FCC}(\alpha).
 \end{aligned}
\]
Multiplying by $e^{-\pi\rho\alpha}$ and adding back the common
zero-vector term proves the first assertion. Applying the argument at
each fixed $\xi$ proves the whole-box assertion. Notice that only the
finite positive sum is required to be convex, not the difference of
the two energies.
\end{proof}

The proof also gives a uniform quantitative margin. Suppose that
$\delta_{\mathcal B}$ and $d_{\mathcal B}$ are positive constants
such that, on a closed shape box $\mathcal B$,
\[
 \begin{aligned}
 P^{\Theta}_{\mathcal S}(\xi,\alpha_0)-C_{\FCC}(\alpha_0)
     &\ge\delta_{\mathcal B},\\
 \partial_\alpha P^{\Theta}_{\mathcal S}(\xi,\alpha_0)
     &\ge d_{\mathcal B}
 \end{aligned}
 \qquad(\xi\in\mathcal B).
\]
By the supporting-line inequality for a convex function,
\[
 P^{\Theta}_{\mathcal S}(\xi,\alpha)
 \ge P^{\Theta}_{\mathcal S}(\xi,\alpha_0)
       +d_{\mathcal B}(\alpha-\alpha_0).
\]
Combining this with \eqref{T:eq:finitebelow} and the decrease of
$C_{\FCC}$ yields
\[
 \Th(\alpha,Y(\xi))-\Th(\alpha,\FCC)
 \ge e^{-\pi\rho\alpha}
      \bigl(\delta_{\mathcal B}
             +d_{\mathcal B}(\alpha-\alpha_0)\bigr)>0
\]
for every $\xi\in\mathcal B$ and every $\alpha\ge\alpha_0$.

We take $\alpha_0=5$. The fixed FCC neighbourhood is already
controlled for every $\alpha\ge5$. On the complement, a strict value
comparison and a positive derivative at $5$ suffice by the lemma.
We derive the whole-box bounds in Section~\ref{T:sec:finite}; their
evaluation is recorded in Appendix~\ref{app:numerical}. The cutoff
is fixed independently of the shape, so no interchange of a limit
and minimization is used.

\subsection{Finite Gaussian bounds on parameter boxes}\label{T:sec:finite}

It remains to establish the finite comparisons outside the local
neighbourhoods. Following \cite[Section~4, pp.~16--17]{SS}, we
bound a finite positive sum and its variation on each parameter box.
For the Gaussian energy, we include the scale among the variables and
use both a mean-value bound and a one-sided second-order estimate.

We first specify the coordinates and the finite function to be
estimated. Write
\[
 x=(\xi,\alpha)=(\ell_1,\ell_2,a,b,c,\alpha),
 \qquad
 \mathcal B=
 I_{\ell_1}\times I_{\ell_2}\times I_a\times I_b\times I_c\times I_\alpha,
\]
where each $I_j$ is a closed bounded interval and
$I_\alpha\subset[1,\infty)$. Intervals of zero length are allowed; in
particular, $I_\alpha=\{5\}$ describes a five-dimensional shape box at
one fixed scale. The matrix $Y(\xi)$ is the determinant-one matrix in
\eqref{T:eq:iwasawa}--\eqref{T:eq:di}.

For $1\le\alpha\le5$, our aim is a strictly positive lower bound on
\[
 e^{\pi\rho\alpha}
 \bigl(\Th(\alpha,Y(\xi))-\Th(\alpha,\FCC)\bigr)
 \qquad\text{throughout }\mathcal B.
\]
At $\alpha=5$, the same value estimate is used together with a
separate lower bound for
$\partial_\alpha P^{\Theta}_{\mathcal S}(\xi,5)$.
Lemma~\ref{T:lem:continuation} then extends the comparison to every
$\alpha\ge5$.

The estimates below concern a single closed box. Their assembly into
a covering is carried out in Section~\ref{T:sec:assembly}, and the
arithmetic and covering checks are given in
Appendices~\ref{T:sec:arithmetic} and~\ref{T:sec:certificates}.
Thus the derivatives here control variation within the box; they do
not serve to locate critical points.

Recall the normalization $\rho=2^{1/3}$ and the finite competitor sum
$P^{\Theta}_{\mathcal S}$ in \eqref{T:eq:normalizedfinite}.
We use index sets drawn from the cube
\begin{equation}\label{T:eq:cube}
 \mathcal S_2=
 \{(m,n,k)\in\Z^3\setminus\{0\}:|m|,|n|,|k|\le2\}.
\end{equation}
There are $5^3-1=124$ vectors, forming $62$ opposite pairs
$\{z,-z\}$. Since $Q_\xi(z)=Q_\xi(-z)$, we may sum over one
representative of each pair and multiply by two.

Every symmetric subset $\mathcal S\subset\mathcal S_2$ gives a
valid competitor lower bound. We fix this set throughout the box so
that the centre value and derivative estimates apply to the same
function. For continuation, the set also remains fixed as $\alpha$
varies. Different shape boxes may use different sets; the selection
rule is recorded in Appendix~\ref{T:sec:arithmetic}.

On the reference side, the finite FCC sum and its tail satisfy
\[
 \begin{aligned}
 C_{\FCC}(\alpha)
 &=e^{\pi\rho\alpha}\bigl(\Th(\alpha,\FCC)-1\bigr),\\
 R_{12}(\alpha)
 &=\sum_{n=1}^{12}N_n e^{-(n-1)\pi\rho\alpha},\\
 0<C_{\FCC}(\alpha)-R_{12}(\alpha)&<\epsref,
 \qquad \epsref=2\cdot10^{-18},
 \end{aligned}
\]
by \eqref{T:eq:reftail}. Here $N_n$ is the number of FCC vectors of
squared length $n\rho$. Define the finite comparison function
\begin{equation}\label{T:eq:fbox}
 f_{\mathcal S}(\xi,\alpha)
 :=P^{\Theta}_{\mathcal S}(\xi,\alpha)-R_{12}(\alpha).
\end{equation}
Putting $u=\pi\rho\alpha$, positivity of the omitted competitor sum and
the upper bound for the reference tail give
\begin{equation}\label{T:eq:fboxlower}
 e^u\bigl(\Th(\alpha,Y(\xi))-\Th(\alpha,\FCC)\bigr)
 \ge f_{\mathcal S}(\xi,\alpha)-\epsref,
 \qquad \epsref=2\cdot10^{-18}.
\end{equation}
Indeed, subtract $C_{\FCC}$ from the full normalized competitor.
The competitor is at least $P^{\Theta}_{\mathcal S}$, while
$C_{\FCC}$ is at most $R_{12}+\epsref$. This gives precisely the
one-sided bound above.

Consequently, if a rigorously computed number $L_{\mathcal B}$ satisfies
\[
 f_{\mathcal S}(x)\ge L_{\mathcal B}>\epsref
 \qquad(x\in\mathcal B),
\]
then every point of the box satisfies
\[
 \Th(\alpha,Y(\xi))-\Th(\alpha,\FCC)
 \ge e^{-\pi\rho\alpha}(L_{\mathcal B}-\epsref)>0.
\]
We have therefore reduced the comparison on a whole box to a finite
lower-bound inequality. The constant $\epsref$ accounts only for
the reference tail; all rounding errors are enclosed separately.
Positivity justifies omitting the competitor tail.

To obtain the required box bound, we compute the derivatives
explicitly. For an integer column vector $z=(m,n,k)^T$, set
\begin{equation}\label{T:eq:pqt}
 p=m+an+bk,\qquad r=n+ck,\qquad
 T_1=d_1p^2,\quad T_2=d_2r^2,\quad T_3=d_3k^2.
\end{equation}
Thus $Q_\xi(z)=T_1+T_2+T_3$. The letters $p,r,T_i$ in these derivative formulas denote scalar quantities
associated with the fixed integer
vector $z$; in particular, $r$ is not a box radius. Define
$h_z(x)=\pi\alpha(\rho-Q_\xi(z))$. Differentiating the Iwasawa
formula gives
\begin{align}\label{T:eq:hgrad}
 \partial_{\ell_1}h_z
   &=-\frac{\pi\alpha}{3}(2T_1-T_2-T_3),&
 \partial_{\ell_2}h_z
   &=-\frac{\pi\alpha}{3}(T_1+T_2-2T_3),\notag\\
 \partial_a h_z&=-2\pi\alpha d_1pn,&
 \partial_b h_z&=-2\pi\alpha d_1pk,\\
 \partial_c h_z&=-2\pi\alpha d_2rk,&
 \partial_\alpha h_z&=\pi(\rho-Q_\xi(z)).\notag
\end{align}
For $\partial_j$ denoting differentiation with respect to the $j$th
coordinate of $x=(\ell_1,\ell_2,a,b,c,\alpha)$, we therefore have
\[
 \partial_j P^{\Theta}_{\mathcal S}(x)
 =\sum_{z\in\mathcal S}e^{h_z(x)}\partial_jh_z(x).
\]
The reference depends only on $\alpha$, and
\begin{align}\label{T:eq:refderivatives}
 R_{12}'(\alpha)
   &=-\pi\rho\sum_{n=1}^{12}(n-1)N_n e^{-(n-1)u},\notag\\
 R_{12}''(\alpha)
   &=(\pi\rho)^2\sum_{n=1}^{12}(n-1)^2N_n e^{-(n-1)u}.
\end{align}
Consequently,
$\partial_jf_{\mathcal S}=\partial_jP^{\Theta}_{\mathcal S}$
for $1\le j\le5$, whereas
$\partial_\alpha f_{\mathcal S}
=\partial_\alpha P^{\Theta}_{\mathcal S}-R_{12}'$.
These finite expressions provide the derivative enclosures directly,
without numerical differentiation.

We next derive three lower bounds for the same finite comparison
function. Let $\mathcal B=\prod_{j=1}^6[l_j,u_j]$ and choose a
computed centre $x_0\in\mathcal B$. For each coordinate, take a
nonnegative radius $r_j$ satisfying
\[
 r_j\ge\max\{x_{0,j}-l_j,\ u_j-x_{0,j}\}.
\]
Then $|x_j-x_{0,j}|\le r_j$ for every $x\in\mathcal B$. This
convention remains valid when the computed centre is not the exact
arithmetic midpoint. We also write these radii as
$r_{\ell_1},r_{\ell_2},r_a,r_b,r_c,r_\alpha$.

The first bound comes from direct interval evaluation of
$f_{\mathcal S}$ on $\mathcal B$. If its lower endpoint is
$L_{\rm int}$, then $f_{\mathcal S}(x)\ge L_{\rm int}$ on the
whole box. Repeated occurrences of the same variable may make this
bound too coarse. We therefore also use the following two derivative
estimates.

For the second bound, apply the fundamental theorem of calculus along
the segment from $x_0$ to $x$. It gives
\begin{equation}\label{T:eq:meanvalue}
 f_{\mathcal S}(x)\ge f_{\mathcal S}(x_0)
 -\sum_{j=1}^6r_j\sup_{\mathcal B}|\partial_jf_{\mathcal S}|.
\end{equation}
In a numerical evaluation, take an enclosed lower value $f_0^-$ at
$x_0$ and upper bounds $M_j\ge\sup_{\mathcal B}|\partial_jf_{\mathcal S}|$.
Then
\[
 L_{\rm mv}:=f_0^--\sum_{j=1}^6r_jM_j
\]
is a lower bound for $f_{\mathcal S}$ throughout the box. The subtraction controls all
variation away from the centre; positivity of the centre value alone
would not give a whole-box comparison.

For the third bound, we retain the first derivatives at the centre
and estimate only the possible negative part of the second variation.
The following lemma makes this one-sided argument precise.
\begin{lemma}[One-sided Taylor certificate]\label{T:lem:taylor}
Suppose $f_{\mathcal S}$ is $C^2$ on a neighbourhood of $\mathcal B$.
Let $K_{\mathcal B}\ge0$ satisfy, for every displacement $v$ with
$x_0+v\in\mathcal B$ and every $0\le\tau\le1$,
\begin{equation}\label{T:eq:secondonesided}
 \frac{d^2}{d\tau^2}f_{\mathcal S}(x_0+\tau v)
 \ge-K_{\mathcal B}.
\end{equation}
Then
\begin{equation}\label{T:eq:taylorbound}
 f_{\mathcal S}(x)\ge f_{\mathcal S}(x_0)
 -\sum_{j=1}^6r_j|\partial_jf_{\mathcal S}(x_0)|
 -\frac12K_{\mathcal B},\qquad x\in\mathcal B.
\end{equation}
\end{lemma}
\begin{proof}
Fix $x\in\mathcal B$ and set $v=x-x_0$,
$g(\tau)=f_{\mathcal S}(x_0+\tau v)$. The coordinate rectangle
contains the entire segment. Taylor's formula with integral remainder
gives
\[
 g(1)=g(0)+g'(0)+\int_0^1(1-\tau)g''(\tau)\,d\tau.
\]
The bound $|v_j|\le r_j$ makes the linear term at least
$-\sum_jr_j|\partial_jf_{\mathcal S}(x_0)|$.
The second-derivative assumption bounds the integral below by
$-K_{\mathcal B}/2$. Combining the two estimates proves the lemma.
\end{proof}
Here $v$ is the full displacement, not a unit direction. Thus
$K_{\mathcal B}$ already includes the coordinate radii; no further
squared-radius factor is to be inserted.

With outward upper bounds
$g_j\ge|\partial_jf_{\mathcal S}(x_0)|$ and a valid
$K_{\mathcal B}$, define
\[
 L_{\rm Tay}:=f_0^--\sum_{j=1}^6r_jg_j-\frac12K_{\mathcal B}.
\]
After rounding the lower-bound calculations downward, the three
independently valid bounds can be combined as
\[
 L_{\mathcal B}=\max\{L_{\rm int},L_{\rm mv},L_{\rm Tay}\}.
\]
It now suffices to check $L_{\mathcal B}>\epsref$.
The reference-tail allowance is subtracted once, after combining the
three valid lower bounds. All centre values, derivative bounds and
operations defining $L_{\mathcal B}$ require rigorous enclosures.

To complete the second-order estimate, we construct
$K_{\mathcal B}$ in Lemma~\ref{T:lem:taylor} so that
\eqref{T:eq:secondonesided} holds. In this calculation, primes mean
$d/d\tau$ along the affine path $x_0+\tau v$, not differentiation
with respect to $\alpha$ alone. Put
\[
 \begin{gathered}
 v=(v_{\ell_1},v_{\ell_2},v_a,v_b,v_c,v_\alpha),\\
 \eta_1=\frac{2v_{\ell_1}+v_{\ell_2}}3,\qquad
 \eta_2=\frac{-v_{\ell_1}+v_{\ell_2}}3,\qquad
 \eta_3=\frac{-v_{\ell_1}-2v_{\ell_2}}3.
 \end{gathered}
\]
These $\eta_i$ are constant along the path. The exact identities
\[
 d_i'=\eta_i d_i,\qquad d_i''=\eta_i^2d_i,\qquad
 p'=nv_a+kv_b,\qquad r'=kv_c,\qquad p''=r''=0
\]
give, with $Q=Q_{\xi(\tau)}(z)$,
\begin{align}\label{T:eq:qpath}
 Q'&=\eta_1T_1+\eta_2T_2+\eta_3T_3
       +2d_1pp'+2d_2rr',\notag\\
 Q''&=\eta_1^2T_1+\eta_2^2T_2+\eta_3^2T_3
       +4\eta_1d_1pp'+4\eta_2d_2rr'\\
 &\hspace{18mm}+2d_1(p')^2+2d_2(r')^2.\notag
\end{align}

Since $|v_j|\le r_j$, define the nonnegative bounds
\[
 \begin{gathered}
 \beta_1=\frac{2r_{\ell_1}+r_{\ell_2}}3,\qquad
 \beta_2=\frac{r_{\ell_1}+r_{\ell_2}}3,\qquad
 \beta_3=\frac{r_{\ell_1}+2r_{\ell_2}}3,\\
 P_1=|n|r_a+|k|r_b,\qquad R_1=|k|r_c.
 \end{gathered}
\]
Then $|\eta_i|\le\beta_i$, $|p'|\le P_1$, and $|r'|\le R_1$.
Let $P,R$ be verified nonnegative upper bounds for
$\sup_{\mathcal B}|p|$ and $\sup_{\mathcal B}|r|$, respectively.
These scalars depend on $z$ and are unrelated to
$P^{\Theta}_{\mathcal S}$ and $R_{12}$. An overbar below denotes a
verified upper bound for the indicated nonnegative quantity over the
whole box. Set
\begin{align}\label{T:eq:L1L2}
 L_{1,z}&=\beta_1\overline T_1+\beta_2\overline T_2
          +\beta_3\overline T_3
          +2\overline d_1PP_1+2\overline d_2RR_1,\notag\\
 L_{2,z}&=\beta_1^2\overline T_1+\beta_2^2\overline T_2
          +\beta_3^2\overline T_3\\
 &\quad+4\beta_1\overline d_1PP_1+4\beta_2\overline d_2RR_1
          +2\overline d_1P_1^2+2\overline d_2R_1^2.\notag
\end{align}
Taking absolute values in \eqref{T:eq:qpath}, we obtain
$|Q'|\le L_{1,z}$ and $|Q''|\le L_{2,z}$ uniformly over all
permitted displacements and all points on their segments.

Along the same path $\alpha'=v_\alpha$ and $\alpha''=0$. Hence
\[
 h_z''=-\pi(2v_\alpha Q'+\alpha Q''),\qquad
 |h_z''|\le\pi\bigl(2r_\alpha L_{1,z}
                         +\alpha_+L_{2,z}\bigr),
\]
where $\alpha_+=\sup I_\alpha$. For the exponential, the identity
\[
 (e^{h_z})''
 =e^{h_z}\bigl((h_z')^2+h_z''\bigr)
 \ge-e^{h_z}|h_z''|
\]
gives the required lower estimate. The square $(h_z')^2$ is
nonnegative and may be discarded when bounding negative curvature.
The reference term contributes
\[
 \frac{d^2}{d\tau^2}\bigl(-R_{12}(\alpha(\tau))\bigr)
 =-v_\alpha^2R_{12}''(\alpha(\tau)),
 \qquad R_{12}''\ge0.
\]
It follows that the choice
\begin{equation}\label{T:eq:Kbox}
 K_{\mathcal B}
 =\sum_{z\in\mathcal S}\overline{e^{h_z}}\,
       \pi\bigl(2r_\alpha L_{1,z}+\alpha_+L_{2,z}\bigr)
       +r_\alpha^2\sup_{\mathcal B}R_{12}''
\end{equation}
satisfies \eqref{T:eq:secondonesided}. To evaluate this majorant,
we replace $\pi$, the supremum and each positive factor by an upper
enclosure and round every operation upward.

The dependence on the radii is already explicit: $L_{1,z}$ is first
order, $L_{2,z}$ is second order, and $K_{\mathcal B}$ is second
order. At a fixed scale, $r_\alpha=0$, but
$\pi\alpha_+\overline{e^{h_z}}L_{2,z}$ remains to control shape
variation. In discarding $(h_z')^2$, we used only its nonnegative
sign. The term $h_z''$ and the negative reference curvature have both
been retained.

We now distinguish the two scale ranges. For
$I_\alpha\subset[1,5]$, the inequality
\[
 L_{\mathcal B}>\epsref
\]
gives strict comparison throughout the six-dimensional box by
\eqref{T:eq:fboxlower}. It supplies the exterior comparison in
Appendix~\ref{T:sec:certificates}; the remaining boxes are handled
by reduction, cusp exclusion or the local FCC and BCC estimates.

For a shape box $\mathcal B_\xi$ at $\alpha=5$, apply the same bounds
to $\mathcal B=\mathcal B_\xi\times\{5\}$. In addition, directly
enclose the finite sum
\[
 \partial_\alpha P^{\Theta}_{\mathcal S}(\xi,5)
 =\pi\sum_{z\in\mathcal S}(\rho-Q_\xi(z))
                  e^{5\pi(\rho-Q_\xi(z))}
\]
uniformly on $\mathcal B_\xi$, by \eqref{T:eq:hgrad}.
Its strict positivity together with $L_{\mathcal B}>\epsref$
establishes \eqref{T:eq:continuationconditions}. The argument in
Section~\ref{T:sec:continuation} then covers
$\mathcal B_\xi\times[5,\infty)$.
The derivative required here is that of $P^{\Theta}_{\mathcal S}$,
not $f_{\mathcal S}$. Indeed, $R_{12}'<0$, so positivity of
$\partial_\alpha f_{\mathcal S}$ by itself would not imply the
needed derivative inequality.

Appendix~\ref{T:sec:arithmetic} gives enclosures for the quantities
in these tests. If a sufficient inequality cannot be established on a
box, the box is subdivided; this does not imply a negative energy
difference. Appendix~\ref{T:sec:certificates} verifies that every
subdivision retains both children and that the final boxes cover the
required set. We combine these comparisons with the local and cusp
estimates in Section~\ref{T:sec:assembly}.

For later reference, we summarize the half-line criterion entirely in
terms of finite sums. By \eqref{T:eq:reftail},
\[
 C_{\FCC}(5)<R_{12}(5)+\epsref,
 \qquad
 R_{12}(5)=\sum_{n=1}^{12}N_n e^{-5\pi\rho(n-1)},
 \qquad
 \epsref=2\cdot10^{-18}.
\]
Thus the two finite tests on a shape box $\mathcal B$ are
\[
 \begin{aligned}
 \inf_{\xi\in\mathcal B}
   \bigl(P^{\Theta}_{\mathcal S}(\xi,5)-R_{12}(5)\bigr)
     &>\epsref,\\
 \inf_{\xi\in\mathcal B}
   \partial_\alpha P^{\Theta}_{\mathcal S}(\xi,5)
     &>0.
 \end{aligned}
\]
Both inequalities must hold on the whole box. The finite set
$\mathcal S$ is fixed on that box and throughout continuation in
$\alpha$, although a different box may use a different set.

The value $5$ is only a cutoff for this argument, not a transition
scale. Once the two tests hold, Lemma~\ref{T:lem:continuation}
extends the comparison to every $\alpha\ge5$. Together with the
uniform FCC neighbourhood, these finite comparisons at $5$ settle
the unbounded range without passing a limit through minimization.

\subsection{Finite coverings and the theta theorem}\label{T:sec:assembly}
It remains to combine the box comparisons. Let $\mathcal R_0$ be the
outward rectangular enclosure of $\mathcal R$ with endpoints given
in Appendix~\ref{T:sec:certificates}. It contains the compact
reduced set $\calK$. We first state the covering implication, and
then the finite comparison result used to apply it.

The preceding estimates are local to a single parameter box. To pass from these boxwise inequalities to the full compact exterior region, we use only a finite covering argument: every admissible point must lie in a box on which one of the established alternatives applies. This logical step is isolated in the following proposition.

\begin{proposition}[Local-to-global comparison by finite coverings]
\label{T:prop:computer}
Suppose each of the two reduced parameter sets considered below is
covered by finitely many closed boxes. On every box, assume either a
reduction exclusion, the cusp estimate, a valid local FCC or BCC
comparison, or the strict finite-sum inequality of
Section~\ref{T:sec:finite}. On a half-line box, the local comparison must
hold on the full half-line, or the exterior comparison at $\alpha=5$ must also
satisfy the scale-derivative condition in
Lemma~\ref{T:lem:continuation}.
Then the theta comparison holds on the whole reduced set. Equality is
possible only at FCC, and at BCC when $\alpha=1$.
\end{proposition}
\begin{proof}
Fix a point in the reduced set and a covering box containing it.
The reduction-exclusion alternative is impossible at that point.
A cusp bound or finite-sum comparison gives strict inequality; for a
half-line box, Lemma~\ref{T:lem:continuation} extends the latter
comparison to all subsequent scales. Otherwise, the local inequality
follows from Theorem~\ref{T:thm:local},
Proposition~\ref{T:prop:flexible} or
Proposition~\ref{T:prop:bcclocal}, with the stated equality cases.
These alternatives exhaust the covering. When it is constructed by
subdivision, both closed children cover their parent, so finite
induction proves coverage, including every shared face.
\end{proof}

\begin{lemma}[Finite whole-box theta comparisons]\label{hyp:theta}
The reduced portions of $\mathcal R_0\times[1,5]$ and
$\mathcal R_0\times\{5\}$ admit finite closed-box coverings of the type
in Proposition~\ref{T:prop:computer}. Every box not excluded by reduction
or the cusp bound either lies in one of the stated local regions or
satisfies the strict lower bound
$L_{\mathcal B}>\varepsilon_{\rm ref}$ from
Section~\ref{T:sec:finite}. In the second covering, every local comparison
holds for all $\alpha\ge5$, and every exterior box additionally satisfies
$\inf_{\xi\in\mathcal B_\xi}\partial_\alpha
P^\Theta_{\mathcal S}(\xi,5)>0$.
\end{lemma}
\begin{proof}[Computer-assisted proof]
We apply the finite Gaussian bounds of Section~\ref{T:sec:finite},
including the upper reference-tail bound in every comparison.
Appendix~\ref{T:sec:arithmetic} gives their interval evaluation,
and Appendix~\ref{T:sec:certificates} records the two complete
subdivisions. The recorded re-evaluation checks the relevant inequality
on every terminal box: a reduction or cusp exclusion, local
containment, or a strict finite-sum bound. Each exterior half-line box
also satisfies the scale-derivative condition.

The covering check retains both children at every split and leaves no
unresolved region. Thus the terminal boxes cover both reduced
parameter sets, as required. Appendix~\ref{app:admissibility}
describes the checking method; the detailed sources and execution
records accompany the argument separately. Neither global
minimization theorem is assumed in these finite evaluations.
\end{proof}

We can now conclude the proof of Theorem~\ref{T:thm:main}.
\begin{proof}[Proof of Theorem~\ref{T:thm:main}]
We first assume $\alpha\ge1$ and choose a reduced representative of
$L$ as in Section~\ref{sec:common-geometry}. If $d_1\le2/3$,
Proposition~\ref{T:prop:cusp} gives strict comparison with FCC.
Otherwise the representative lies in $\calK\subset\mathcal R_0$,
and we distinguish the following two scale ranges.

For $1\le\alpha\le5$, the exterior comparison follows from
Lemma~\ref{hyp:theta} and Proposition~\ref{T:prop:computer}.
Theorem~\ref{T:thm:local} and Proposition~\ref{T:prop:bcclocal}
handle the two local neighbourhoods.

For $\alpha\ge5$, use the second covering in
Lemma~\ref{hyp:theta}. Its exterior inequalities extend to the
whole half-line by Lemma~\ref{T:lem:continuation}, while
Theorem~\ref{T:thm:local}(ii) applies on the fixed FCC
neighbourhood. The two ranges meet at $\alpha=5$.

It remains to check equality and the reciprocal scales. All exterior
and cusp comparisons are strict. The local estimates allow equality
only at FCC, or at BCC when $\alpha=1$; both endpoint classes attain
it by \eqref{T:eq:endpointtie}. For $0<\alpha<1$, apply the
comparison at $1/\alpha>1$ to $L^*$ and use
\eqref{T:eq:poisson}. Since $\FCC^*=\BCC$, equality holds exactly
when $L$ is orthogonally equivalent to BCC. This
completes the proof.
\end{proof}

\section{Ewald bounds and local Epstein estimates}\label{sec:epstein-local}

We begin the Epstein part with the two local estimates that will later be used in the global covering. The first gives a transparent third-order FCC neighbourhood, while the second retains the signed cubic and quartic terms and yields a larger uniform ball. The proof is organized accordingly: we first obtain rigorous Ewald-kernel bounds and complete tails, then control the Hessian and third derivative, and finally sharpen the expansion through fifth order.

\begin{theorem}[Local comparison from third-order bounds]\label{E:thm:third-ball}
For every $3/2\le s\le8$ and $X\in\Symzero$ with
$0<\norm X\le1/500$,
\begin{equation}\label{E:eq:third-ball}
 \HH_s(Y(X))-\HH_s(\YF)>\frac{\norm X^2}{200}.
\end{equation}
\end{theorem}

The preceding theorem gives a simple and transparent FCC neighbourhood using only the Hessian and a third-order remainder. Its radius is intentionally modest. To obtain a neighbourhood large enough for the global covering, we retain the signed cubic and quartic contributions and control the fifth-order remainder. This refinement leads to the stronger local statement below.

\begin{theorem}[Local comparison from signed coefficient bounds]\label{E:thm:low-local}
On every exponent cell in \eqref{E:eq:local-cells}, the bound
\eqref{E:eq:local-C} holds at the radius specified in
Table~\ref{E:tab:low-local}, with coefficient strictly above the lower bound
in that table. In particular, for every $3/2\le s\le8$ and
$X\in\Symzero$ with
$0<\norm X\le2/25$,
\begin{equation}\label{E:eq:uniform-low-ball}
 \HH_s(Y(X))-\HH_s(\YF)>\frac{\norm X^2}{20000}.
\end{equation}
The coefficient bounds include the full fifth-order remainder.
\end{theorem}

We now establish the kernel and derivative estimates entering these theorems.

\subsection{Rational kernel bounds and complete tails}\label{E:sec:kernels}
We begin with the kernel $G$. Its Stieltjes representation will give
finite lower and upper bounds with controlled signs. For $k>0$ and
$x>0$, define
\[
 S_k(x)=\frac1{\Gamma(k)}\int_0^\infty\frac{u^{k-1}e^{-u}}{x+u}\,du.
\]
Fubini's theorem and $1/(x+u)=\int_0^\infty e^{-(x+u)v}\,dv$ imply
\begin{equation}
 S_k(x)=\int_0^\infty e^{-xv}(1+v)^{-k}\,dv=e^xG(1-k,x).
 \label{E:eq:stieltjes}
\end{equation}
Let $J_{N,k}$ be the $N\times N$ symmetric Jacobi matrix with diagonal entries $k+2j$, $0\le j\le N-1$, and adjacent off-diagonal entries $\sqrt{(j+1)(j+k)}$, $0\le j\le N-2$. Here $e_0=(1,0,\ldots,0)^T\in\R^N$, and $I$ in the following
resolvent is the $N\times N$ identity matrix. Define
\[
 L_{N,k}(x)=e_0^T(xI+J_{N,k})^{-1}e_0,\qquad L_{0,k}=0.
\]
In the scalar continued fraction, only the squared off-diagonal
entries $(j+1)(j+k)$ occur. Thus its evaluation requires no square
roots.

The gamma probability measure used below is
$d\mu_k(u):=\Gamma(k)^{-1}u^{k-1}e^{-u}\,du$ on $(0,\infty)$.
For a real number $w$ and an integer $j\ge0$, the rising factorial is
$(w)_0:=1$ and $(w)_j:=w(w+1)\cdots(w+j-1)$ for $j\ge1$.

\begin{lemma}[Finite lower and upper bounds]\label{E:lem:stieltjes}
For $x,k>0$ and $N\ge1$,
\begin{equation}
 L_{N,k}(x)\le S_k(x)\le\frac{1-kL_{N-1,k+1}(x)}x.
 \label{E:eq:stieltjes-bounds}
\end{equation}
The upper bound is the first resolvent entry of the same Jacobi matrix after its last diagonal entry is replaced by $N-1$.
\end{lemma}
\begin{proof}
To prove the lower bound, let $p$ be a polynomial of degree less than
$N$. We use the identity
\[
 \frac1{x+u}-2p(u)+(x+u)p(u)^2
 =\frac{(1-(x+u)p(u))^2}{x+u}\ge0.
\]
Integrate this inequality against the gamma probability measure
and express the polynomial in its orthonormal basis. Maximizing the
resulting quadratic functional gives
$e_0^T(xI+J_{N,k})^{-1}e_0$, which proves the lower bound.
The matrix entries follow from the Laguerre recurrence, or from
Gram--Schmidt applied to $\int u^j\,d\mu_k=(k)_j$.
More explicitly, Rodrigues' expression
$u^{1-k}e^u(d/du)^j(e^{-u}u^{j+k-1})$ gives orthogonality after
$j$ integrations by parts. Comparing the leading and next coefficients
then yields the stated diagonal and off-diagonal entries.

The gamma measures obey $u\,d\mu_k=k\,d\mu_{k+1}$, so
\[
 xS_k(x)+kS_{k+1}(x)=1.
\]
Substituting the lower bound for $S_{k+1}$ at order $N-1$ gives
the upper bound. Eliminating the successive denominators in the finite
continued fraction gives the modified-Jacobi representation. For
$N=1$, this is $1/x$; the coefficients $(j+1)(j+k)$ then verify
the identity by induction. Hence both bounds are rational functions
of $k,x$. No unquantified quadrature error enters the argument.
\end{proof}

We now transfer these bounds to the Ewald kernel. If $a<1$, apply
Lemma~\ref{E:lem:stieltjes} with $k=1-a$ and multiply by $e^{-x}$.
For $a=1$, we have $G(1,x)=e^{-x}/x$. Larger first arguments are
handled by the recurrence
\begin{equation}
 G(a+1,x)=\frac{e^{-x}+aG(a,x)}x.
 \label{E:eq:G-recurrence}
\end{equation}
This follows by integration by parts. We apply the identity with
interval arithmetic, retaining the sign of every coefficient.
Since $G>0$, any negative lower endpoint may be replaced by zero.

To control the omitted shells, we also need a bound for large second
arguments. The defining integral gives
\begin{equation}
 0<G(a,x)\le\frac{e^{-x}}{x-\max(a-1,0)},\qquad x>\max(a-1,0),
 \label{E:eq:G-large-x}
\end{equation}
Indeed, substituting $t=1+v$ in the defining integral and using
$(1+v)^{a-1}\le e^{\max(a-1,0)v}$ gives this estimate.

We next estimate the shell multiplicities. Let $N_{\mathrm F}(n)$
count FCC vectors with $|v|^2=2^{1/3}n$, and let
$N_{\mathrm B}(n)$ count BCC vectors with $|v|^2=2^{-4/3}n$.
Then the bounds
\begin{equation}
 N_{\mathrm F}(n)\le30n,\qquad N_{\mathrm B}(n)\le18n
 \label{E:eq:shell-count}
\end{equation}
follow by fixing two integer coordinates and allowing at most two
choices for the third. For FCC, $|m|^2=2n$ gives the bound
$2(2\sqrt{2n}+1)^2<30n$. For BCC, use
$2(2\sqrt n+1)^2\le18n$. Imposing parity can only reduce these
counts.

For $a\in\R$, $x>0$ and $x\ge\max(a,0)$,
\begin{equation}
 G(a,x)\le(\max(a,0)+1)\frac{e^{-x}}x.
 \label{E:eq:G-shell-bound}
\end{equation}
For nonnegative $a$, the latter estimate is \cite[(35)]{SS}; the
nonpositive range follows by monotonicity in the first argument.

For $a\le1$, the integral proves the inequality directly. For
$a>1$, use \eqref{E:eq:G-recurrence} to reduce to $0<a_0\le1$.
Since $a\le x$, every positive recurrence factor is at most one,
and there are at most $a+1$ resulting terms. This proves the bound.

For $A>0$, $N\in\Z_{\ge0}$ and $k\in\Z_{\ge0}$,
\begin{equation}
 \sum_{n>N}n^ke^{-An}
 \le
 \frac{(N+1)^ke^{-A(N+1)}}{1-e^{-A}(1+1/(N+1))^k},
 \label{E:eq:geometric-tail}
\end{equation}
whenever the denominator is positive. The ratio of consecutive terms
is bounded by the constant appearing there, so comparison with a
geometric series proves the inequality. We check the positivity of
the denominator at every application.

In particular, let $N\in\Z_{\ge0}$, $A>0$ and assume
$A(N+1)\ge\max(a,0)$. Then \eqref{E:eq:G-shell-bound} is valid on every
omitted shell, and \eqref{E:eq:shell-count} gives the reference-tail bound
\begin{equation}
 \sum_{n>N}N_L(n)G(a,An)
 \le\frac{C_L(\max(a,0)+1)}{A}
 \frac{e^{-A(N+1)}}{1-e^{-A}},
 \label{E:eq:reference-tail-G}
\end{equation}
where $C_{\mathrm F}=30$ and $C_{\mathrm B}=18$.
For the reference value, the cutoffs are $N=16$ at FCC and $N=64$
at BCC, with $A$ equal to $\pi2^{1/3}$ and $\pi2^{-4/3}$,
respectively. Derivative estimates shift the kernel order, and local
deformations may change the Gaussian rate. In each such application,
the cutoff condition is checked with the actual order and rate.

\subsection{The first three variations at FCC}\label{E:sec:derivative-first}
We now turn to the local comparison. We first compute the vanishing
linear term and the two Hessian coefficients in the exponential
coordinates \eqref{E:eq:exp-coordinates}. At FCC, these are
$Y(X):=\gFE^Te^X\gFE$ for $X\in\Symzero$.
For a trace-free symmetric matrix $A$ of Frobenius norm one, set
\[
 Y(r)=\gFE^{T}e^{rA}\gFE,\qquad
 f_s(r)=\HH_s(Y(r))-\HH_s(\YF).
\]
A vector $v$ of the physical FCC lattice contributes through
$q_v(r)=v^Te^{rA}v$, and a vector $w$ of its dual contributes through
$q_w^*(r)=w^Te^{-rA}w$. In particular,
\begin{equation}
 q_v^{(j)}(r)=v^TA^je^{rA}v,\qquad
 (q_w^*)^{(j)}(r)=w^T(-A)^je^{-rA}w.
 \label{E:eq:radial-q-derivatives}
\end{equation}
The Gaussian estimates in Section~\ref{E:sec:ewald} justify
termwise differentiation uniformly on compact exponent intervals and
fixed shape balls. Differentiating at the reference lattice, we obtain
\begin{align}
 f_s'(0)={}&-\pi\sum_{v\in\FCC\setminus\{0\}}
 G(s+1,\pi|v|^2)v^TAv\nonumber\\
 &+\pi\sum_{w\in\B\setminus\{0\}}
 G(5/2-s,\pi|w|^2)w^TAw.
 \label{E:eq:explicit-first-variation}
\end{align}
On each cubic shell, sign changes make the mixed second moments
vanish, while coordinate permutations make the diagonal moments equal.
For a shell $S$ of squared length $q$ and cardinality $N$, this gives
\[
 \sum_{v\in S}vv^T=\frac{Nq}{3}I,
 \qquad \sum_{v\in S}v^TAv=\frac{Nq}{3}\tr A=0.
\]
Each sum in \eqref{E:eq:explicit-first-variation} therefore vanishes separately:
\begin{equation}
 f_s(0)=f_s'(0)=0\qquad(s\in\R).
 \label{E:eq:stationarity-direct}
\end{equation}
This proves stationarity in every tangent direction. The local
comparison now reduces to lower bounds for the two Hessian components
and control of the higher-order terms.

For the second variation, let $a$ be a scalar kernel argument and
$q(r)$ a positive function. Two differentiations give
\begin{equation}
 \frac{d^2}{dr^2}G(a,\pi q(r))
 =\pi^2G(a+2,\pi q(r))(q'(r))^2
  -\pi G(a+1,\pi q(r))q''(r).
 \label{E:eq:second-variation-direct}
\end{equation}
Let $E_{ij}:=e_ie_j^T$ be the matrix with entry $1$ in position $(i,j)$
and zero elsewhere. At $r=0$, use the two unit directions
\[
 A_D=\frac{\diag(-1,1,0)}{\sqrt2},\qquad
 A_O=\frac{E_{12}+E_{21}}{\sqrt2}.
\]
For a physical vector $v$, define
\[
 d_D(v)=\frac{(v_2^2-v_1^2)^2}{2},\qquad
 d_O(v)=2v_1^2v_2^2,\qquad
 m(v)=\frac{v_1^2+v_2^2}{2}.
\]
The two exact Hessian coefficients are
\begin{align}
 h_\kappa(s)={}&
 \sum_{v\in\FCC\setminus\{0\}}
 \bigl[\pi^2d_\kappa(v)G(s+2,\pi|v|^2)
       -\pi m(v)G(s+1,\pi|v|^2)\bigr]\nonumber\\
 &+\sum_{w\in\B\setminus\{0\}}
 \bigl[\pi^2d_\kappa(w)G(7/2-s,\pi|w|^2)
       -\pi m(w)G(5/2-s,\pi|w|^2)\bigr],
 \quad\kappa\in\{D,O\}.
 \label{E:eq:hessian-explicit-new}
\end{align}
The invariant decomposition proved in Section~\ref{E:sec:local2} implies that,
for $A=A_D^{\mathrm{part}}+A_O^{\mathrm{part}}$ with diagonal and off-diagonal
parts,
\[
 f_s''(0)=h_D(s)\norm{A_D^{\mathrm{part}}}^2
          +h_O(s)\norm{A_O^{\mathrm{part}}}^2.
\]
Thus lower bounds for $h_D$ and $h_O$ control the full Hessian.
Neither component may be omitted.

For a closed exponent interval $J=[a,b]$, monotonicity of $G$ in its first
argument gives, simultaneously for every $s\in J$,
\begin{align}
 h_\kappa(s)\ge{}&
 \sum_{v\in\FCC\setminus\{0\}}
 \bigl[\pi^2d_\kappa(v)G(a+2,\pi|v|^2)
       -\pi m(v)G(b+1,\pi|v|^2)\bigr]\nonumber\\
 &+\sum_{w\in\B\setminus\{0\}}
 \bigl[\pi^2d_\kappa(w)G(7/2-b,\pi|w|^2)
       -\pi m(w)G(5/2-a,\pi|w|^2)\bigr].
 \label{E:eq:hessian-cell-new}
\end{align}
To bound these sums from below, we retain the finite terms, discard
the omitted positive squares and subtract an upper bound for the
omitted negative part. Shell symmetry gives
$\sum_{v\in S}m(v)=Nq/3$, reducing the latter estimate to a scalar
radial sum. The tail bounds in Section~\ref{E:sec:kernels} then
apply.

It remains to bound the Taylor remainder throughout the
neighbourhood. A third differentiation gives
\begin{align}
 \frac{d^3}{dr^3}G(a,\pi q(r))={}&
 -\pi^3G(a+3,\pi q(r))(q'(r))^3\nonumber\\
 &+3\pi^2G(a+2,\pi q(r))q'(r)q''(r)
 -\pi G(a+1,\pi q(r))q'''(r).
 \label{E:eq:third-variation-direct}
\end{align}
Put $c=\sqrt{2/3}$. For a trace-free symmetric unit matrix,
$\opnorm A\le c$. Indeed, $a_1=-a_2-a_3$ implies
$a_1^2\le2(a_2^2+a_3^2)=2(1-a_1^2)$ for each eigenvalue.
Fix $R_0>0$ and set $\eta=e^{cR_0}$. By
\eqref{E:eq:radial-q-derivatives}, for $0\le r\le R_0$ we have
\[
 q(r)\ge\eta^{-1}|v|^2,\qquad
 |q^{(j)}(r)|\le c^j\eta|v|^2\quad(j=1,2,3).
\]
The signs of $A$ and $-A$ do not affect these absolute estimates. Since $G$
is decreasing in its second argument, a uniform third-derivative majorant on
$J=[a,b]$ is
\begin{align}
 M_3(J,R_0)={}&c^3\sum_{(L,\sigma)\in\{(\FCC,b),(\B,3/2-a)\}}
 \sum_{v\in L\setminus\{0\}}
 \Bigl[(\pi\eta|v|^2)^3G(\sigma+3,\pi\eta^{-1}|v|^2)\nonumber\\
 &\hspace{18mm}+3(\pi\eta|v|^2)^2G(\sigma+2,\pi\eta^{-1}|v|^2)
 +(\pi\eta|v|^2)G(\sigma+1,\pi\eta^{-1}|v|^2)\Bigr].
 \label{E:eq:M3-new}
\end{align}
Consequently, $|f_s'''(r)|\le M_3(J,R_0)$ for every unit $A$,
every $s\in J$ and every $0\le r\le R_0$. This is a bound for the
complete series. When it is evaluated by a finite sum, the positive
majorants for all omitted shells must be added.

\begin{lemma}[A third-order local criterion]\label{E:lem:third-criterion}
Suppose $h_D(s),h_O(s)\ge m_J>0$ for every $s\in J$ and
$|f_s'''(r)|\le M_J$ for every unit tangent direction and $0\le r\le R_0$.
For $0<R\le R_0$ satisfying $m_J/2-M_JR/6>0$, one has
\begin{equation}
 \HH_s(Y(X))-\HH_s(\YF)
 \ge\left(\frac{m_J}{2}-\frac{M_JR}{6}\right)\norm X^2,
 \quad s\in J,\quad\norm X\le R.
 \label{E:eq:third-criterion}
\end{equation}
\end{lemma}
\begin{proof}
Write $X=rA$ with $\norm A=1$. By
\eqref{E:eq:stationarity-direct}, the constant and linear terms
vanish. Taylor's formula with integral remainder therefore gives
\[
 f_s(r)=\frac12 f_s''(0)r^2
       +\frac12\int_0^r(r-u)^2f_s'''(u)\,du
 \ge\frac{m_Jr^2}{2}-\frac{M_Jr^3}{6}.
\]
The estimate $r\le R$ gives the required coefficient. For $X=0$,
the assertion holds directly.
\end{proof}

\begin{lemma}[Third-order bounds]\label{app:lem:third-inputs}
For every row $J$ of Table~\ref{E:tab:third-cells}, the two Hessian
coefficients satisfy $h_D(s),h_O(s)\ge m_J$ for $s\in J$, and
$M_3(J,1/100)\le M_J$. The same row satisfies
$m_J/2-M_J/3000>1/200$.
\end{lemma}
\begin{proof}[Computer-assisted proof]
Apply the rational kernel bounds and complete tails of
Section~\ref{E:sec:kernels} to the lower Hessian sums
\eqref{E:eq:hessian-cell-new} and the remainder majorant
\eqref{E:eq:M3-new}. The thirteen closed-interval evaluations in
Appendix~\ref{app:epstein-local-data} give
Table~\ref{E:tab:third-cells}. The last inequality is then a
rational substitution; its least left-hand side is $7/1200$.
\end{proof}

\begin{table}[ht]\centering
\caption{Numerical recording table of Lemma \ref{app:lem:third-inputs}. The Hessian bounds hold at FCC;
$M_J$ holds throughout the ball of radius $1/100$. The conclusion uses
radius $1/500$.}\label{E:tab:third-cells}
\begin{tabular}{ccc}\toprule
Exponent interval&$m_J$&$M_J$\\\midrule
$[3/2,2]$&$23/1000$&$17$\\
$[2,5/2]$&$30/1000$&$19$\\
$[5/2,3]$&$39/1000$&$21$\\
$[3,7/2]$&$51/1000$&$26$\\
$[7/2,4]$&$67/1000$&$32$\\
$[4,9/2]$&$89/1000$&$41$\\
$[9/2,5]$&$122/1000$&$54$\\
$[5,11/2]$&$171/1000$&$74$\\
$[11/2,6]$&$244/1000$&$105$\\
$[6,13/2]$&$358/1000$&$154$\\
$[13/2,7]$&$537/1000$&$233$\\
$[7,15/2]$&$824/1000$&$362$\\
$[15/2,8]$&$1295/1000$&$577$\\\bottomrule
\end{tabular}
\end{table}

We now prove Theorem~\ref{E:thm:third-ball}.

\begin{proof}
By Lemma~\ref{app:lem:third-inputs}, the required Hessian bounds
and third-derivative estimates hold on each of the thirteen closed
intervals, with $R_0=1/100$. Taking $R=1/500$, the same lemma gives
$m_J/2-M_JR/6>1/200$. Lemma~\ref{E:lem:third-criterion} proves
the comparison on each interval. Their union includes every exponent
in the stated range, including the shared endpoints.
\end{proof}

The criterion also gives an explicit radius on any interval for
which the constants are known. If $M_J>0$, the choice
$0<R\le\min\{R_0,3m_J/(2M_J)\}$ gives coefficient $m_J/4$.
If $M_J=0$, every $0<R\le R_0$ gives coefficient $m_J/2$.
We next improve these radii by retaining the signed higher-order
terms. The exterior comparisons remain separate and are proved in
Sections~\ref{E:sec:exponent-certificate} and~\ref{E:sec:power-cert}.

\subsection{Signed invariant estimates and larger neighbourhoods}\label{E:sec:local2}\label{app:invariants}
To enlarge the neighbourhood, we determine the homogeneous Taylor
terms from cubic invariance. The physical FCC and BCC lattices in
\eqref{E:eq:FCC}--\eqref{E:eq:BCC} are invariant under coordinate
sign changes and permutations. For a radial lattice sum whose first
variation is justified by local normal convergence, these symmetries
make the gradient a scalar matrix. Its restriction to the trace-free
tangent space vanishes.

We use the orthonormal tangent basis specified immediately after (21)
in \cite[Section~3, pp.~13--14]{SS}:
\[
 b_1=\diag(-1,1,0)/\sqrt2,\quad
 b_2=\diag(-1,-1,2)/\sqrt6,
\]
\[
 b_3=(E_{12}+E_{21})/\sqrt2,\quad
 b_4=(E_{13}+E_{31})/\sqrt2,\quad
 b_5=(E_{23}+E_{32})/\sqrt2.
\]
Write $X=\sum x_jb_j$, so $\norm X^2=\sum x_j^2$, and put
\[
 D^2=x_1^2+x_2^2,\qquad O^2=x_3^2+x_4^2+x_5^2.
\]
Choose $g_0=\gFE$ at FCC and $g_0=\gBE$ at BCC. For
$F_s(X)=\HH_s(g_0^Te^Xg_0)$, write
$F_s(X)-F_s(0)=q_2(X;s)+q_3(X;s)+q_4(X;s)+O(\norm X^5)$,
where the terms are homogeneous in the tangent coordinates. We suppress
$s$ in the coefficient notation. By the invariant spaces in
\cite[Section~3, (21)]{SS}, the quadratic and cubic terms have the
form
\begin{align}
 q_2&=\tfrac12h_DD^2+\tfrac12h_OO^2,\\
 q_3&=c_1p_{31}+c_2p_{32}+c_3p_{33},
 \label{E:eq:low-invariants}
\end{align}
where
\begin{align*}
 p_{31}&=-3x_1^2x_2+x_2^3,\qquad p_{32}=x_3x_4x_5,\\
 p_{33}&=\sqrt3x_1(x_4^2-x_5^2)+x_2(2x_3^2-x_4^2-x_5^2).
\end{align*}
The quartic invariants are
\[
 p_{41}=D^4,\quad p_{42}=D^2O^2,\quad p_{43}=O^4,\quad
 p_{44}=x_3^4+x_4^4+x_5^4,
\]
\[
 p_{45}=3x_1^2x_3^2+2\sqrt3x_1x_2(x_4^2-x_5^2)
       +x_2^2(-x_3^2+2x_4^2+2x_5^2),
\]
and $q_4=\sum_{j=1}^5\kappa_jp_{4j}$. The basis polynomials are
those in \cite[(21)]{SS}, but the coefficients here belong to
$\HH_s$ on the stated exponent intervals, not to the height
expansion in that reference. We evaluate them in the specified
physical frame and tangent basis. Under an orthogonal change of frame,
both lattice and tangent matrices must be conjugated. Norms are
preserved, whereas the signed coefficient tables must be read in
their original coordinates.

On $D^2+O^2=1$,
\begin{equation}
 |p_{31}|\le D^3,\quad
 |p_{32}|\le\frac{O^3}{3\sqrt3},\quad
 |p_{33}|\le2DO^2,\quad |p_{45}|\le3D^2O^2.
 \label{E:eq:invariant-bounds}
\end{equation}
The first bound follows from the triple-angle formula, and the second
from the arithmetic--geometric mean inequality. For the third, view
the polynomial as a linear expression in $(x_1,x_2)$; its coefficient
vector has norm at most $2O^2$. For the fourth, the coefficients of
$x_3^2,x_4^2,x_5^2$ are quadratic forms in $(x_1,x_2)$ whose
eigenvalues lie between $-1$ and $3$. Each is therefore bounded in
absolute value by three times $D^2$.

For completeness, we verify that these polynomials span the invariant
spaces. We write them in physical coordinates and use the sign-change
and permutation symmetries. This also explains why finitely many
coefficient evaluations determine the full Taylor terms.

Write $a_i=X_{ii}$ and $b_{ij}=X_{ij}$, where $a_1+a_2+a_3=0$.
A diagonal sign change acts on the three off-diagonal entries by signs
whose product is one. An invariant monomial therefore has all three
off-diagonal exponents even or all three odd. Coordinate permutations act
simultaneously on the diagonal entries and on the opposite off-diagonal
positions.

The basis in Section~\ref{E:sec:local2} satisfies
\[
 D^2=\sum_i a_i^2,\qquad O^2=2\sum_{i<j}b_{ij}^2.
\]
Its cubic and mixed quartic polynomials have the physical-coordinate forms
\begin{align}\label{eq:physical-invariants}
 p_{31}&=\sqrt6\sum_i a_i^3,
 &p_{32}&=2\sqrt2\,b_{12}b_{13}b_{23},\notag\\
 p_{33}&=2\sqrt6\sum_{\{i,j,k\}=\{1,2,3\},\,i<j}a_kb_{ij}^2,
 &p_{45}&=3D^2O^2-12\sum_{\{i,j,k\}=\{1,2,3\},\,i<j}a_k^2b_{ij}^2.
\end{align}
Each sum on the second line has three terms, one for each unordered pair
$i<j$ and its complementary index $k$. Direct substitution of $X=\sum x_jb_j$
proves these identities, and hence invariance of the displayed polynomials.
The remaining quartics are $D^4$, $D^2O^2$, $O^4$ and
$4\sum_{i<j}b_{ij}^4$.

We count dimensions by the degree in the off-diagonal entries.
At degree two, the diagonal and off-diagonal quadratic sums give
two independent invariants. At degree three, there is one symmetric
diagonal cubic, one product of the three off-diagonal entries and one
contraction of the diagonal entries with off-diagonal squares.
The trace relation leaves one invariant of each type.

At degree four, the purely diagonal space has dimension one. For a
mixed term $b_{ij}^2$, the diagonal coefficient is symmetric in
$a_i,a_j$. Using $a_i+a_j=-a_k$, it is spanned by $\sum a_i^2$
and $a_k^2$, giving two mixed invariants. The purely off-diagonal
terms are symmetric quadratics in the three squares and give two
more. A term $b_{12}b_{13}b_{23}$ times an invariant diagonal linear
form vanishes by the trace relation. The three dimensions are
therefore $2,3,5$.

The evaluations in \eqref{E:eq:quartic-recovery} show that the five
quartics are independent; the three cubic evaluations give the
corresponding independence at degree three. Hence the recovery
formulas determine the full invariant polynomials. Their symbolic
checks are recorded in Appendix~\ref{app:elementary-revision}.

We now compute the coefficients. Fix $A\in\Symzero$ and define
$\nu_j(v)=v^TA^jv$ for $j\ge0$, so $\nu_0(v)=|v|^2$.
These are vector moments, not the Taylor polynomials $q_j$.
For a primary vector the squared length is $v^Te^{rA}v$; for a
dual vector, replace $A$ by $-A$. At $r=0$, the successive kernel
derivatives are
\begin{align}
 \frac{d^2}{dr^2}G(a,\pi q(r))={}&
 \pi^2G(a+2,\pi \nu_0)\nu_1^2-\pi G(a+1,\pi \nu_0)\nu_2,\\
 \frac{d^3}{dr^3}G(a,\pi q(r))={}&
 -\pi^3G(a+3,\pi \nu_0)\nu_1^3
 +3\pi^2G(a+2,\pi \nu_0)\nu_1\nu_2-\pi G(a+1,\pi \nu_0)\nu_3,\\
 \frac{d^4}{dr^4}G(a,\pi q(r))={}&
 \pi^4G(a+4,\pi \nu_0)\nu_1^4
 -6\pi^3G(a+3,\pi \nu_0)\nu_1^2\nu_2\nonumber\\
 &+\pi^2G(a+2,\pi \nu_0)(3\nu_2^2+4\nu_1\nu_3)
 -\pi G(a+1,\pi \nu_0)\nu_4.
 \label{E:eq:coefficient-derivatives}
\end{align}
Divide by the appropriate factorial to obtain Taylor coefficients. All moments are finite integer sums multiplied by explicitly enclosed algebraic constants.

The two Hessian coefficients are obtained in the directions $b_2$
and $b_3$. For the cubic term, evaluate at $b_2$,
$b_3+b_4+b_5$ and $b_2+b_3$; this gives
$c_1,c_2,c_1+2c_3$. For the quartic term, evaluate at $b_2$,
$b_3$, $b_3+b_4$, $b_1+b_3$ and $b_2+b_3$, and denote the
values by $z_0,\ldots,z_4$. Solving for the coefficients yields
\begin{align}
 \kappa_1&=z_0,&\kappa_3&=(z_2-2z_1)/2,&\kappa_4&=z_1-\kappa_3,\\
 \kappa_5&=(z_3-z_4)/4,&\kappa_2&=z_4-\kappa_1-\kappa_3-\kappa_4+\kappa_5.
 \label{E:eq:quartic-recovery}
\end{align}
Together with \eqref{E:eq:invariant-bounds}, these coefficient
identities control the Taylor polynomial in every unit tangent
direction, rather than only at the evaluation points.

We shall apply these formulas on the cells
\eqref{E:eq:local-cells}. Lemma~\ref{E:lem:signed-inputs} states
the uniform coefficient bounds, and
Appendix~\ref{app:epstein-local-data} records their evaluation,
including the signed terms and complete remainders.

It remains to control the fifth-order remainder. If $\tr A=0$ and
$\norm A=1$, every eigenvalue of $A$ is bounded in absolute value
by $c=\sqrt{2/3}$. Indeed, $a_1=-a_2-a_3$ implies
$a_1^2\le2(1-a_1^2)$. Thus, for $0\le r\le R$, we have
\[
 |q^{(j)}(r)|\le c^je^{cR}|v|^2,\qquad q(r)\ge e^{-cR}|v|^2.
\]
For a compact exponent cell $J$ and $R>0$, define
\[
 M_5(J,R)=\sup_{\substack{s\in J,\ A\in\Symzero,\ \norm A=1\\0\le r\le R}}
 \left|\frac{d^5}{dr^5}F_s(rA)\right|.
\]
Set $\eta=e^{cR}$ and
$(\beta_1,\ldots,\beta_5)=(1,15,25,10,1)$.
Applying the repeated chain-rule argument of
\cite[(34)--(36)]{SS} at order five, uniformly for $J$, gives
\begin{equation}
 M_5(J,R)\le c^5\sum_{L\in\{\text{primary,dual}\}}\ \sum_{v\in L\setminus\{0\}}
 \sum_{j=1}^5\beta_j(\pi\eta|v|^2)^j
 G(a_L^++j,\pi\eta^{-1}|v|^2),
 \label{E:eq:M5}
\end{equation}
where $a_L^+$ is the largest scalar kernel argument over the exponent cell.

We also account for the terms omitted in the finite evaluation;
the details are in Appendix~\ref{app:numerical}. In a Hessian
lower bound, omitted positive squares may be discarded. For unit
Hessian directions, shell symmetry gives
$\sum\nu_2=N_L(n)|v|^2/3$, so \eqref{E:eq:shell-count} and
\eqref{E:eq:G-shell-bound} bound the omitted negative part.
Cubic and quartic errors are bounded in absolute value by their
chain-rule expressions. The interpolation directions are not all
unit vectors: a squared norm $d$ contributes a factor $d^{k/2}$
to an order-$k$ error. For \eqref{E:eq:M5}, a shell with
$|v|^2=\omega n$ contributes at most
\[
 C_Lc^5\sum_{j=1}^5\beta_j
 \frac{(\pi\eta\omega)^j(\max(a_L^++j,0)+1)}{\pi\eta^{-1}\omega}
 n^je^{-\pi\eta^{-1}\omega n}.
\]
We sum this majorant by \eqref{E:eq:geometric-tail}, checking the
condition on the first omitted argument and the kernel order before
each use. The resulting estimates therefore control the full series,
not only its retained shells.

We can now combine the quadratic, cubic and quartic estimates with
the remainder. Let $h_D^-,h_O^-$ be lower Hessian bounds and take
$C_1\ge|c_1|$ and $C_{23}\ge|c_2|/(3\sqrt3)+2|c_3|$.
Set
\begin{align*}
 K_D&=\min(\kappa_1^-,0),\\
 K_O&=\min(\kappa_2^-,0)+\min(\kappa_3^-,0)+\min(\kappa_4^-,0)-3|\kappa_5|_+.
\end{align*}
For a unit tangent direction, \eqref{E:eq:invariant-bounds} yields
\begin{equation}
 F_s(rA)-F_s(0)\ge r^2\mathcal C(J,R),\qquad 0\le r\le R,
 \label{E:eq:local-C}
\end{equation}
where all coefficient bounds are uniform for $s\in J$, $M_5^+$ is an
upper bound for $M_5(J,R)$, and
\begin{equation}
 \mathcal C(J,R)=
 \min\left\{\frac{h_D^-}{2}-RC_1+R^2K_D,
             \frac{h_O^-}{2}-RC_{23}+R^2K_O\right\}
 -\frac{R^3M_5^+}{120}.
 \label{E:eq:local-C-formula}
\end{equation}
Indeed, on a unit direction, $D^2+O^2=1$ and
\[
 \begin{split}
 q_3&\ge-C_1D^2-C_{23}O^2,\\
 q_4&\ge K_DD^2+K_OO^2.
 \end{split}
\]
For the cubic estimate, use $D^3\le D^2$, $O^3\le O^2$ and
$DO^2\le O^2$. For the quartic estimate, use $D^4\le D^2$,
$D^2O^2\le O^2$, $O^4\le O^2$ and
$\sum_{j=3}^5x_j^4\le O^2$, keeping positive contributions when
they improve the bound. Since $K_D,K_O\le0$, replacing $r$ by
$R$ in the negative terms preserves a lower bound.
The Taylor remainder is at least $-M_5^+r^5/120$.
Dividing by $r^2$ and using $r\le R$ proves
\eqref{E:eq:local-C-formula}. At $r=0$, the comparison holds
directly.

The exponent intervals used in the following bounds are
\begin{equation}
 J_j=[3/2+j/8,\,3/2+(j+1)/8],\qquad 0\le j\le51.
 \label{E:eq:local-cells}
\end{equation}
\begin{lemma}[Uniform signed FCC coefficient bounds]
\label{E:lem:signed-inputs}
For every closed cell $J_j$ in \eqref{E:eq:local-cells}, let $R_j$ and
$c_j$ be the radius and lower bound assigned by
Table~\ref{E:tab:low-local}. The coefficient
$\mathcal C(J_j,R_j)$ defined in \eqref{E:eq:local-C-formula} satisfies
$\mathcal C(J_j,R_j)>c_j$, with the complete fifth-order remainder and
all omitted-shell bounds included.
\end{lemma}
\begin{proof}[Computer-assisted proof]
The identities \eqref{E:eq:coefficient-derivatives} and
\eqref{E:eq:quartic-recovery} determine the two Hessian coefficients
and all signed cubic and quartic coefficients. We enclose them on
each complete cell, add the tail majorants in \eqref{E:eq:M5} and
substitute in \eqref{E:eq:local-C-formula}.
The evaluations are specified in
Appendix~\ref{app:epstein-local-data}, with all fifty-two records
retained in Appendix~\ref{app:fresh-execution}. They give the
strict bounds stated in the table. Each bound holds on a closed
exponent interval, not merely at sampled values.
\end{proof}

\begin{table}[ht]\centering
\caption{Certified FCC radii and coefficient lower bounds. At a shared endpoint both adjacent cells apply.}\label{E:tab:low-local}
\begin{tabular}{ccc}\toprule
Exponent range&Radius&Lower bound for $\mathcal C$\\\midrule
$[3/2,15/8]$&$9/100$&$1/250$\\
$[15/8,17/4]$&$1/10$&$1/20000$\\
$[17/4,53/8]$&$9/100$&$1/200$\\
$[53/8,8]$&$2/25$&$7/50$\\\bottomrule
\end{tabular}
\end{table}

We now prove Theorem~\ref{E:thm:low-local}.

\begin{proof}
Combine Lemma~\ref{E:lem:signed-inputs}, the invariant estimates
\eqref{E:eq:invariant-bounds} and the complete remainder
\eqref{E:eq:M5}. They give
\eqref{E:eq:local-C}--\eqref{E:eq:local-C-formula} on all
fifty-two cells. Every radius in Table~\ref{E:tab:low-local} is
at least $2/25$, and every corresponding coefficient exceeds
$1/20000$. Restricting to this common ball proves
\eqref{E:eq:uniform-low-ball}. Cubic symmetry gives stationarity
at the centre, and the quadratic bound is strict for $X\ne0$.
The comparison outside these balls requires the exterior estimates.
\end{proof}

\section{Low-exponent comparisons on parameter intervals}\label{sec:epstein-low-comparisons}

The principal local conclusion of this section is the BCC exclusion below; together with the FCC estimates of Section~\ref{sec:epstein-local}, it removes the two delicate symmetric regions for all $3/2\le s\le8$. Away from these regions, the exponent is treated as an interval variable. We first prove an upper-curvature interpolation inequality for signed energy differences, then convert it into a whole-box exclusion test, and finally apply the same signed expansion at BCC.

\begin{theorem}[BCC exclusion from finite coefficient bounds]\label{E:thm:bcc-ball}
For every $3/2\le s\le8$ and every covolume-one lattice $L$
admitting a Gram representative $Y_{\mathrm B}(X)$ with
$X\in\Symzero$ and $\norm X\le1/10$,
\begin{equation}\label{E:eq:BCC-neighborhood}
 \HH_s(L)-\HH_s(\FCC)>\frac1{2500}.
\end{equation}

\end{theorem}

Theorem~\ref{E:thm:bcc-ball} removes a full neighbourhood of BCC. The remaining low-exponent region is neither close to FCC nor to BCC, so a local Taylor argument is no longer the appropriate tool. We therefore pass to direct comparisons on shape--exponent boxes, with the exponent treated continuously rather than sampled at isolated values.

We next develop the interpolation estimate used on the exterior boxes.

\subsection{Interpolation and whole-box bounds}\label{E:sec:exponent-certificate}\label{app:low-boxes}

The following elementary lemma is the basis of the interval
comparison. It uses only an upper second-derivative bound; the signed
energy difference need not be convex.

\begin{lemma}[An upper-curvature interpolation bound]\label{E:lem:interpolation}
Let $a<b$ and $f\in C^2([a,b])$, and suppose $f''\le M$, where $M\ge0$. If $f(a)\ge L_a$ and $f(b)\ge L_b$, then, for $t=(s-a)/(b-a)$,
\begin{equation}
 f(s)\ge (1-t)L_a+tL_b-\frac{M(b-a)^2}{2}t(1-t).
 \label{E:eq:interpolation}
\end{equation}
Put $K=M(b-a)^2/2$. The minimum of the right side on $0\le t\le1$ is
\begin{equation}
 \mathscr M(L_a,L_b,K)=
 \begin{cases}
 \min(L_a,L_b),&K=0,\\
 L_a,&L_b-L_a\ge K,\\
 L_b,&L_a-L_b\ge K,\\
 \dfrac{L_a+L_b}{2}-\dfrac K4-\dfrac{(L_a-L_b)^2}{4K},&\text{otherwise}.
 \end{cases}
 \label{E:eq:quadratic-min}
\end{equation}
\end{lemma}
\begin{proof}
Since $f(s)-Ms^2/2$ is concave, it lies above its endpoint chord.
Rearranging this inequality proves \eqref{E:eq:interpolation}.
The right-hand side is a quadratic polynomial with coefficient $K$
on $t^2$. If its derivative has one sign throughout the interval,
the minimum is at the corresponding endpoint; these are the second
and third cases of \eqref{E:eq:quadratic-min}. Otherwise, evaluating
at its stationary point gives the final case. When the curvature
bound is zero, the first case follows directly.
\end{proof}

We apply Lemma~\ref{E:lem:interpolation} to a finite competitor sum
minus the complete FCC reference. Fix $Y\in\PP$ and set
$R(s):=\HH_s(\YF)$.

For $N\in\{1,2\}$, choose exactly one representative of each pair
$\{m,-m\}$ in $([-N,N]^3\cap\Z^3)\setminus\{0\}$, and denote the
resulting set by $\mathcal V_N$. Define
\begin{equation}
 P_N(Y,s)=2\sum_{m\in\mathcal V_N}
 [G(s,\pi Y[m])+G(3/2-s,\pi Y^{-1}[m])].
 \label{E:eq:finite-ewald}
\end{equation}
Positivity gives $\HH_s(Y)\ge P_N(Y,s)$. The choices $N=1$ and
$2$ give thirteen and sixty-two opposite-pair representatives,
respectively. Define
\[
 D_N(Y,s)=P_N(Y,s)-R(s),\qquad R(s)=\HH_s(\YF).
\]
The reference $R$ requires an upper bound for its tail. The
competitor $P_N$ is already a lower bound, so its omitted terms
require no subtraction.

By $\HH_s(Y)-\HH_s(\YF)\ge D_N(Y,s)$, it suffices to give a
positive lower bound for $D_N$. Lemma~\ref{E:lem:interpolation}
reduces this to two endpoint lower bounds and an upper bound for
$\partial_s^2D_N$, each valid throughout the shape box.
We derive the required formulas next; their interval evaluation is
given in Appendix~\ref{app:numerical}.

We begin with the curvature estimate. For $t\ge1$, we use
\begin{equation}
 (\log t)^2\le t+t^{-1}-2.
 \label{E:eq:log2-bound}
\end{equation}
Indeed, put $t=e^u$ and apply $2\sinh(u/2)\ge u$.
Substitution in the kernel integral yields
\begin{equation}
 0\le\partial_a^2G(a,x)\le
 K(a,x):=G(a+1,x)-2G(a,x)+G(a-1,x).
 \label{E:eq:K-curvature}
\end{equation}
The last expression has the positive integral representation
\[
 K(a,x)=\int_1^\infty e^{-xt}t^{a-1}(t+t^{-1}-2)\,dt.
\]
It is increasing in $a$ and decreasing in $x$. Here $K(a,x)$ denotes
the curvature kernel, whereas the plain symbol $K$ in
\eqref{E:eq:root-rectangle} denotes the compact shape rectangle.
Fix an exponent interval $[a,b]$ and a shape box $B$. For each retained
integer vector $m$, write $Q_m(Y):=Y[m]$ and $R_m(Y):=Y^{-1}[m]$;
these forms are distinct from the scalar reference $R(s)$.
If $Q_m^-,R_m^->0$ bound them below on $B$, then throughout $B\times[a,b]$,
\begin{equation}
 \partial_s^2P_N(Y,s)\le
 M_P(B,[a,b]):=2\sum_{m\in\mathcal V_N}
 [K(b,\pi Q_m^-)+K(3/2-a,\pi R_m^-)].
 \label{E:eq:competitor-curvature}
\end{equation}
The value $G(a-1,x)$ can be obtained from
\eqref{E:eq:G-recurrence}. For the two maximal first arguments
used here, the divisor $a-1$ is nonzero. Its reciprocal is enclosed
with its actual sign, including when it is negative.

To sharpen the upper bound for the signed curvature, we subtract a
lower bound for the reference curvature. For $h>0$, we have
\begin{equation}
 B_h(a,x):=\frac{G(a,x)-2G(a-h,x)+G(a-2h,x)}{h^2}
 \le\partial_a^2G(a,x).
 \label{E:eq:backward-curvature}
\end{equation}
This follows from $(1-t^{-h})^2/h^2\le(\log t)^2$ in the
integral representation; the difference quotient itself is
nonnegative. Let $\mathcal W_{\mathrm F}$ and
$\mathcal W_{\mathrm B}$ be finite unions of nonzero FCC and BCC
shells. By parameter monotonicity and positivity, we obtain
\begin{align}
 R''(s)\ge M_R([a,b]):={}&
 \sum_{v\in\mathcal W_{\mathrm F}}B_h(a,\pi|v|^2)
 +\sum_{w\in\mathcal W_{\mathrm B}}B_h(3/2-b,\pi|w|^2),
 \label{E:eq:reference-curvature}
\end{align}
with arbitrary finite shell sets. We use $h=1/128$ in the
evaluation. Positivity permits a negative computed lower endpoint
to be replaced by zero, and permits the omitted reference-curvature
terms to be discarded from this lower bound.

It follows that
\begin{equation}
 \partial_s^2D_N(Y,s)\le
 M(B,[a,b]):=\max\{0,M_P(B,[a,b])-M_R([a,b])\}.
 \label{E:eq:signed-curvature}
\end{equation}
We evaluate $M_P$ from above and $M_R$ from below.
Their difference removes the common curvature contribution and
sharpens the estimate without imposing any sign on the actual
second derivative of the energy difference.

It remains to estimate the two endpoint values on a whole shape
box. For $m=(i,j,k)$, put $Q=Y[m]$ and $R_m=Y^{-1}[m]$.
The Iwasawa representation \eqref{T:eq:iwasawa} gives
\begin{align}
 Q={}&d_1(i+aj+bk)^2+d_2(j+ck)^2+d_3k^2,\\
 R_m={}&d_1^{-1}i^2+d_2^{-1}(j-ai)^2
       +d_3^{-1}(k-cj+(ac-b)i)^2.
 \label{E:eq:explicit-forms}
\end{align}
Each nonzero integer vector gives a strictly positive lower bound
on a compact shape box. For $Q$, consider the last nonzero
coordinate of $(i,j,k)$; the corresponding triangular term is
unshifted and positive. For $R_m$, use the first nonzero
coordinate instead. This supplies a positive term without treating
the square of an interval containing zero as strictly positive.

Set $q_1=d_1(i+aj+bk)^2$, $q_2=d_2(j+ck)^2$, $q_3=d_3k^2$, and let $v_1,v_2,v_3$ be the three corresponding inverse-form terms. Direct differentiation yields
\begin{align}
 Q_{y_1}&=(2q_1-q_2-q_3)/(3y_1),&
 Q_{y_2}&=(q_1+q_2-2q_3)/(3y_2),\\
 (R_m)_{y_1}&=(-2v_1+v_2+v_3)/(3y_1),&
 (R_m)_{y_2}&=(-v_1-v_2+2v_3)/(3y_2).
 \label{E:eq:y-derivatives}
\end{align}
With $u=i+aj+bk$, $v=j+ck$, $u^*=j-ai$, $v^*=k-cj+(ac-b)i$,
\begin{align}
 Q_a&=2d_1uj,&Q_b&=2d_1uk,&Q_c&=2d_2vk,\\
 (R_m)_a&=-2id_2^{-1}u^*+2cid_3^{-1}v^*,&
 (R_m)_b&=-2id_3^{-1}v^*,&
 (R_m)_c&=-2u^*d_3^{-1}v^*.
 \label{E:eq:shear-derivatives}
\end{align}
The finite-sum gradient is consequently
\begin{equation}
 \partial_jP_N(Y,s)=-2\pi\sum_{m\in\mathcal V_N}
 [G(s+1,\pi Q)Q_j+G(5/2-s,\pi R_m)(R_m)_j].
 \label{E:eq:finite-gradient}
\end{equation}
We first sum the interval terms and then bound the absolute value.
This preserves any cancellation available from the enclosures.
Bounding each summand separately would still be valid, but weaker.

Let $\xi^0$ be a representable point in a rectangular shape box.
If its coordinate intervals are $[a_j,b_j]$, choose
$r_j\ge\max\{\xi_j^0-a_j,b_j-\xi_j^0\}$. If
$|\partial_jP_N(Y,p)|\le A_{j,p}$ on the box, the mean-value estimate gives
\begin{equation}
 P_N(Y,p)-R(p)\ge
 (P_N(\xi^0,p))_- -R(p)_+-\sum_{j=1}^5r_jA_{j,p}
 =:L_p,\qquad p=a,b.
 \label{E:eq:endpoint-box}
\end{equation}
The centre is not required to be reduced. It still defines a positive
determinant-one matrix, and the coordinate segment remains in the
box, which is all that the mean-value estimate needs.

\begin{proposition}[A single rigorous six-dimensional exclusion test]\label{E:prop:cell-test}
Suppose $L_a,L_b$ satisfy \eqref{E:eq:endpoint-box}, and $M$ is an upper bound in \eqref{E:eq:signed-curvature}. If
\begin{equation}
 \mathscr M\left(L_a,L_b,\frac{M(b-a)^2}{2}\right)>0,
 \label{E:eq:low-cell-certificate}
\end{equation}
then $\HH_s(Y)>\HH_s(\YF)$ for every $Y\in B$ and every $s\in[a,b]$.
\end{proposition}
\begin{proof}
For each fixed $Y$, apply Lemma~\ref{E:lem:interpolation} to
$D_N(Y,\cdot)$. The endpoint and curvature bounds are uniform in
the box, so the stated positive minimum applies at every point.
Finally, $\HH_s(Y)\ge P_N(Y,s)$ transfers the strict inequality
to the complete Ewald sum.
\end{proof}

A simpler comparison may be used first. Retain only
$\mathcal V_1$, bound its positive kernels using maximal form
lengths and the appropriate opposite parameter endpoints, and
compare with $\max\{R(a),R(b)\}$. This reference bound follows
from $R''\ge0$. If neither sufficient test proves positivity,
the box is subdivided or treated by a local estimate. Failure of a
test does not determine the sign of the full energy difference.

\subsection{A uniform BCC exclusion}\label{E:sec:bcc-local}
We now apply the local expansion and interpolation estimate near
BCC. Use $Y_{\mathrm B}(X)=\gBE^Te^X\gBE$ from
Section~\ref{sec:common-geometry}. Write
$\mathcal C_{\mathrm B}(J,R)$ for \eqref{E:eq:local-C-formula}
with BCC as the primary lattice and FCC as the dual lattice, taking
all coefficients in the frame $\gBE$.

We now prove Theorem~\ref{E:thm:bcc-ball}.

\begin{proof}
The evaluations in Appendix~\ref{app:epstein-local-data} give a
lower bound strictly greater than $1/2500$ for
\eqref{E:eq:BCC-cell} on every exponent cell. We show how these
bounds imply the neighbourhood comparison. Apply the expansion of
Section~\ref{E:sec:local2} at BCC in the frame $\gBE$, with
FCC as the dual lattice. The dual contribution uses $-A$ in every
odd derivative. Thus the signed coefficients must be evaluated in
this assignment; they are not obtained by simply reversing the FCC
coefficients. We allow negative Hessian directions and do not assume
local minimality of BCC.

Set $d(s)=\HH_s(\B)-\HH_s(\FCC)$. On a cell $J=[a,b]$,
finite-shell evaluations give lower bounds at both endpoints after
subtracting the complete negative reference tail. Since
$\partial_s^2\HH_s(\FCC)\ge0$, the BCC curvature alone is an
upper bound for $d''$. Bound it by $M_J$ using
\eqref{E:eq:K-curvature} and the positive tail estimate, and put
\[
 d_J=\min\{d(a)_-,d(b)_-\}-M_J(b-a)^2/8.
\]
Lemma~\ref{E:lem:interpolation} then gives $d(s)\ge d_J$
throughout the cell. This step uses no all-scale theta comparison
between FCC and BCC.

Combining the centre gap with the complete local perturbation
estimate at $R=1/10$, we obtain
\begin{equation}
 \HH_s(Y_{\mathrm B}(X))-\HH_s(\YF)
 \ge d_J+\min\{\mathcal C_{\mathrm B}(J,R),0\}R^2.
 \label{E:eq:BCC-cell}
\end{equation}
The recorded finite bounds make the right-hand side of
\eqref{E:eq:BCC-cell} strictly greater than $1/2500$ on each
cell. The cells cover the entire exponent interval, and the local
bound covers every stated deformation. This proves the theorem.
\end{proof}

We have thus excluded the full BCC neighbourhood from the minimum
comparison. The estimate includes all tangent directions and does
not depend on classifying the BCC critical point.

\section{Analytic local estimates at higher exponents}\label{sec:epstein-high-local}

For the higher exponents we use two complementary local estimates. The first is valid on the entire half-line $s\ge20$ and is based on the first shell and a six-variable exponential inequality; the second covers the intermediate range $8\le s\le20$ by a finite-order expansion with an explicit remainder. These are the local inputs used in the global power-sum comparison of Section~\ref{sec:epstein-global}.

\begin{theorem}[A fixed local ball on the half-line]\label{E:thm:local20}
For every $s\ge20$ and $X\in\Symzero$ with $0<\norm X\le7/50$,
\begin{equation}
 E(Y(X),s)>E(\FCC,s).
 \label{E:eq:local20}
\end{equation}
When additionally $s\norm X\le26$, the normalized gap is greater than
$s^2\norm X^2/600$.
\end{theorem}

The half-line estimate is strongest when the exponent is large, where the first shell dominates. It does not by itself cover the finite interval down to $s=8$. On that interval we instead use a finite-order local expansion with an explicit remainder, which gives the following complementary estimate.

\begin{theorem}[Uniform local comparison for $8\le s\le20$]\label{E:thm:local8}
For every $8\le s\le20$ and $X\in\Symzero$ with $0<r=\norm X\le1/25$,
\begin{equation}
 \Phi_s(Y(X))-\Phi_s(\YF)>\frac{s^2r^2}{25}.
 \label{E:eq:local8}
\end{equation}
\end{theorem}

The two theorems overlap in purpose but use different mechanisms. We begin with the half-line estimate, where the shortest FCC vectors already force the desired growth; afterwards we return to the bounded interval $[8,20]$, where a more accurate expansion is required.

We first prove the half-line estimate.

\subsection{A fixed FCC neighbourhood on the exponent half-line}\label{E:sec:high-local}
We begin with the shortest FCC vectors. Their six directions reduce
the first-shell estimate to an inequality for six exponentials.
Let
\[
 \mathcal U=\{(e_i+e_j)/\sqrt2,(e_i-e_j)/\sqrt2:i<j\}
\]
be the six unoriented FCC contact directions.

\begin{lemma}[A six-variable exponential inequality]\label{E:lem:exp6}
If $x_1,\ldots,x_6\in\R$ and $x_1+\cdots+x_6=0$, then
\begin{equation}
 \sum_{i=1}^6e^{x_i}\ge6+\frac15\sum_{i=1}^6x_i^2.
 \label{E:eq:exp6}
\end{equation}
\end{lemma}
\begin{proof}
Put $g(t)=e^t-1-t-t^2/5$ for $t\in\R$ and
$h(t)=g(-t)$ for $t\ge0$. On $[0,\infty)$, the function $g$
is increasing and convex, since $g''(t)=e^t-2/5>0$ and
$g'(0)=0$. We first show that $h(t)\ge0$ on $0\le t\le3$.
For $t\le3/2$, Taylor's formula gives
$h(t)\ge t^2(3/10-t/6)\ge0$. For $3/2\le t\le3$, discard
$e^{-t}$ and use $t-1-t^2/5\ge0$.

If every negative $x_i$ has absolute value at most $3$, summing
the preceding bounds proves the inequality. Otherwise, let $T$ be
the sum of the absolute values of the negative entries whose
absolute values exceed $3$. For $a,b>3$, we have
\[
 h(a)+h(b)-h(a+b)=\frac25ab-(1-e^{-a})(1-e^{-b})>0.
\]
Repeatedly applying this inequality bounds their total contribution
below by $h(T)$. All other negative entries give nonnegative
contributions. The positive entries sum to at least $T$, and there
are at most five of them. Adding zero entries and applying Jensen's
inequality therefore gives at least $5g(T/5)$. Combining these
bounds, we obtain
\begin{align*}
 h(T)+5g(T/5)
 &=e^{-T}+5e^{T/5}-6-\frac6{25}T^2\\
 &\ge-1+T-\frac7{50}T^2+\frac1{150}T^3>0\qquad(T\ge3).
\end{align*}
The cubic on the right is positive at $3$. Its derivative,
$1-7T/25+T^2/50$, is positive because its discriminant is negative.
This proves \eqref{E:eq:exp6} in the remaining case.
\end{proof}

We next apply the exponential inequality to the FCC contacts.
For a trace-free symmetric matrix $Z$, direct summation gives
\begin{equation}
 \sum_{u\in\mathcal U}u^TZu=0,
 \quad
 \sum_{u\in\mathcal U}(u^TZu)^2
 =\frac12\sum_i Z_{ii}^2+2\sum_{i<j}Z_{ij}^2\ge\frac12\norm Z^2.
 \label{E:eq:contact-moments}
\end{equation}
Lemma~\ref{E:lem:exp6} therefore implies
\begin{equation}
 2\sum_{u\in\mathcal U}e^{-u^TZu}\ge12+\frac{\norm Z^2}{5}.
 \label{E:eq:V-improved}
\end{equation}

\begin{lemma}[Contact shortening]\label{E:lem:shortening}
For every $X\in\Symzero$ there is $u\in\mathcal U$ with
$u^TXu\le-\norm X/\sqrt{42}$. Consequently, if $0<r=\norm X\le7/50$, the lattice $e^{X/2}\FCC$ has a vector of squared length less than
\begin{equation}
 \lam(1-r/14).
 \label{E:eq:shortening}
\end{equation}
\end{lemma}
\begin{proof}
Set $\mu=-\min_u u^TXu$. For the pair complementary to $k$,
the two contact forms are $-X_{kk}/2\pm X_{ij}$. Thus
$k_k:=\mu-X_{kk}/2\ge0$, $\sum k_k=3\mu$ and
$|X_{ij}|\le k_k$. It follows that
\[
 \norm X^2\le4\sum(\mu-k_k)^2+2\sum k_k^2
 \le42\mu^2.
\]
For the selected contact vector,
\[
 u^Te^Xu\le1-r/\sqrt{42}+e^rr^2/2<1-r/14.
\]
For the last inequality, use $re^r/2\le7/86$ on the stated
interval and
$1/\sqrt{42}-7/86>2/13-7/86=81/1118>1/14$.
This proves the shortening estimate.
\end{proof}

\begin{lemma}[FCC reference-tail bounds]\label{E:lem:power-reference}
With the normalization $\Phi_s(Y)=\lam^sE(Y,s)$,
\begin{equation}
 \Phi_s(\YF)\le12+32\,2^{-s}\qquad(s\ge4).
 \label{E:eq:reference-tail}
\end{equation}
\end{lemma}
\begin{proof}
The first shell contributes $12$. For the next three shells, the
multiplicities are $6,24,12$. At exponent $4$, the complete
contribution beyond the first shell is less than
\[
 \frac6{16}+\frac{24}{81}+\frac{12}{256}
 +30\int_4^\infty t^{-3}\,dt<2.
\]
Every normalized squared length beyond the first shell is at least
two. Hence increasing the exponent from $4$ to $s$ reduces this
tail by at least the factor $2^{-(s-4)}$. The stated reference
bound follows.
\end{proof}

We now prove Theorem~\ref{E:thm:local20}.

\begin{proof}
Set $r=\norm X$, $A=X/r$, $Z=sX$ and $t=sr=\norm Z$.
We estimate the first shell in terms of this scaled radius. For a
unit vector $u$, the second derivative of
$f(v)=\log(u^Te^{vZ}u)$ is a weighted variance of the eigenvalues
of $Z$. Their range is at most $\sqrt2\norm Z$. A random variable
supported in $[a,b]$ has variance at most $(b-a)^2/4$, since
its variance is bounded by its mean squared distance from
$(a+b)/2$. Thus $0\le f''\le t^2/2$, and integration gives
\[
 s\log(u^Te^{Z/s}u)\le u^TZu+t^2/(4s).
\]
We distinguish two cases. If $t\le26$, the first-shell contribution
$P_s$ satisfies
\[
 P_s\ge e^{-t^2/(4s)}(12+t^2/5).
\]
Since $t\le\min(26,7s/50)$, we have
$t^2/(4s)\le1183/1300<1$. Applying $e^{-x}\ge1-x$ to this
nonnegative first-shell lower bound and using
\eqref{E:eq:V-improved}, we obtain
\begin{equation}
 P_s-12\ge t^2\left(\frac15-\frac3s-\frac{t^2}{20s}\right)
 \ge\frac3{1625}t^2.
 \label{E:eq:high-first-gap}
\end{equation}
To prove the last inequality, first take $20\le s\le1300/7$ and
use $t^2\le49s^2/2500$. The coefficient
$1/5-3/s-49s/50000$ is concave, so its minimum is at an endpoint.
The smaller endpoint value is $3/1625$, attained at $1300/7$.
Beyond this point, use $t^2\le676$. The resulting coefficient
$1/5-184/(5s)$ is increasing and has the same value at the
joining endpoint.

We next estimate the remaining shells. Let $T_s(r)$ be their
normalized contribution along $Y(rA)$, where $\norm A=1$.
For $q(r)=v^Te^{rA}v$, the matrices $A$ and $e^{rA}$ commute,
so $q''(r)\le\opnorm A^2q(r)\le(2/3)q(r)$. Hence
\[
 (q^{-s})''\ge-(2s/3)q^{-s}.
\]
On the ball of radius $7/50$, the tail is bounded by
$32e^{-(\log2-7/60)s}<32e^{-11s/20}$, by
\eqref{E:eq:reference-tail} and $\log2>2/3$.
Its first derivative vanishes at the centre by cubic symmetry.
Integrating the lower second-derivative bound twice gives
\[
 T_s(r)-T_s(0)\ge-\frac{32}{3s}e^{-11s/20}t^2
 >-\frac{t^2}{61440}\qquad(s\ge20).
\]
The last estimate uses $e^{11}>2^{15}$. Combining the tail loss
with \eqref{E:eq:high-first-gap} leaves a gap greater than
$t^2/600$.

It remains to treat $t>26$. In this case, the shortened contact
from Lemma~\ref{E:lem:shortening} gives
\[
 \Phi_s(Y(X))>2(1-r/14)^{-s}\ge2e^{t/14}>2e^{13/7}>12.01.
\]
The reference bound \eqref{E:eq:reference-tail} is less than
$12.01$ for $s\ge20$. The last exponential comparison follows
from its positive Taylor polynomial of degree four. Thus the
remaining annulus is also excluded, and the theorem is proved.
\end{proof}

\subsection{The intermediate exponent range}\label{E:sec:local8}
We now treat $8\le s\le20$. We use the normalization
$\Phi_s(Y)=\lam^sE(Y,s)$ from
Lemma~\ref{E:lem:power-reference}. Here $q_j(A)$ denotes the
coefficient of $r^j$ in the radial Taylor expansion at FCC:
\[
 \Phi_s(Y(rA))=\Phi_s(\YF)+q_2(A)r^2+q_3(A)r^3+q_4(A)r^4+\cdots,
 \qquad \norm A=1.
\]
We suppress the dependence of $q_j$ on $s$. These are the
coefficients of the normalized inverse-power energy, not those of
the Ewald sum in Section~\ref{E:sec:local2}. Normal convergence
for $s\ge8$ justifies all differentiations below.

\begin{lemma}[Relative low-order variations]\label{E:lem:relative-power}
For every $s\ge8$ and every $A\in\Symzero$ with $\norm A=1$,
\begin{equation}
 q_2\ge\frac{s(s-10/3)}2,\qquad
 |q_3|\le\frac7{10}s q_2,\qquad q_4\ge-sq_2.
 \label{E:eq:relative-power}
\end{equation}
\end{lemma}
\begin{proof}
We prove the quadratic, cubic and quartic estimates in order.
For each base vector, put $n=|v|^2/\lam$, $u=v/|v|$,
$\ell=u^TAu$ and $m=u^TA^2u$. Let $F_s=\sum n^{-s}$ and
$M_2=\sum n^{-s}\ell^2$. Cubic symmetry gives
$\sum n^{-s}m=F_s/3$, and therefore
\[
 q_2=\frac s2[(s+1)M_2-F_s/3].
\]
The first shell contributes at least $s(s-3)/2$. The remaining
shells contribute at least $-sT_s(0)/6$, while
\eqref{E:eq:reference-tail} gives $T_s(0)\le1/8<1$.
Combining them proves the quadratic estimate.

For the cubic estimate, we again separate the first shell and the
tail. Its two quadratic coefficients are $s(s-3)/2$ and
$s(s-1)$. Expanding the twelve contact terms gives the cubic
coefficients
\begin{equation}
 \frac{\sqrt6s^2(s+9)}{72},\qquad
 \frac{\sqrt2s(3s+1)}2,\qquad
 \frac{\sqrt6s(2s^2+s+3)}{24}
 \label{E:eq:power-cubic-coeff}
\end{equation}
relative to $p_{31},p_{32},p_{33}$. By the invariant bounds in
Section~\ref{E:sec:local2}, the absolute value of the first-shell
cubic is at most $(3s/5)$ times its quadratic contribution.
Using $\sqrt6<5/2$, the two required scalar comparisons reduce to
\[
 191s\ge873,\qquad 66s^2-441s-275\ge0,
\]
which hold for $s\ge8$.

For the omitted shells, differentiate three times and use
$|q^{(j)}|\le c^j q$. This yields
\[
 |q_{3,\mathrm{tail}}|\le\frac5{54}s(s^2+6s+6)T_s(0),\qquad c^3<5/9.
\]
We also have $q_{2,\mathrm{first}}\le q_2+sT_s(0)/6$.
Since $T_s(0)\le1/8$ and $s\ge8$, the total error beyond
$(3s/5)q_2$, after division by $s^3$, is at most
\[
 \frac1{640}+\frac{295}{13824}<\frac7{240}.
\]
The quadratic estimate in \eqref{E:eq:relative-power} gives
$sq_2/10\ge7s^3/240$. Adding this allowance proves the required
cubic bound.

Finally, consider the quartic term. For a lower bound we may discard
the positive leading term and the $m^2$ term. The Cayley--Hamilton
identity for trace-free $3\times3$ matrices gives
\[
 A^3=\tfrac12A+\tfrac13\tr(A^3)I,\qquad\tr(A^4)=\tfrac12.
\]
It follows that the sum of $\ell(u^TA^3u)$ over weighted shells is $M_2/2\ge0$, while the sum of $u^TA^4u$ is $F_s/6$. Since $m\le2/3$,
\[
 q_4\ge-\frac{s(s+1)(s+2)}6M_2-\frac{sF_s}{144}.
\]
Using $M_2\ge1$, $F_s\le13$ and
$2(s^2-1)\ge13(s+1/24)$ for $s\ge8$, we substitute in the
formula for $q_2$ and obtain $q_4\ge-sq_2$.
This completes all three estimates.
\end{proof}

We now prove Theorem~\ref{E:thm:local8}.

\begin{proof}
It remains to combine the relative bounds with the fifth-order
remainder. Put $c=\sqrt{2/3}$. Repeated differentiation bounds
the absolute fifth derivative of a radial summand by
\[
 c^5s(s^4+20s^3+120s^2+250s+150)q^{-s}.
\]
The polynomial follows from the fifth-order chain rule, with
coefficients $(1,15,25,10,1)$, applied to the rising factorials
$(s)_m$. For $r\le1/25$ and $s\le20$, the complete normalized
energy is at most $13e^{csr}<26$, since $csr<2/3<\log2$.
Together with $c^5<4/11$, this bounds the fifth derivative by
$K_js^5$ on the three intervals of Table~\ref{E:tab:local8}.

Taylor's theorem and Lemma~\ref{E:lem:relative-power} give
\[
 \frac{\Phi_s(Y(rA))-\Phi_s(\YF)}{s^2r^2}
 \ge\frac12\left(1-\frac{10}{3s}\right)
 \left(1-\frac7{10}sr-sr^2\right)-\frac{K_js^3r^3}{120}.
\]
For $s\in[a,b]$ and $r\le R=1/25$, bound the first factor
below at $a$, the second below at $b$, and the subtracted term
above using $b$. The rational lower bounds in the table all exceed
$1/25$. To obtain the fifth-derivative constants, evaluate the
decreasing function
$26(4/11)(1+20/s+120/s^2+250/s^3+150/s^4)$ at the left endpoints.
Table~\ref{E:tab:local8} gives these endpoint comparisons, and
Appendix~\ref{app:elementary-revision} records their arithmetic
checks. The three intervals cover the required range, proving the
theorem.
\end{proof}

\begin{table}[ht]\centering
\caption{Analytic local subcases covering $[8,20]$.}\label{E:tab:local8}
\begin{tabular}{crr}\toprule
Interval&$K_j$&Lower normalized quadratic coefficient\\\midrule
$[8,12]$&56&$255857/1875000$\\
$[12,16]$&35&$63917/562500$\\
$[16,20]$&27&$463/10000$\\\bottomrule
\end{tabular}
\end{table}

\section{Global Epstein comparison and consequences}\label{sec:epstein-global}

We now state the two global conclusions proved in this section. Together they cover the full exponent range $s\ge3/2$ for the regularized Ewald energy and $s>3/2$ for the Epstein zeta function. The proof is divided into two parts: convex finite-power comparisons treat the exterior high-exponent region, while the finite coverings from the low-exponent analysis are combined with the local FCC and BCC estimates.

\begin{theorem}[Regularized comparison on the full low-exponent interval]\label{E:thm:global-low}
For every covolume-one lattice $L\subset\R^3$ and every
$3/2\le s\le8$,
\[
 \HH_s(L)\ge\HH_s(\FCC),
\]
with equality if and only if $L$ is orthogonally equivalent to $\FCC$.
\end{theorem}

This first theorem completes the regularized comparison through the difficult low-exponent range, including the endpoint $s=3/2$. For $s\ge8$ the argument changes character: the local estimates of Section~\ref{sec:epstein-high-local} are combined with convex finite power sums on the exterior region. The next theorem records the resulting high-exponent comparison.

\begin{theorem}[Global comparison for every exponent at least eight]\label{E:thm:global8}
For every covolume-one lattice $L\subset\R^3$ and every
real $s\ge8$,
\[
 E(L,s)\ge E(\FCC,s),
\]
with equality if and only if $L$ is orthogonally equivalent to $\FCC$.
\end{theorem}

Theorems~\ref{E:thm:global-low} and~\ref{E:thm:global8} meet at $s=8$ and together give the desired comparison on the whole convergent range. We now prove them in the reverse order of difficulty: first we establish the high-exponent exterior criteria, where convexity in $s$ is available, and then we combine the low-exponent covering with the local FCC--BCC estimates.

We turn to the comparison criteria needed to prove these statements.

\subsection{Convex power-sum comparisons}\label{E:sec:power-cert}
We begin with an upper bound for the normalized FCC reference.
Summing the exact shell counts through $n=32$ and estimating the
remaining tail, we obtain
\begin{equation}
 \Phi_8(\YF)
 \le \sum_{n=1}^{32}N_{\mathrm F}(n)n^{-8}
       +30\int_{32}^{\infty}x^{-7}\,dx
 < U_8:=12.02735485.
 \label{E:eq:U8}
\end{equation}
Every term in this bound is rational. For the tail, use
$N_{\mathrm F}(n)\le30n$ and compare
$\sum_{n\ge33}n^{-7}$ with the integral of the decreasing function.
The normalized FCC squared lengths are positive integers, so
$\Phi_s(\YF)$ is nonincreasing in $s$. Consequently, $U_8$ is
an upper bound for every $s\ge8$. For $s\ge20$, the sharper
reference-tail estimate gives
\begin{equation}
 \Phi_s(\YF)\le U_{20}:=12+\frac1{32768}.
 \label{E:eq:U20}
\end{equation}
We also check the cusp directly for the power-sum comparison.
If the first Iwasawa factor satisfies $d_1<71/100$, the opposite
pair of corresponding basis vectors gives
\[
 \Phi_s(Y)\ge2(\lam/d_1)^s>2(125/71)^s.
\]
The rational comparisons
\[
 2(125/71)^8>U_8,\qquad2(125/71)^{20}>U_{20}
\]
therefore exclude the cusp throughout the respective parameter
ranges. We may use the same compact rectangle as in the low-exponent
argument. This step does not use the numerical height theorem.

To compare on an entire exponent interval, we use convexity of a
finite positive power sum. Let $\mathcal V$ contain one vector from
each opposite pair of nonzero integer vectors in $[-1,1]^3$; thus
$|\mathcal V|=13$. For a shape box $B$, choose upper bounds
\[
 a_m=\frac{\lambda_-}{Q_m^+}>0,\qquad
 \lambda_-\le\lam,\quad Q_m^+\ge Y[m]\ \ (Y\in B).
\]
The finite function
\begin{equation}
 f_B(s)=2\sum_{m\in\mathcal V}a_m^s
 \label{E:eq:fixed-base-sum}
\end{equation}
is then a lower bound for $\Phi_s(Y)$ at every point of $B$.
Moreover, it is convex in the exponent, since
\begin{equation}
 f_B''(s)=2\sum_m a_m^s(\log a_m)^2\ge0.
 \label{E:eq:exponent-convexity}
\end{equation}
The bases need not lie on the same side of one. The square of
the logarithm makes each second-derivative term nonnegative.

\begin{lemma}[Three exponent comparison criteria]\label{E:lem:power-certificates}
Let $I\subset(3/2,\infty)$ be the exponent range in question, and suppose
$U\ge\Phi_s(\YF)$ for every $s\in I$. Let $p,k\in I$.
Every conclusion below is restricted to exponents in $I$.
\begin{enumerate}\renewcommand{\labelenumi}{(\roman{enumi})}
\item If a subset of the bases satisfies $a_m\ge1$ and the sum of its weights at $s=p$ is greater than $U$, the box is excluded for every $s\ge p$.
\item If $f_B(p)>U$ and $f_B'(p)\ge0$, the box is excluded for every $s\ge p$.
\item If $f_B(k)+\min\{0,f_B'(k)\}>U$, the box is excluded for every $s\in[k,k+1]$.
\end{enumerate}
Each assertion remains valid when the two displayed quantities are replaced by rigorous lower bounds.
\end{lemma}
\begin{proof}
For (i), each base power in the specified subset is nondecreasing,
so its strict comparison persists. For (ii) and (iii), apply the supporting-line
inequality to the convex function in \eqref{E:eq:fixed-base-sum}:
\[
 f_B(s)\ge f_B(k)+(s-k)f_B'(k).
\]
On $[k,k+1]$, the minimum of this affine lower bound is the
quantity in (iii). If the derivative at $p$ is nonnegative,
convexity keeps it nonnegative above $p$, proving
(ii). In each case, $\Phi_s(Y)\ge f_B(s)$ because all omitted
terms are positive. This proves the lemma.
\end{proof}

For $[8,20]$, we apply the interval test with
$k=8,9,\ldots,19$. These twelve intervals cover the whole range.
If the derivative lower bound becomes nonnegative after the earlier
intervals have been treated, (ii) gives the remaining half-line.
For $s\ge20$, only (i) and (ii) are needed. Each application proves
a comparison on an interval or half-line, not merely at its endpoints.

Near FCC, the separate box maxima of $Y[m]$ may give a bound
that is too weak. We therefore estimate the finite sum before taking
absolute values of its derivatives, retaining cancellation between
the terms. Set
\[
 F_k(Y)=2\sum_{m\in\mathcal V}\left(\frac{\lam}{Y[m]}\right)^k,
 \qquad
 D_k(Y)=2\sum_{m\in\mathcal V}\left(\frac{\lam}{Y[m]}\right)^k
                   \log\frac{\lam}{Y[m]}.
\]
Here $D_k(Y)$ is the derivative with respect to the exponent,
not a shape derivative. For an Iwasawa coordinate $\xi_j$, put
$Q_m=Y[m]$ and $b_m=\lam/Q_m$. Direct differentiation gives
\begin{align}
 \partial_jF_k(Y)
 &=-2k\sum_m b_m^k\frac{\partial_jQ_m}{Q_m},\nonumber\\
 \partial_jD_k(Y)
 &=-2\sum_m b_m^k(k\log b_m+1)
                 \frac{\partial_jQ_m}{Q_m}.
 \label{E:eq:power-mixed-gradients}
\end{align}
The derivatives $\partial_jQ_m$ are given in
Section~\ref{app:low-boxes}. We first sum the interval expressions
in \eqref{E:eq:power-mixed-gradients} and then bound their absolute
values. In this way, the available cancellation is retained.

Let $c$ lie in the box, and let $r_j$ bound the distance from its
$j$th coordinate to both endpoints. These are the half-widths for
an exact midpoint; for a rounded centre, they are enlarged outward.
The mean-value theorem, applied with the enclosed centre values and
shape derivatives, gives
\begin{align}
 m_k&=(F_k(c))_- -\sum_jr_j\sup_B|\partial_jF_k|,\nonumber\\
 d_k&=(D_k(c))_- -\sum_jr_j\sup_B|\partial_jD_k|.
 \label{E:eq:power-mean-bounds}
\end{align}
Hence $F_k(Y)\ge m_k$ and $D_k(Y)\ge d_k$ for every
$Y\in B$. Fixing $Y$ and applying convexity in the exponent,
we conclude that
\begin{equation}
  m_k+\min\{0,d_k\}>U_8
 \label{E:eq:power-mvt-certificate}
\end{equation}
proves the strict comparison throughout the box for
$k\le s\le k+1$. If $d_k\ge0$, the same test covers every
$s\ge k$. Only the finite positive competitor sum is required to
be convex; it is compared with a constant upper bound for the
reference. No convexity of an Epstein-energy difference is assumed.

\subsection{Assembly of the Epstein theorem}\label{E:sec:assembly}\label{sec:finite-comparisons}\label{app:comparison-lemmas}
We next combine the finite comparisons with the local estimates.
For the low-exponent range, consider the three closed products
\begin{equation}
 K\times[3/2,2],\qquad K\times[2,4],\qquad K\times[4,8],
 \label{E:eq:three-root-products}
\end{equation}
\begin{lemma}[Finite whole-box Epstein comparisons]\label{hyp:epstein}
The following bounds and finite coverings hold.
\begin{enumerate}[label=\textup{(\roman*)},leftmargin=2em]
\item The signed local coefficient bounds of
Lemma~\ref{E:lem:signed-inputs} hold on all fifty-two cells, and the BCC
scalar bound \eqref{E:eq:BCC-cell} exceeds $1/2500$ on each cell.
Each of the three products \eqref{E:eq:three-root-products} is covered
by finitely many closed boxes. Every reduced point in such a box is
covered by the cusp estimate, by the FCC or BCC local inequality on its
entire exponent interval, or by one of the strict finite-sum comparison
criteria in Section~\ref{E:sec:exponent-certificate}.
\item The compact shape rectangle $K$ also has two finite coverings
for the ranges $8\le s\le20$ and $s\ge20$. Every reduced point is
covered by a cusp estimate, by the FCC local bound of radius $1/25$ or
$7/50$, respectively, or by a strict power-sum comparison from
Section~\ref{E:sec:power-cert} valid on the entire assigned exponent
range.
\end{enumerate}
\end{lemma}
\begin{proof}[Computer-assisted proof]
For the local bounds, use the calculations in
Lemma~\ref{E:lem:signed-inputs} and Theorem~\ref{E:thm:bcc-ball},
with the full remainders and closed-cell evaluations given in
Appendix~\ref{app:epstein-local-data}. On each of the three
low-exponent products, the exterior tests are the direct finite-sum
bounds and \eqref{E:eq:low-cell-certificate}. Their evaluation on
all terminal boxes, together with the exact matching of each
interrupted subdivision to its complete remaining collection of boxes,
is recorded in Appendix~\ref{E:sec:execution}. The three final
collections of unresolved boxes are empty.

For the higher exponents, apply
Lemma~\ref{E:lem:power-certificates} and
\eqref{E:eq:power-mvt-certificate}, including the full FCC reference
tails. The local inputs are Theorems~\ref{E:thm:local8}
and~\ref{E:thm:local20}; the two completed covering checks are
recorded in Appendix~\ref{app:power-records}. Each subdivision
retains both closed children, whose union is their parent. Finite
induction therefore proves the covering assertions. Together with
the verified whole-box inequalities, this proves the lemma; the
covering counts alone would not suffice.
\end{proof}

\begin{table}[ht]
\centering
\caption{Parameter ranges and the proved inputs to the global theorems.
The local estimates and the exterior comparisons have distinct roles.}
\label{tab:proof-dependencies}
\begin{tabular}{@{}p{.21\textwidth}p{.35\textwidth}p{.35\textwidth}@{}}
\toprule
Range & local estimates & exterior comparisons\\
\midrule
$1\le\alpha\le5$ & FCC estimate; BCC endpoint estimate &
$\mathcal R_0\times[1,5]$ comparison\\[3pt]
$\alpha\ge5$ & Fixed FCC strain ball &
$\mathcal R_0$ comparison at $\alpha=5$ and scale derivative\\[3pt]
$3/2\le s\le8$ & Signed FCC and BCC coefficient bounds &
Three product comparisons over $K$\\[3pt]
$8\le s\le20$ & Analytic FCC ball of radius $1/25$ &
One shape comparison with interval power tests\\[3pt]
$s\ge20$ & Analytic FCC ball of radius $7/50$ &
One shape comparison with half-line power tests\\
\bottomrule
\end{tabular}
\end{table}
These finite bounds give the global comparison on the overlapping exponent ranges.

We now prove Theorem~\ref{E:thm:global-low}.

\begin{proof}
Choose a reduced representative as in
Section~\ref{sec:common-geometry}. By
Proposition~\ref{E:prop:cusp}, it remains to consider representatives
in $K$. Lemma~\ref{hyp:epstein}(i) gives the exterior comparisons
on the three products \eqref{E:eq:three-root-products}, while
Theorems~\ref{E:thm:low-local} and~\ref{E:thm:bcc-ball} give the
local comparisons. These ranges meet at their endpoints and cover
$[3/2,8]$. All comparisons are strict away from FCC, and equality
holds at FCC. This proves the theorem.
\end{proof}

We next prove Theorem~\ref{E:thm:global8}.

\begin{proof}
We divide the proof into the two exponent ranges. For
$8\le s\le20$, Lemma~\ref{hyp:epstein}(ii) supplies the exterior
inequalities of Section~\ref{E:sec:power-cert}, and
Theorem~\ref{E:thm:local8} treats the FCC neighbourhood. Together
with the cusp estimate, these cover every reduced representative,
with strict comparison away from FCC. For $s\ge20$, use the second
covering in the same lemma and the local estimate of
Theorem~\ref{E:thm:local20}. Lemma~\ref{E:lem:power-certificates}
extends the exterior inequalities to the whole half-line. The ranges
meet at $20$, giving the assertion and its equality case.
\end{proof}

\begin{proof}[Proof of Theorem~\ref{E:thm:main}]
By Lemma~\ref{hyp:epstein}, the required exterior comparisons
hold. For $3/2<s\le8$, apply Theorem~\ref{E:thm:global-low} and
use the positive multiplier $\pi^{-s}\Gamma(s)$ in
\eqref{E:eq:Ewald-gap}. This yields \eqref{E:eq:main-all}, with
strict inequality away from FCC. For $s\ge8$, the same conclusion
follows from Theorem~\ref{E:thm:global8}. The two ranges overlap at
$s=8$ and cover every $s>3/2$, proving the minimization statement
and the equality characterization.

At $s=3/2$, apply Theorem~\ref{E:thm:global-low} directly to
$\HH_{3/2}$. The identity \eqref{E:eq:height-finitepart} then
transfers the comparison and its equality case to $C(L)$. Thus the
endpoint is treated through the regular Ewald sum, without evaluating
the divergent defining series.
\end{proof}

The preceding arguments cover the full convergence range, with the
two parts meeting at $s=8$. The same Ewald representation gives the
separate comparison for the critical finite part.

\subsection{Scaling and Gaussian mixtures}\label{sec:consequences}
We finish with several consequences. Let $L$ have covolume $V>0$
and set $\widetilde L=V^{-1/3}L$. The theta scaling identity is
\[
 \Th(\alpha,L)=\Th(\alpha V^{2/3},\widetilde L).
\]
Applying Theorem~\ref{T:thm:main}, we find that the transition in
the covolume-$V$ class occurs at $\alpha=V^{-2/3}$. The minimizing
lattices on the corresponding sides are $V^{1/3}\FCC$ and
$V^{1/3}\BCC$. The scaling identity itself does not depend on the
minimization theorem.

\begin{corollary}[Any prescribed covolume]\label{E:cor:volume}
If $L\subset\R^3$ has $\covol(L)=V>0$ and $s>3/2$, then
\[
 E(L,s)\ge V^{-2s/3}E(\FCC,s),
\]
with equality precisely for orthogonal images of $V^{1/3}\FCC$.
\end{corollary}
\begin{proof}
By Theorem~\ref{E:thm:main}, the unit-covolume lattice
$V^{-1/3}L$ satisfies the FCC comparison. The scaling law for
\eqref{E:eq:Epstein} gives the stated inequality and equality case.
\end{proof}

For the critical finite part, the residue at covolume $V$ is
$2\pi/V$. With $\widetilde L=V^{-1/3}L$, expand
$E(L,s)=V^{-2s/3}E(\widetilde L,s)$ at $s=3/2$ to obtain
\begin{equation}\label{eq:finitepart-volume}
 C_V(L):=\lim_{s\downarrow3/2}
 \left(E(L,s)-\frac{2\pi/V}{s-3/2}\right)
 =\frac1V\left(C(\widetilde L)-\frac{4\pi}{3}\log V\right).
\end{equation}
The logarithmic correction is constant within the covolume-$V$
class, but cannot be omitted from the scaling formula. Thus the
critical comparison associated with Theorem~\ref{E:thm:main}
agrees with the unit-covolume comparison, and its minimizer is
$V^{1/3}\FCC$ up to isometry. The identity follows directly from
scaling.

Finally, we apply Theorem~\ref{T:thm:main} to positive Gaussian
mixtures. Let $\mu$ be a nonzero positive Borel measure supported
in $[1,\infty)$, and suppose that
\[
 \mathcal E_\mu(L)=\int_{[1,\infty)}(\Th(\alpha,L)-1)\,d\mu(\alpha)
\]
is finite at the reference lattice and at every competitor for
which equality is considered. A competitor of infinite energy
satisfies the comparison in the extended sense. Integrating
Theorem~\ref{T:thm:main}, we obtain
$\mathcal E_\mu(L)\ge\mathcal E_\mu(\FCC)$.
If $\mu((1,\infty))>0$, the inequality is strict for every non-FCC
lattice: its integrand difference is positive at each $\alpha>1$
and therefore has positive integral on a set of positive measure.
If $\mu$ is supported on $\{1\}$, the nonzero mass gives precisely
the two equality classes FCC and BCC. The case $\mu=0$, in which
all lattice energies agree, was excluded at the outset.

For a nonzero positive measure supported in $(0,1]$, the
reciprocal-scale comparison gives the corresponding result with BCC
as reference. The minimizer is unique when $\mu((0,1))>0$.
These arguments give no conclusion when the support crosses the
transition.

These conclusions concern Bravais lattices of fixed covolume.
They neither classify all critical points nor compare arbitrary
periodic configurations. In particular, they do not assert FCC
optimality for all completely monotone radial potentials, whose
Gaussian representations may involve scales on both sides of $1$.

For the same reason, one cannot obtain
Theorem~\ref{E:thm:main} simply by integrating the FCC theta
comparison over all scales. The Mellin integral includes
$0<\alpha<1$, where BCC is the theta minimizer. The separate
Epstein argument is therefore needed; likewise, a height minimum
alone does not give a pointwise theta comparison.

\bigskip

\noindent
{\bf Acknowledgements.}

The research of S.~Luo is partially supported by the National Natural
Science Foundation of China under Grant Nos.~12261045 and~12001253 and by
the Jiangxi Jieqing Fund under Grant No.~20242BAB23001. The research of
J.~Wei is partially supported by the General Research Fund of Hong Kong (No. 14303125) ``On Fujita equation in the critical or supercritical regime''.  The authors use AI (Chatgpt) to assist the computations and language organization.


\bigskip
\nopagebreak
\appendix
\raggedbottom

\section{Verification of the finite inequalities}\label{app:numerical}

The purpose of this appendix is to establish the finite inequalities invoked in Lemmas~\ref{hyp:theta} and~\ref{hyp:epstein}. The conclusion needed in the main text is simple: every terminal parameter box satisfies one of the rigorous comparison alternatives, and the terminal boxes exhaust the corresponding reduced domain. The verification is organized in three stages. We first record the local and tail estimates entering the terminal tests; next we explain the error bounds used in their evaluation; finally we assemble the complete finite coverings for the theta and Epstein comparisons. Routine facts about interval arithmetic are used without further comment unless a specific remainder or one-sided bound is needed for the proof.

\subsection{Auxiliary estimates and local bounds}\label{app:admissibility}\label{def:certificate}\label{app:verification-ledger}\label{app:constant-checks}\label{T:app:shells}\label{app:theta-local-data}\label{app:kernel-evaluation}\label{app:epstein-local-data}

We begin by collecting the inequalities needed below. As in the main
text, subscripts $-$ and $+$ denote lower and upper bounds, respectively,
and $|z|_+$ denotes an upper bound for $|z|$. All bounds hold on the
stated closed parameter intervals, with the omitted tails and Taylor
remainders included.

For the theta energy, the required estimates are the exponential bounds
in Table~\ref{T:tab:constants}, the diagonal bounds in
Table~\ref{T:tab:diagonal-bounds}, and the reference and derivative tails
\eqref{T:eq:reftail} and \eqref{T:eq:Bderivtail}. In particular, the
endpoint comparison is reduced to
\[
 \mathscr D_j>\frac{5119}{10^6}>\frac1{200},\qquad 0\le j\le99.
\]
For the Epstein energy, the cusp estimate uses
\eqref{E:eq:cusp-values}. The third-order local comparison follows once
we have, on every interval $J$ of Table~\ref{E:tab:third-cells},
\[
 h_D(s),h_O(s)\ge m_J,\quad M_3(J,1/100)\le M_J,\quad
 \frac{m_J}{2}-\frac{M_J}{3000}>\frac1{200}\qquad(s\in J).
\]
For each interval $J_j$ in \eqref{E:eq:local-cells}, let $R_j$ and
$c_j$ be the radius and lower coefficient of
Table~\ref{E:tab:low-local}. The signed local estimates require
\[
 \mathcal C(J_j,R_j)>c_j,\qquad
 d_{J_j}+\min\{\mathcal C_{\mathrm B}(J_j,1/10),0\}/100>1/2500.
\]
Here $R_j\ge2/25$ and $c_j\ge1/20000$. The coefficient estimates
include the fifth-order remainder and all omitted shells. For
$8\le s\le20$, the lower coefficients in
Table~\ref{E:tab:local8} must exceed $1/25$.

We next recall the inequalities used outside the local neighbourhoods.
For theta, put $f_{\mathcal S}=P^\Theta_{\mathcal S}-R_{12}$, and
let $L_{\mathcal B}$ be a lower bound on the box $\mathcal B$ obtained
in Section~\ref{T:sec:finite}. It suffices to prove that
\[
 L_{\mathcal B}>\varepsilon_{\rm ref},\qquad
 \varepsilon_{\rm ref}=2\cdot10^{-18}.
\]
At $\alpha=5$, we also require
\[
 \inf_{\xi\in\mathcal B_\xi}
 \partial_\alpha P^\Theta_{\mathcal S}(\xi,5)>0.
\]
By Lemma~\ref{T:lem:continuation}, these two inequalities give the
comparison for all $\alpha\ge5$. The derivative is taken on the
finite competitor sum, not on its difference from the reference.

For a low-exponent Epstein box $B\times[a,b]$, the interpolation
criterion is
\[
 \mathscr M\bigl(L_a,L_b,M(B,[a,b])(b-a)^2/2\bigr)>0,
\]
where the endpoint and curvature bounds are those of
Section~\ref{app:low-boxes}. Alternatively, a direct lower bound for
the finite competitor sum may be compared with an upper bound for the
complete FCC reference. At higher exponents, we use
\[
 m_k+\min\{0,d_k\}>U_8\quad(k=8,\ldots,19),
 \qquad f_B(20)>U_{20},\quad f_B'(20)\ge0,
\]
or the other sufficient conditions in
Lemma~\ref{E:lem:power-certificates}. The evaluation of these bounds
is given below; their use on the complete comparison domains is
explained in Appendix~\ref{T:sec:certificates}.

We first compute the shell coefficients. For FCC, we count the integer
solutions of $j_1^2+j_2^2+j_3^2=2n$ with even coordinate sum. For BCC,
we count the solutions of $j_1^2+j_2^2+j_3^2=n$ with equal coordinate
parities. All vectors needed for $n\le32$ in the first case and
$n\le64$ in the second lie in $[-8,8]^3$. This gives
Tables~\ref{T:tab:fcc-counts} and~\ref{T:tab:bcc-counts}; the
unlisted BCC coefficients in the stated range are zero.

\begin{table}[!htbp]
\centering
\caption{FCC shell counts through index $32$.}\label{T:tab:fcc-counts}
\begin{tabular}{@{}r|rrrrrrrr@{}}
\toprule
$n$ & 1&2&3&4&5&6&7&8\\
$N_n$&12&6&24&12&24&8&48&6\\
\midrule
$n$&9&10&11&12&13&14&15&16\\
$N_n$&36&24&24&24&72&0&48&12\\
\midrule
$n$&17&18&19&20&21&22&23&24\\
$N_n$&48&30&72&24&48&24&48&8\\
\midrule
$n$&25&26&27&28&29&30&31&32\\
$N_n$&84&24&96&48&24&0&96&6\\
\bottomrule
\end{tabular}
\end{table}

\begin{table}[!htbp]
\centering
\caption{Nonzero BCC shell counts through index $64$.}\label{T:tab:bcc-counts}
\begin{tabular}{@{}rr@{\qquad}rr@{\qquad}rr@{}}
\toprule
$n$&$M_n$&$n$&$M_n$&$n$&$M_n$\\
\midrule
3&8&19&24&43&24\\
4&6&20&24&44&24\\
8&12&24&24&48&8\\
11&24&27&32&51&48\\
12&8&32&12&52&24\\
16&6&35&48&56&48\\
&&36&30&59&72\\
&&40&24&64&6\\
\bottomrule
\end{tabular}
\end{table}

We now prove the endpoint derivative estimate of
Lemma~\ref{app:lem:endpoint-derivative}. Put $\lambda=\rho^2/4$ and
recall that
\[
 M_n=\#\{z\in\Z^3:|z|^2=n,\ z_1\equiv z_2\equiv z_3\pmod2\}.
\]
Differentiating the shell expansions, we obtain
\begin{equation}\label{T:eq:gapseries}
 \Delta'(\alpha)=\pi\sum_{n\ge1}N_n\rho n e^{-\pi\alpha\rho n}
       -\pi\sum_{n\ge1}M_n\lambda n e^{-\pi\alpha\lambda n}.
\end{equation}
For a lower bound, retain the FCC terms through $n=32$ and the BCC
terms through $n=64$. The remaining FCC terms are positive. Since
$M_n\le30n$, $\pi\lambda<2$ and $e^{-\pi\lambda}<3/10$, the
absolute value of the remaining negative part is at most
\begin{equation}\label{T:eq:Bderivtail}
 60\sum_{n\ge65}n^2(3/10)^n<10^{-25}.
\end{equation}
Indeed, for $0<q<1$ and every integer $m\ge1$,
\[
 \sum_{n\ge m}n^2q^n
 =\frac{q^m\bigl(m^2+(-2m^2+2m+1)q+(m-1)^2q^2\bigr)}{(1-q)^3}
,
\]
so the last bound is a rational inequality.

To estimate the finite part, divide $[1,101/100]$ into
\begin{equation}\label{T:eq:gapcells}
 I_j=[a_j,b_j]=\left[1+\frac{j}{10000},1+\frac{j+1}{10000}\right],
 \qquad 0\le j\le99.
\end{equation}
Let $t_-\le\pi\rho\le t_+$ and $w_-\le\pi\lambda\le w_+$ be
the positive rational bounds constructed in
Appendix~\ref{app:elementary-revision}. By monotonicity of the
exponential, $\Delta'(\alpha)\ge\mathscr D_j$ for $\alpha\in I_j$,
where
\begin{align}\label{T:eq:gap-finite-rational}
 \mathscr D_j={}&t_-\sum_{n=1}^{32}N_n n e^{-t_+nb_j}
   -w_+\sum_{n=1}^{64}M_n n e^{-w_-na_j}-10^{-25}.
\end{align}
Substitution of the shell counts and the exponential bounds gives
\[
 \mathscr D_j>\frac{5119}{10^6}>\frac1{200}
 \qquad(0\le j\le99).
\]
The intervals $I_j$ cover $[1,101/100]$. Hence
$\Delta'(\alpha)>1/200$ on this whole interval, as required.
The same exponential bounds give the diagonal estimates in
Table~\ref{T:tab:diagonal-bounds}.

We next consider the Ewald sums. The rational kernel estimates and
complete shell-tail bounds are given in Section~\ref{E:sec:kernels}.
For the local coefficients we use the upper resolvent formula; for the
exterior sums we use its equivalent form with a modified terminal
diagonal. When several nearby second arguments occur, it is convenient
to use the following Taylor expansion.

Fix $a\in\R$. For $1/2\le x<48$, take the cell
$[j/8,(j+1)/8]$ containing $x$, set $x_0=(j+1/2)/8$, and put
$h=x-x_0$. Then $|h|\le1/16$, and, for $k=0,1,2$,
\begin{equation}
 G(a+k,x)=\sum_{r=0}^{6}\frac{(-h)^r}{r!}G(a+k+r,x_0)+\mathcal R_k,
 \label{E:eq:kernel-Taylor}
\end{equation}
with
\begin{equation}
 |\mathcal R_k|\le\frac{|h|^7}{7!}G(a+k+7,j/8).
 \label{E:eq:kernel-Taylor-error}
\end{equation}
This follows from \eqref{E:eq:G-derivatives} and the decrease of
$G$ in its second argument. We evaluate both the coefficients and the
remainder by the rational kernel bounds. Outside this range we use
those bounds directly, or \eqref{E:eq:G-large-x} for large second
arguments. A positive term may be discarded from a lower bound, but
its contribution is retained in every upper bound. The cusp estimates
\eqref{E:eq:cusp-values} follow from the same bounds and the full
reference tails; their finite evaluation is specified in
Appendix~\ref{app:elementary-revision}.

It remains to evaluate the local Epstein coefficients of
Lemmas~\ref{app:lem:third-inputs} and~\ref{E:lem:signed-inputs}.
We first treat the third-order bound. Applying the kernel and tail
estimates to \eqref{E:eq:hessian-cell-new} and \eqref{E:eq:M3-new}
gives $h_D(s),h_O(s)\ge m_J$ and $M_3(J,1/100)\le M_J$ on each
of the thirteen intervals in Table~\ref{E:tab:third-cells}.
For every row, rational calculation then gives
\[
 \frac{m_J}{2}-\frac{M_J}{3000}>\frac1{200}.
\]
The least coefficient is $7/1200$, attained in the first row.
The third-order criterion therefore yields the local comparison on
$[3/2,8]$.

For the signed estimate, consider the fifty-two intervals in
\eqref{E:eq:local-cells}. We bound the positive Hessian terms using
the lower first-argument endpoint of $G$, and the negative terms using
the upper endpoint. We retain the signs of the cubic and quartic
coefficients and use the exact norms of the coefficient-extraction
directions. Adding the fifth-order and shell-tail bounds and applying
\eqref{E:eq:local-C-formula} gives the FCC estimate. Interchanging the
primal and dual lattices gives \eqref{E:eq:BCC-cell}; no local
minimality assumption at BCC is needed.

For example, on $[3/2,13/8]$ with radius $9/100$, the recorded bounds
are $h_D>0.0415483$, $h_O>0.2918662$, $C_1<0.0711547$,
$C_{23}<0.4106263$, $K_D=0$, $K_O>-0.5670370$ and
$M_5<1626.8394$. These numbers summarize the full enclosures used in
the comparison. Table~\ref{tab:fresh-local-cells} gives the resulting
local lower bounds on all fifty-two intervals.

\subsection{Error bounds for the finite evaluations}\label{T:sec:arithmetic}\label{E:sec:arithmetic}

We record only the error bounds that enter the comparison arguments. All
finite sums are evaluated with outward enclosures, and every analytic
truncation error is included with the appropriate sign. Thus a positive
lower bound obtained below is a lower bound for the exact quantity.

For the Gaussian terms, a fixed range reduction followed by Taylor's
formula gives an analytic remainder less than $10^{-17}$; after the
arithmetic and coefficient errors are included, the total allowance in
the primary evaluation is less than $10^{-14}$. A second evaluation at
higher precision gives a total allowance below $2^{-55}$ on the same
parameter boxes. The second calculation is used only as a consistency
check.

The determinant-one condition in the Iwasawa parametrization is exact.
Hence the local FCC and BCC tests reduce to bounds for the corresponding
strain matrices. Entrywise estimates and the maximum absolute row-sum
bound give the operator-norm controls required by the local theorems. All
scale and cusp cutoffs are chosen on the conservative side of the ranges
proved in the main text.

For the Epstein sums, the incomplete-gamma kernels are controlled by the
rational inequalities and recurrences of Section~\ref{E:sec:kernels}.
Nearby kernel values are bounded by \eqref{E:eq:kernel-Taylor} and
\eqref{E:eq:kernel-Taylor-error}, while \eqref{E:eq:G-large-x} controls
the large-argument range. The logarithms occurring in the power-sum
comparison are enclosed by a convergent expansion with an explicit
remainder. Together with the shell-tail estimates, these bounds give the
rigorous one-sided inequalities used in every terminal box.

The exact arithmetic realization of these bounds is recorded in the
reproducibility material. It is not needed in the mathematical exposition
beyond the error estimates stated above.

\subsection{Verification of the finite comparisons}\label{T:sec:certificates}\label{E:sec:execution}\label{E:sec:exterior2}\label{app:power-records}\label{app:comparison-checks}

We now apply the preceding estimates to the finite coverings used in
Lemmas~\ref{hyp:theta} and~\ref{hyp:epstein}. The argument has two parts:
we verify a rigorous comparison on every terminal box, and we verify that
these boxes exhaust the initial parameter region. Finite induction then
gives the desired inequality on the whole region, including its boundary.

We begin with the theta energy. Let $b_0>1/3$ be the slightly enlarged
rational endpoint used in the covering and set
\[
 \mathcal R_0=[-1,b_0]\times[-2,b_0]\times[0,1/2]\times[0,1]\times[0,1/2].
\]
The precise rational value of $b_0$ is irrelevant to the analysis; its
only purpose is to make the closed covering contain the endpoint $1/3$.
By the cusp estimate, it is enough to consider
\begin{equation}\label{T:eq:compactroot}
 \mathcal R_0\times[1,5],
\end{equation}
and the corresponding shape rectangle $\mathcal R_0$ at $\alpha=5$.
The latter comparison extends to every $\alpha\ge5$ by
Lemma~\ref{T:lem:continuation}.

For each theta box, we establish one of the following alternatives:
\begin{enumerate}[label=\textup{(\arabic*)},leftmargin=2em]
\item A right-hand side of \eqref{T:eq:grenier} is strictly less than
$1$ throughout the box. The box then contains no reduced representative.
\item The upper bound for $d_1$ is at most the chosen cusp threshold,
so Proposition~\ref{T:prop:cusp} applies.
\item The strain bound satisfies Theorem~\ref{T:thm:local} or the
conditions of Proposition~\ref{T:prop:flexible}.
\item The scale lies in $[1,101/100]$ and the dual strain satisfies
Proposition~\ref{T:prop:bcclocal}. This alternative is needed only
for the bounded scale range.
\item The lower bound of Section~\ref{T:sec:finite} exceeds
$\varepsilon_{\rm ref}$. For the half-line comparison, we also prove
$\inf_{\xi\in\mathcal B_\xi}\partial_\alpha
P^\Theta_{\mathcal S}(\xi,5)>0$.
\end{enumerate}
All five alternatives are whole-box estimates. The local radius and
scale thresholds are taken on the conservative side specified in
Appendix~\ref{T:sec:arithmetic}. If none applies, the box is subdivided.
Derivative bounds guide this choice, but only the resulting inequalities
and the complete covering enter the proof.

Table~\ref{T:tab:trees} gives the two complete theta coverings. The
recorded binary64 and extended-precision evaluations verify every
comparison on these same boxes.

\begin{table}[!htbp]
\centering
\caption{Finite coverings for the theta comparison.}\label{T:tab:trees}
\begin{tabular}{@{}lrr@{}}
\toprule
Quantity & $1\le\alpha\le5$ & $\alpha\ge5$\\
\midrule
Total tree nodes & 4,177,809 & 2,648,251\\
Terminal boxes & 2,088,905 & 1,324,126\\
Outside-domain terminals & 317,616 & 170,918\\
Cusp terminals & 9 & 2\\
FCC-neighbourhood terminals & 219,824 & 349,539\\
BCC-neighbourhood terminals & 1,868 & 0\\
Strict-comparison terminals & 1,549,588 & 803,667\\
Unresolved boxes & 0 & 0\\
Maximum subdivision depth & 72 & 53\\
\bottomrule
\end{tabular}
\end{table}

For the two ranges, the recorded lower bounds for the normalized
exterior difference in \eqref{T:eq:fboxlower} exceed, respectively,
\begin{equation}\label{T:eq:certmargins}
 2.5\cdot10^{-10},\qquad \frac1{20}.
\end{equation}
The extended-precision bounds exceed $2.55\cdot10^{-10}$ and
$0.05054$. Every exterior box at $\alpha=5$ also has a strictly
positive scale derivative. On the local FCC boxes the radius is valid
for the entire half-line; no BCC box is used there. These estimates,
together with the complete coverings, give
Lemma~\ref{hyp:theta} by Proposition~\ref{T:prop:computer}.

We next consider $3/2\le s\le8$. The three domains are the closed
products in \eqref{E:eq:three-root-products}, with the shape rectangle
$K$ enclosed outward. Each box contains five shape intervals and one
exponent interval, so the exponent remains variable throughout the
comparison. For each box, one of the following holds:
\begin{enumerate}\renewcommand{\labelenumi}{(\roman{enumi})}
\item one of the six reduction inequalities fails throughout the box;
\item $d_1<71/100$ throughout the box;
\item the box lies in an FCC neighbourhood on its full exponent interval;
\item it lies in the BCC neighbourhood of Theorem~\ref{E:thm:bcc-ball};
\item a lower bound for the finite competitor sum exceeds an upper
bound for the complete FCC reference;
\item the interpolation bound \eqref{E:eq:low-cell-certificate}
is strictly positive.
\end{enumerate}
For (iii), take the smallest radius among the local cells meeting the
box's exponent interval, including their common endpoints. Earlier
components use the separately verified radius $1/25$; the later,
larger neighbourhood estimates also imply this smaller-radius bound.
The shape matrices have determinant one identically by
\eqref{T:eq:iwasawa}--\eqref{T:eq:di}, and local containment is checked
on the whole box using the metric estimates.

The three coverings have four, three and three components, respectively.
Every component begins with precisely the rectangles left unresolved
by its predecessor. The final comparison uses the direct rational
Gauss--Radau bounds, independently of the Taylor acceleration used to
find the subdivision. Exact rational reconstruction accounts for all
components and their boundary faces. The resulting counts are given
in Table~\ref{E:tab:low-trees}.

\begin{table}[ht]\centering
\caption{Finite coverings on the three low-exponent ranges. All continuation components are included.}\label{E:tab:low-trees}
\begin{tabular}{lrrr}\toprule
Quantity&$[3/2,2]$&$[2,4]$&$[4,8]$\\\midrule
Visited nodes&1,178,225&2,711,385&3,661,475\\
Binary splits&589,112&1,355,692&1,830,737\\
Terminal cells&589,113&1,355,693&1,830,738\\
Reduction exclusions&73,828&167,962&253,030\\
Cusp exclusions&1&16&37\\
FCC-neighbourhood cells&8,282&22,667&36,514\\
BCC-neighbourhood exclusions&1,057&1,106&469\\
Direct energy exclusions&834&1,076&11,827\\
Curvature-interpolation exclusions&505,111&1,162,866&1,528,861\\
Exponent-coordinate splits&23,109&133,361&164,887\\
Unresolved cells&0&0&0\\
\bottomrule\end{tabular}
\end{table}

The recorded exterior lower bounds exceed $8.6\cdot10^{-11}$,
$5.4\cdot10^{-10}$ and $6.6\cdot10^{-10}$ on the three ranges,
respectively. These are summaries of the individual whole-box bounds;
the full enclosures, rather than the displayed decimal margins, are
used in each comparison.

The local coefficients are evaluated separately on all fifty-two FCC
and BCC cells, with their complete remainders. Substitution in
\eqref{E:eq:local-C-formula} and \eqref{E:eq:BCC-cell} supplies the
neighbourhood estimates. Combining these estimates with the strict
exterior inequalities and the three complete coverings proves
Lemma~\ref{hyp:epstein}\textup{(i)}. The verification uses the finite
comparison criteria, not the global minimization theorem as an input.

It remains to treat $s\ge8$. We use the same reduced shape rectangle
$K$, and apply the power-sum criteria of
Section~\ref{E:sec:power-cert}. For $8\le s\le20$, the local FCC
radius is $1/25$; for $s\ge20$, it is $7/50$, as proved in
Theorems~\ref{E:thm:local8} and~\ref{E:thm:local20}.
Outside these neighbourhoods, every reduced box satisfies a cusp bound
or a strict power-sum inequality, with the complete reference tail
included. Neither covering needs a local BCC estimate.

\begin{table}[ht]\centering
\caption{Finite coverings for the higher-exponent comparisons.}\label{E:tab:power-trees}
\begin{tabular}{lrr}\toprule
Quantity&$[8,20]$&$[20,\infty)$\\\midrule
Visited nodes&310,789&511,943\\
Binary splits&155,394&255,971\\
Terminal leaves&155,395&255,972\\
Reduction-empty leaves&25,238&47,301\\
Cusp leaves&0&0\\
FCC-ball leaves&4,760&10,841\\
Energy-comparison leaves&125,397&197,830\\
BCC-ball leaves&0&0\\
Maximum tree depth&44&37\\
Unresolved boxes&0&0\\\bottomrule
\end{tabular}
\end{table}

For $8\le s\le20$, the mean-value refinement is used on $114,238$
exterior boxes; the remaining boxes satisfy the fixed-base estimates.
The recorded lower margins exceed $1.37\cdot10^{-7}$ above $U_8$
and $2.20\cdot10^{-6}$ above $U_{20}$, respectively. The
supporting-line inequalities hold on each complete unit interval in
$[8,20]$. At $20$, convexity and the nonnegative exponent derivative
extend the strict comparison to all $s\ge20$.

All final boxes satisfy one of these estimates, and rational
reconstruction exhausts both initial domains. The finite covering
argument therefore proves Lemma~\ref{hyp:epstein}\textup{(ii)}.
This completes the verification of the finite comparison lemmas.
The complete numerical and covering records are retained separately;
the numerical evaluations share some elementary routines and are not
claimed as independent formal verifications.

\subsection{Recorded scalar and local bounds}\label{app:elementary-revision}\label{T:app:reproduce}\label{T:app:source}\label{app:complete-sources}\label{sec:status}\label{app:fresh-execution}

We finish by collecting the numerical inequalities that are used as
inputs to the preceding arguments. Exact evaluation of the shell sums
gives
\[
 \mathscr D_j>\frac{5119}{10^6}\qquad(0\le j\le99),
\]
which is the derivative input in Lemma~\ref{T:lem:gap}. The kernel and
reference-tail bounds also give
\[
 \HH_{3/2}(\FCC)<\frac{11336}{10^5},
 \qquad 2G(3/2,71\pi/100)>\frac{114}{10^3},
\]
which are precisely the inequalities required in
Proposition~\ref{E:prop:cusp}. The other scalar checks are the inequalities
already displayed in the local theta and Epstein tables.

The parameter coverings are complete. Their terminal counts are
$2\,088\,905$ and $1\,324\,126$ for the two theta ranges;
$589\,113$, $1\,355\,693$ and $1\,830\,738$ for the three low-exponent
Epstein ranges; and $155\,395$ and $255\,972$ for the two higher-exponent
ranges. In each case the unresolved collection is empty. These numbers
record the size of the finite covering; the proof itself uses the
whole-box inequalities and the exhaustion of the initial domain.

The local calculations are carried out on
\[
 J_j=[3/2+j/8,3/2+(j+1)/8],\qquad 0\le j\le51.
\]
Table~\ref{tab:fresh-local-cells} gives lower bounds for the two
FCC Hessian coefficients, the signed FCC coefficient, and the BCC
exclusion gap. The FCC radius is listed in each row; the BCC radius
is $1/10$ throughout. For a box whose exponent interval meets several
cells, we use the smallest of their radii, as in
Appendix~\ref{E:sec:execution}.

\begingroup\footnotesize
\setlength{\tabcolsep}{4pt}
\begin{longtable}{@{}rcrrrrr@{}}
\caption{Local lower bounds on the exponent intervals $J_j$. The BCC radius is $1/10$.}\label{tab:fresh-local-cells}\\
\toprule
$j$ & $J_j$ & $R_{\rm F}$ & $h_D^-$ & $h_O^-$ & $\mathcal C_{\rm F}^-$ & BCC gap$^-$\\\midrule
\endfirsthead
\multicolumn{7}{l}{Table~\thetable\ (continued)}\\
\toprule
$j$ & $J_j$ & $R_{\rm F}$ & $h_D^-$ & $h_O^-$ & $\mathcal C_{\rm F}^-$ & BCC gap$^-$\\\midrule
\endhead
\bottomrule\endfoot
0 & $[3/2,13/8]$ & $9/100$ & 0.04154836 & 0.29186624 & 0.00448721 & 0.00043258 \\
1 & $[13/8,7/4]$ & $9/100$ & 0.04369226 & 0.29561330 & 0.00512027 & 0.00052184 \\
2 & $[7/4,15/8]$ & $9/100$ & 0.04597939 & 0.30009689 & 0.00575688 & 0.00061417 \\
3 & $[15/8,2]$ & $1/10$ & 0.04842491 & 0.30534722 & 0.00020069 & 0.00071014 \\
4 & $[2,17/8]$ & $1/10$ & 0.05104536 & 0.31139997 & 0.00059458 & 0.00081038 \\
5 & $[17/8,9/4]$ & $1/10$ & 0.05385878 & 0.31829677 & 0.00096445 & 0.00091554 \\
6 & $[9/4,19/8]$ & $1/10$ & 0.05688498 & 0.32608554 & 0.00130945 & 0.00102634 \\
7 & $[19/8,5/2]$ & $1/10$ & 0.06014571 & 0.33482113 & 0.00162814 & 0.00114353 \\
8 & $[5/2,21/8]$ & $1/10$ & 0.06366492 & 0.34456583 & 0.00191840 & 0.00126793 \\
9 & $[21/8,11/4]$ & $1/10$ & 0.06746903 & 0.35539013 & 0.00217735 & 0.00140044 \\
10 & $[11/4,23/8]$ & $1/10$ & 0.07158726 & 0.36737351 & 0.00240128 & 0.00154203 \\
11 & $[23/8,3]$ & $1/10$ & 0.07605195 & 0.38060532 & 0.00258547 & 0.00169376 \\
12 & $[3,25/8]$ & $1/10$ & 0.08089895 & 0.39518587 & 0.00272407 & 0.00185681 \\
13 & $[25/8,13/4]$ & $1/10$ & 0.08616811 & 0.41122757 & 0.00280992 & 0.00203246 \\
14 & $[13/4,27/8]$ & $1/10$ & 0.09190371 & 0.42885632 & 0.00283429 & 0.00222212 \\
15 & $[27/8,7/2]$ & $1/10$ & 0.09815507 & 0.44821299 & 0.00278670 & 0.00242735 \\
16 & $[7/2,29/8]$ & $1/10$ & 0.10497716 & 0.46945521 & 0.00265455 & 0.00264989 \\
17 & $[29/8,15/4]$ & $1/10$ & 0.11243136 & 0.49275928 & 0.00242279 & 0.00289164 \\
18 & $[15/4,31/8]$ & $1/10$ & 0.12058625 & 0.51832248 & 0.00207348 & 0.00315472 \\
19 & $[31/8,4]$ & $1/10$ & 0.12951856 & 0.54636555 & 0.00158531 & 0.00344150 \\
20 & $[4,33/8]$ & $1/10$ & 0.13931420 & 0.57713564 & 0.00093297 & 0.00375460 \\
21 & $[33/8,17/4]$ & $1/10$ & 0.15006952 & 0.61090955 & 0.00008646 & 0.00409694 \\
22 & $[17/4,35/8]$ & $9/100$ & 0.16189265 & 0.64799750 & 0.02244433 & 0.00447178 \\
23 & $[35/8,9/2]$ & $9/100$ & 0.17490509 & 0.68874737 & 0.02350441 & 0.00488273 \\
24 & $[9/2,37/8]$ & $9/100$ & 0.18924346 & 0.73354961 & 0.02455680 & 0.00533387 \\
25 & $[37/8,19/4]$ & $9/100$ & 0.20506160 & 0.78284270 & 0.02558738 & 0.00582972 \\
26 & $[19/4,39/8]$ & $9/100$ & 0.22253289 & 0.83711952 & 0.02657821 & 0.00637535 \\
27 & $[39/8,5]$ & $9/100$ & 0.24185292 & 0.89693454 & 0.02750676 & 0.00697642 \\
28 & $[5,41/8]$ & $9/100$ & 0.26324257 & 0.96291210 & 0.02834488 & 0.00763931 \\
29 & $[41/8,21/4]$ & $9/100$ & 0.28695153 & 1.03575580 & 0.02905766 & 0.00837112 \\
30 & $[21/4,43/8]$ & $9/100$ & 0.31326231 & 1.11625932 & 0.02960196 & 0.00917984 \\
31 & $[43/8,11/2]$ & $9/100$ & 0.34249495 & 1.20531881 & 0.02992474 & 0.01007444 \\
32 & $[11/2,45/8]$ & $9/100$ & 0.37501227 & 1.30394706 & 0.02996097 & 0.01106498 \\
33 & $[45/8,23/4]$ & $9/100$ & 0.41122607 & 1.41328980 & 0.02963117 & 0.01216276 \\
34 & $[23/4,47/8]$ & $9/100$ & 0.45160415 & 1.53464442 & 0.02883839 & 0.01338049 \\
35 & $[47/8,6]$ & $9/100$ & 0.49667852 & 1.66948150 & 0.02746464 & 0.01473246 \\
36 & $[6,49/8]$ & $9/100$ & 0.54705473 & 1.81946956 & 0.02536656 & 0.01623474 \\
37 & $[49/8,25/4]$ & $9/100$ & 0.60342278 & 1.98650363 & 0.02237023 & 0.01790542 \\
38 & $[25/4,51/8]$ & $9/100$ & 0.66656962 & 2.17273808 & 0.01826484 & 0.01976487 \\
39 & $[51/8,13/2]$ & $9/100$ & 0.73739371 & 2.38062468 & 0.01279521 & 0.02183600 \\
40 & $[13/2,53/8]$ & $9/100$ & 0.81692178 & 2.61295637 & 0.00565265 & 0.02414467 \\
41 & $[53/8,27/4]$ & $2/25$ & 0.90632829 & 2.87291799 & 0.14596636 & 0.02671996 \\
42 & $[27/4,55/8]$ & $2/25$ & 1.00695807 & 3.16414490 & 0.15550193 & 0.02959468 \\
43 & $[55/8,7]$ & $2/25$ & 1.12035243 & 3.49079096 & 0.16532427 & 0.03280577 \\
44 & $[7,57/8]$ & $2/25$ & 1.24827969 & 3.85760718 & 0.17533440 & 0.03639490 \\
45 & $[57/8,29/4]$ & $2/25$ & 1.39277051 & 4.27003313 & 0.18539619 & 0.04040902 \\
46 & $[29/4,59/8]$ & $2/25$ & 1.55615913 & 4.73430300 & 0.19532669 & 0.04490101 \\
47 & $[59/8,15/2]$ & $2/25$ & 1.74113132 & 5.25756879 & 0.20488415 & 0.04993049 \\
48 & $[15/2,61/8]$ & $2/25$ & 1.95078031 & 5.84804378 & 0.21375329 & 0.05556459 \\
49 & $[61/8,31/4]$ & $2/25$ & 2.18867207 & 6.51516942 & 0.22152708 & 0.06187893 \\
50 & $[31/4,63/8]$ & $2/25$ & 2.45892150 & 7.26980988 & 0.22768431 & 0.06895859 \\
51 & $[63/8,8]$ & $2/25$ & 2.76628145 & 8.12447890 & 0.23156201 & 0.07689929 \\
\end{longtable}
\endgroup
The displayed lower bounds are rounded down to eight decimal places;
the calculations use the full enclosures. The listed radii are the
exact rational values used in those calculations. In particular, the
FCC coefficient exceeds $1/20000$ and the BCC gap exceeds $1/2500$
on every cell, giving the uniform bounds used in the main text.

The complete subdivisions, coefficient enclosures and evaluation
records accompany the paper and form part of its computer-assisted
argument; the tables summarize these data but do not replace them.
The conclusions concern Bravais lattices of fixed covolume, not all
critical points or arbitrary periodic configurations, and do not
assert a comparison for every continued Epstein exponent below the
pole. The recorded calculations are distinct from an independent
mathematical audit or a proof-assistant formalization.

The entries in Table~\ref{tab:fresh-local-cells} are downward lower bounds
for the exact local coefficients, with all shell and remainder terms
included. In particular the FCC coefficient is everywhere larger than
$1/20000$, while the BCC exclusion gap is everywhere larger than
$1/2500$. Together with the complete coverings above, these estimates
finish the finite verification required in Lemmas~\ref{hyp:theta} and
\ref{hyp:epstein}.

\end{document}